\documentclass[reqno]{amsart}
 
\usepackage{amssymb,amsfonts,amsmath, amsthm}
\usepackage{lipsum} 
\usepackage[all,arc]{xy}
\usepackage{enumerate}
\usepackage{mathrsfs}
\usepackage[usenames, dvipsnames]{xcolor}
\definecolor{cerulean}{RGB}{0,123,167}
\usepackage[colorlinks = true, linkcolor = cerulean, citecolor = Thistle]{hyperref}
\usepackage{tikz}
\usepackage{tikz-cd}
\usepackage[super]{nth}
\usetikzlibrary{arrows}
\usepackage{caption}
\usepackage{geometry}
\usepackage{multicol}
\usepackage{esint}
\usepackage{mathtools}
\usepackage{import}
\usepackage{xifthen}
\usepackage{pdfpages}
\usepackage{transparent}
\usepackage{upgreek}
\usepackage{pgfplots}
\pgfplotsset{compat=1.17} 
\usepackage{bbm}

\newcommand{\C}{\mathbb{C}}

\newcommand{\F}{\mathbb{F}}	
\newcommand{\T}{\mathbb{T}}
\newcommand{\R}{\mathbb{R}}

\newcommand{\N}{\mathbb{N}}
\newcommand{\Z}{\mathbb{Z}}

\newcommand{\p}{\partial}

\newcommand{\abs}[1]{\vert#1\vert}

\providecommand{\norm}[1]{\left\Vert#1\right\Vert}

\newcommand{\triplenorm}[1]{{\left\vert\kern-0.25ex\left\vert\kern-0.25ex\left\vert #1 
    \right\vert\kern-0.25ex\right\vert\kern-0.25ex\right\vert}}

\newcommand{\Hzero}{{}_0H^1}
\newcommand{\Hzeros}[1]{{}_0H^{#1}}

\numberwithin{equation}{section}

\DeclareMathOperator{\supp}{supp}

\DeclareMathOperator{\sgn}{sgn}

\DeclareMathOperator{\diverge}{div}

\DeclareMathOperator{\sym}{sym}

\newtheorem{thm}{Theorem}[section]
\newtheorem{cor}[thm]{Corollary}
\newtheorem{prop}[thm]{Proposition}
\newtheorem{lem}[thm]{Lemma}

\theoremstyle{definition}
\newtheorem*{thm*}{Theorem}
\newtheorem*{def*}{Definition}
\newtheorem*{prop*}{Proposition}
\newtheorem{defn}[thm]{Definition}

\theoremstyle{remark}
\newtheorem{rem}[thm]{Remark}

\newcommand{\m}{\mathrm}

\newcommand{\lv}{\lVert}
\newcommand{\rv}{\rVert}

\newcommand{\al}{\alpha}

\newcommand{\es}{\varnothing}

\newcommand{\f}{\frac}
\newcommand{\sig}{\sigma}

\newcommand{\lf}{\lfloor}
\newcommand{\rf}{\rfloor}

\newcommand{\bpm}{\begin{pmatrix}}
\newcommand{\epm}{\end{pmatrix}}

\newcommand{\loc}{\m{loc}}

\renewcommand{\le}{\leqslant}
\renewcommand{\ge}{\geqslant}

\newcommand{\tnorm}[1]{\lv#1\rv}
\newcommand{\bp}[1]{\Big(#1\Big)}
\renewcommand{\sp}[1]{\big(#1\big)}
\newcommand{\tp}[1]{(#1)}
\newcommand{\babs}[1]{\Big|#1\Big|}

\newcommand{\tcb}[1]{\{{#1}\}}
\newcommand{\br}[1]{\left\langle #1 \right\rangle}

\providecommand{\tbr}[1]{\langle #1 \rangle}

\title[Traveling wave solutions to free boundary incompressible Navier-Stokes equations]{Traveling wave solutions to the inclined or periodic free boundary incompressible Navier-Stokes equations}

\author{Junichi Koganemaru}
\address{
Department of Mathematical Sciences\\
Carnegie Mellon University\\
Pittsburgh, PA 15213, USA
}
\email[J. Koganemaru]{jkoganem@andrew.cmu.edu}

\author{Ian Tice}
\address{
Department of Mathematical Sciences\\
Carnegie Mellon University\\
Pittsburgh, PA 15213, USA
}
\email[I. Tice]{iantice@andrew.cmu.edu}
\thanks{I. Tice was supported by an NSF CAREER Grant (DMS \#1653161). }



\begin{document}

\begin{abstract}

This paper concerns the construction of traveling wave solutions to the free boundary incompressible Navier-Stokes system.  We study a single layer of viscous fluid in a strip-like domain that is bounded below by a flat rigid surface and above by a moving surface.  The fluid is acted upon by a bulk force and a surface stress that are stationary in a coordinate system moving parallel to the fluid bottom.  We also assume that the fluid is subject to a uniform gravitational force that can be resolved into a sum of a vertical component and a component lying in the direction of the traveling wave velocity. This configuration arises, for instance, in the modeling of fluid flow down an inclined plane.  We also study the effect of periodicity by allowing the fluid cross section to be periodic in various directions.  The horizontal component of the gravitational field gives rise to stationary solutions that are pure shear flows, and we construct our solutions as perturbations of these by means of an implicit function argument.  An essential component of our analysis is the development of some new functional analytic properties of a scale of anisotropic Sobolev spaces, including that these spaces are an algebra in the supercritical regime, which may be of independent interest.
\end{abstract}

\maketitle

\section{Introduction}

The existence of traveling wave solutions to the equations of fluid mechanics has been a subject of intense study for nearly two centuries (see Section \ref{sec:previous work} for a brief summary).  Until recently, most of the mathematical results in this area focused on inviscid fluids, but work in the last few years constructed traveling wave solutions to the free boundary Navier-Stokes equations with a single horizontally infinite but finite depth layer of incompressible fluid \cite{leonitice}, and with multiple layers \cite{noahtice} in a uniform, downward-pointing gravitational field.  The purpose of the present paper is to extend these constructions into more general physical configurations by considering two effects: inclination of the fluid domain, which results in a component of the gravitational field parallel to the fluid layer; and, periodicity of the fluid layer in certain directions.  The key to the constructions in \cite{leonitice,noahtice} was the identification and utilization of a new scale of anisotropic Sobolev spaces, which serve as the container space for the function describing the free surface of the fluid.  The results in this paper rely crucially on some new functional analytic properties of these spaces, which we prove here: the development of this scale of spaces on domains with periodicity; and, the fact that these spaces form an algebra in the supercritical regime.  While these results are essential for our specific PDE needs, they may be of independent interest to those interested in Sobolev spaces.

\subsection{Kinematic and dynamic description of the fluid}

In this paper we consider a single finite depth layer of viscous, incompressible fluid.  We assume that the fluid evolves in a strip-like domain, bounded below by a flat, rigid, inclined surface and above by a free moving surface that can be described by the graph of a continuous function, in dimensions $n \ge 2$ (though, of course, only the cases $n \in \{2,3\}$ are physically relevant).  Furthermore, we assume that the fluid is subject to a uniform gravitational field $\mathfrak{G} \in \R^n$.

Prior to specifying the fluid domain or the equations of motion, we fix an orthogonal coordinate system as follows.   We choose the unit vector $e_n$ to be normal to the inclined surface, and we choose the unit vector $e_1$ in such a way that   $\mathfrak{G}$ lies in the $e_1$-$e_n$ hyperplane.  In other words, we posit that the uniform gravitational field $\mathfrak{G}$ resolves into two components as $\mathfrak{G} = \kappa e_1 - \mathfrak{g} e_n $, where $\kappa \in \R$ and  $\mathfrak{g} \in (0,\infty)$.   We can then define the angle of incline of the domain, $\theta \in (-\pi/2,\pi/2)$,  via $\tan \theta = \kappa/\mathfrak{g}$ (see Figure \ref{fig:fluid_domain}). 

Next we turn to a description of the fluid cross-section, $\Sigma$,  which will allow us to specify the free surface and rigid bottom of the fluid and to model periodicity.  The flat rigid bottom of the fluid, which is orthogonal to the $e_n$ direction, is described by $(n-1)$-dimensional sets of the form 
\begin{align}\label{sigma}
	\Sigma = 
	\begin{cases}
		\R \times \prod_{i=2}^{n-1} \Sigma_i, &\text{where }  \Sigma_i = \R  \text{ or else }\Sigma_i= L_i \T   \\
		\prod_{i=1}^{n-1} L_i \T,
	\end{cases}
\end{align}
where $L \T = \R / L \Z$ denotes the 1-torus of periodicity length $L>0$.    This choice of $\Sigma$ allows us to model: full periodicity in all directions, which corresponds to the second case; no periodicity, which corresponds to the first case with each $\Sigma_i = \R$; or partial periodicity when $n \ge 3$, which corresponds to the first case with at least one $\Sigma_i = L_i \T$.  However, for technical reasons that we will detail below, in the case of partial periodicity we cannot allow periodicity in the $e_1$ direction, and so the first factor of $\Sigma$ must be $\R$.  To be more explicit in the physically relevant cases $n \in \{2,3\}$, we note that \eqref{sigma} allows for $\Sigma \in \{\R , L \mathbb{T}\}$ when $n =2$, and when $n =3$ it allows for  $\Sigma \in \{ \R^2, \R \times L_2 \T, L_1 \T \times L_2 \T\}$, but we exclude the possibility that $\Sigma = L_1 \T \times \R$.  In any dimension, we endow $\Sigma$ with the usual topological and smooth structure.

With $\Sigma$ in hand, we can now describe strip-like $n$-dimensional domains with cross section $\Sigma$.  Given any function $\upzeta: \Sigma \to (0,\infty)$, we define the set $\Omega_\upzeta \subset \Sigma \times \R$ via $\Omega_\upzeta = \{ x = (x',x_n) \in \Sigma \times \R \mid 0 < x_n < \upzeta(x') \}$
and the $\upzeta$-graph surface $\Sigma_\upzeta \subset \Sigma \times \R$ via $\Sigma_\upzeta = \{ x = (x',x_n) \in \Sigma \times \R \mid x_n = \upzeta(x')  \text{ for some }  x' \in \Sigma\}$.  We note that if $\upzeta$ is continuous, then the upper boundary of $\Omega_\upzeta$ is $\Sigma_\upzeta$ and its lower boundary is given by $\Sigma_0 = \{ x = (x', x_n) \in \Sigma \times \R \mid x_n = 0 \}.$

We assume that at equilibrium, with all external forces and stresses absent, the fluid occupies the flat equilibrium domain $
	\Omega_b = \{x = (x',x_n) \in \Sigma  \times \R \mid 0 < x_n < b \}$
for some equilibrium depth parameter $b > 0$.  Furthermore, we assume that under perturbation, the fluid occupies the time-dependent fluid domain $\Omega_{b +\zeta(\cdot,t)}$, where $\zeta : \Sigma \times [0,\infty) \to (-b, \infty)$ is an unknown free surface function.  The fluid is then bounded above by the $(b+\zeta)$-graph surface $\Sigma_{b+\zeta(\cdot,t)}$ and below by the flat boundary $\Sigma_0$.

\begin{figure}
	\includegraphics[scale=0.6]{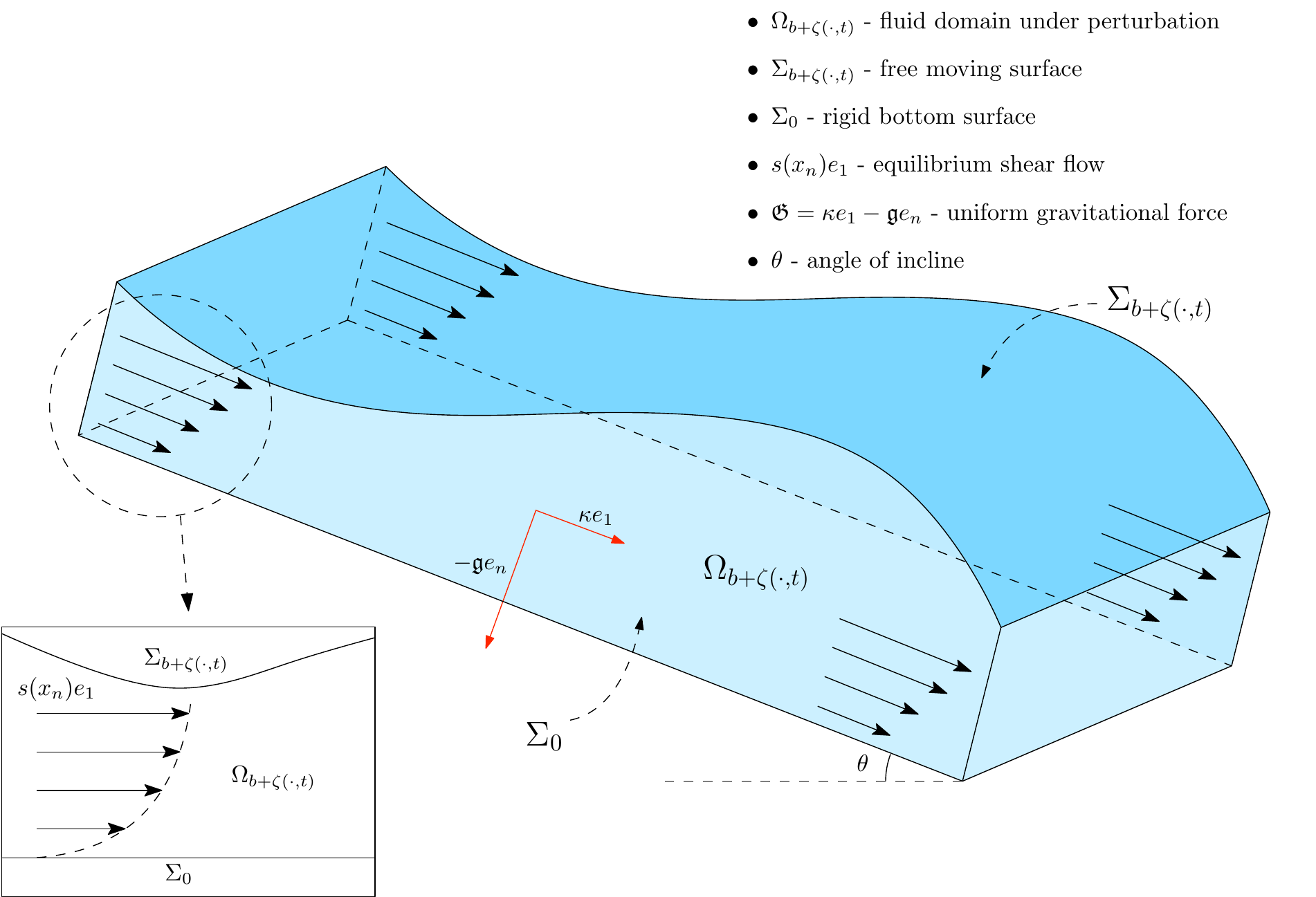}
	\caption{A sample portion of the time-dependent fluid domain under perturbation}
	\label{fig:fluid_domain}
\end{figure}

In addition to the aforementioned gravitational force, we posit that there are four other distinct forces acting upon the fluid: one in the bulk and three on the free surface. The first bulk force is a generic force described by the vector field $\tilde{\mathfrak{f}}(\cdot,t) : \Omega_{b +\zeta(\cdot,t)} \to \R^n$. The first surface force is a constant external pressure $P_{ext}\in \R$ applied by the fluid above the free surface. The second surface force is generated by an externally applied stress tensor, which is described by a map $\tilde{\mathcal{T}} : \Sigma_{b+\zeta(\cdot,t)} \to  \R^{n\times n}_{\sym}$, where $
		 \R^{n\times n}_{\sym} = \{ M \in \R^{n \times n} \mid M = M^T \}$ 
	 is the set of symmetric $n \times n$ matrices. The symmetry condition is imposed to be consistent with the fact that stress tensors are typically symmetric in continuum mechanics, but this condition is not essential and could be dropped in our analysis.  The third surface force is the surface tension generated by the surface itself, which we model in the standard way as $- \sigma \mathcal{H}(\zeta)$, where $\sigma \ge 0$ is the coefficient of the surface tension, and 
	 \begin{align} \label{eq:H}
		 \mathcal{H}(\zeta) = \diverge' \left( \frac{\nabla' \zeta }{\sqrt{1+ \abs{\nabla' \zeta}^2 }} \right) 
	 \end{align}
	 is the mean-curvature operator.

We assume that the evolution of the fluid  is described for time $t \ge 0$ by its velocity field $w(\cdot, t) : \Omega_{b +\zeta(\cdot,t)} \to \R^n$ and its pressure $P(\cdot, t): \Omega_{b +\zeta(\cdot,t)} \to \R$. For each $t > 0$, we require that the fluid velocity $w$, the pressure $P$, and the free surface $\zeta$ satisfy the free boundary incompressible Navier Stokes equations:
\begin{align} \label{eq:1}
	\begin{cases}
		\rho(\p_t w + w \cdot \nabla w) - \mu \Delta w + \nabla P =- \rho \mathfrak{g} e_n + \rho \kappa e_1  + \tilde{\mathfrak{f}}, & \text{in} \; \Omega_{b + \zeta(\cdot, t)} \\
		\diverge w = 0, & \text{in} \; \Omega_{b + \zeta(\cdot, t)} \\
		\p_t \zeta = w \cdot \nu \sqrt{1 + \abs{\nabla' \zeta}^2 },  & \text{on} \; \Sigma_{b + \zeta(\cdot, t)} \\	
		(PI - \mu \mathbb{D}w) \nu = [ - \sigma \mathcal{H}(\zeta) I + P_{ext} I + \tilde{\mathcal{T}} ] \nu,  & \text{on} \; \Sigma_{b + \zeta(\cdot, t)}	\\
	w= 0, & \text{on} \; \Sigma_0.
\end{cases}
	\end{align}
Here $\rho > 0$ is the fluid density, $\mu > 0$ is the fluid viscosity, $\mathbb{D}w = (\nabla w) + (\nabla w)^T \in \R^{n\times n}_{sym}$ is the symmetrized gradient of $w$, and $ \nu(\cdot, t) = (-\nabla'\zeta(\cdot,t), 1)/\sqrt{1+\abs{\nabla'\zeta(\cdot,t)}^2} \in \R^n$ 
is the outward pointing unit normal to the surface $\Sigma_{b+\zeta(\cdot, t)}$. The first two equations in \eqref{eq:1} are the standard incompressible Navier-Stokes equations; the first equation asserts the Newtonian balance of forces, the second enforces the conservation of mass. The third equation is the kinematic boundary condition describing the evolution of the free surface with the fluid. We note in particular that the third equation may be written as a transport equation in the form of $ \p_t \zeta + \nabla' \zeta \cdot w'\rvert_{\Sigma_{b + \zeta(\cdot, t)}} = w_n \rvert_{\Sigma_{b + \zeta(\cdot, t)}},$ which shows that the free surface $\zeta$ is transported by the horizontal component of the velocity $w'$ and driven by the vertical component of the velocity $w_n$. The fourth equation encodes the dynamic boundary conditions asserting the balance of forces on the free surface. The fifth equation is the typical no-slip condition enforced on flat rigid surfaces. 

For the sake of convenience, we shall assume without loss of generality that $\rho = \mu = \mathfrak{g} = 1$. This can be achieved by dividing both sides of the first equation in \eqref{eq:1} by $\rho$, rescaling the space-time variables, and renaming $b,\sigma, \kappa$, $\tilde{\mathfrak{f}}$, $P_{ext}$, $\tilde{\mathcal{T}}$, and (in the periodic settings) the periodicity length scales $L_i$, accordingly. This yields the system
\begin{align} \label{eq:2}
	\begin{cases}
		\p_t w + w \cdot \nabla w - \Delta w + \nabla P=-  e_n +\kappa e_1  + \tilde{\mathfrak{f}}, & \text{in} \; \Omega_{b + \zeta(\cdot, t)} \\
		\diverge w = 0, & \text{in} \; \Omega_{b + \zeta(\cdot, t)} \\
		\p_t \zeta = w \cdot \nu \sqrt{1 + \abs{\nabla' \zeta}^2 },  & \text{on} \; \Sigma_{b + \zeta(\cdot, t)} \\	
		(PI - \mathbb{D}w) \nu = [- \sigma \mathcal{H}(\zeta)I+ P_{ext} I + \tilde{\mathcal{T}} ]\nu,  & \text{on} \; \Sigma_{b + \zeta(\cdot, t)}	\\
	 w= 0, & \text{on} \; \Sigma_0.
\end{cases}
	\end{align}
For a differentiable vector field $u$ and a scalar $p$, we define the stress tensor $S(p,u) = pI - \mathbb{D}u \in \R^{n \times n}_{\sym}$, where $I$ is the $n \times n$ identity and $\mathbb{D}u$ is the symmetrized gradient of $u$. By defining the divergence operator to act on tensors in the canonical way, we have $\diverge S(p,u) = \nabla p - \Delta u + \nabla \diverge u$.  This means that in the first equation of \eqref{eq:2} we can rewrite 
\begin{equation}\label{eq:stress_div}
 \nabla P - \Delta w = \diverge S(P,w).
\end{equation}

\subsection{Shear flows and perturbations}

The system \eqref{eq:2} admits steady shear flow solutions that reduce to hydrostatic equilibrium when $\kappa=0$.  To see this, we suppose that $\zeta = 0, \tilde{\mathfrak{f}} = 0, \tilde{\mathcal{T}} = 0$. We then define the smooth functions $s_0,s: \R \to \R$ via 
\begin{align}\label{eq: shear}
	s_0(x_n) =   b x_n - \frac{x_n^2}{2} \text{ and }  s(x_n) = \kappa s_0(x_n).
\end{align}
We then define the steady shear velocity field $U^1: \Omega_{b} \to \R^n$ via  
\begin{align}\label{eq:modU1}
	 W^1(x',x_n) = s(x_n) e_1,
\end{align}
and the equilibrium hydrostatic pressure $p: \Omega_{b} \to \R$ via 
\begin{align}\label{eq: mod pressure}
	 p(x', x_n) = P_{ext} - x_n + b.
\end{align}
See Figure \ref{fig:fluid_domain} for a sketch of the $W^1$ profile.  One can readily check that $w = W^1, P = p, \zeta = 0$ is a steady shear flow solution to \eqref{eq:2} when $\tilde{\mathfrak{f}} = 0, \tilde{\mathcal{T}} = 0$.  Note that these shear flow solutions are special solutions induced by $\kappa \neq 0$, and they exist due to the presence of viscosity, with no clear analogue in the Euler system. In the literature, these solutions are sometimes referred to as Nusselt solutions.

We will study the system \eqref{eq:2} as a perturbation of this steady solution.  We define the perturbation of the velocity field and pressure field given by 
\begin{align} \label{eq:first perturb}
	\overline{w}^1 (x,t) = w(x,t) + W^1(x), \quad \overline{P}^1(x,t) = P(x,t) + p(x),
\end{align}
and we see that $(\overline{w}^1, \overline{P}^1, \zeta)$ is a solution to \eqref{eq:2} if and only if $(w,P, \zeta)$ satisfies 
\begin{align} \label{eq:3}
	\begin{cases}
		\p_t w + \diverge S(P,w) + ( w +  W^1) \cdot \nabla ( w+ W^1) = \tilde{\mathfrak{f}}, & \text{in} \; \Omega_{b + \zeta(\cdot, t)}\\
		\diverge w  = 0, & \text{in} \; \Omega_{b + \zeta(\cdot, t)} \\
		\p_t \zeta + \nabla' \zeta \cdot w' + (\p_1 \zeta) s(\zeta+b)= w_n ,  & \text{on} \; \Sigma_{b + \zeta(\cdot, t)} \\	
		S(P,w) \nu =  [(\zeta -\sigma \mathcal{H}(\zeta) )I  + \mathcal{\tilde{T}} - \kappa \zeta (e_1 \otimes e_n + e_n \otimes e_1)]\nu    ,  & \text{on} \; \Sigma_{b + \zeta(\cdot, t)}	\\
		w= 0, & \text{on} \; \Sigma_0,
\end{cases}
	\end{align}
Here we have utilized the identity \eqref{eq:stress_div} in the first equation of \eqref{eq:3}, and the tensor product $v \otimes w \in \R^{n \times n}$ of two vectors $v,w \in \R^{n}$ is defined in the standard way via $(v \otimes w)_{ij} = v_i w_j$.

\subsection{Traveling wave solutions around shear flows}
In this paper our main goal is to construct traveling wave solutions to \eqref{eq:3}, which are solutions that are stationary when viewed in a coordinate system moving at a constant speed. For this stationary condition to hold, the moving coordinate system must travel at a constant velocity parallel to the flat rigid surface $\Sigma_0$. In this paper we assume that the traveling waves move at a constant velocity in the direction of incline; in other words, that the moving coordinate system's velocity relative to the Eulerian coordinates of \eqref{eq:1} is $\gamma  e_1$ for some $\gamma  \in \R \setminus \{ 0 \}$. The speed of the traveling wave is then $\abs{\gamma }$, and $\sgn(\gamma )$ indicates the direction of travel along the $e_1$ axis.

Next we proceed to reformulate the problem in the traveling coordinates. We assume that the stationary free surface is given by an unknown function $\eta: \Sigma \to (-b, \infty)$, and it is related to $\zeta$ via $\zeta(x', t) = \eta(x' - \gamma  t e_1)$. The stationary velocity, pressure, force, and stress $v,q,\mathfrak{f}, \mathcal{T}$ are related to $w,P,\tilde{\mathfrak{f}},\tilde{\mathcal{T}}$ via 
\begin{align} \label{eq: traveling ansatz}
	w(x,t) = v(x-\gamma  te_1), \quad
	P(x,t) = q(x - \gamma  t e_1),\quad
	\tilde{\mathfrak{f}}(x,t) = \mathfrak{f}(x - \gamma  t e_1), \quad
	\tilde{\mathcal{T}}(x,t) = \mathcal{T}(x' - \gamma  t e_1).
\end{align}
Then using \eqref{eq: traveling ansatz}, the system \eqref{eq:3} reduces to a time-independent system for $(v, q, \eta)$ given the forcing terms $\mathfrak{f}$ and $\mathcal{T}$,
	\begin{align}\label{eq:4}
		\begin{cases}
			\diverge S(q,v) - \gamma  e_1  \cdot \nabla v + ( v +  W^1) \cdot \nabla ( v + W^1)  = \mathfrak{f}, & \text{in} \; \Omega_{b + \eta} \\
			\diverge v = 0 , & \text{in} \; \Omega_{b + \eta} \\
			-\gamma  \p_1 \eta + \nabla' \eta \cdot v' + (\p_1 \eta) s(\eta+b)= v_n  ,  & \text{on} \; \Sigma_{b + \eta} \\	
			S(q,v) \mathcal{N} =   [(\eta -\sigma \mathcal{H}(\eta) )I  + \mathcal{T}  - \kappa \eta  (e_1 \otimes e_n + e_n \otimes e_1) ]\mathcal{N}   ,  & \text{on} \; \Sigma_{b + \eta}	\\
		v= 0, & \text{on} \; \Sigma_0,
	\end{cases}
	\end{align}
where we have defined the non-unit normal vector field
\begin{align}\label{eq:N}
	\mathcal{N} = (-\nabla' \eta, 1).
\end{align}
For technical reasons to be discussed later in Section~\ref{sec: discussion}, it is convenient for us to remove the $\eta$ terms appearing in the fourth equation of \eqref{eq:4} by introducing an additional time-independent perturbation term 
\begin{align}\label{eq:modU2}
	\quad W^2(x',x_n)= \kappa x_n \eta(x') e_1
\end{align}
and by defining the modified shear velocity and modified pressure via
\begin{align}\label{eq: mod pressure zeta}
	 \overline{v}^1(x,t) = v + W^2(x), \; \overline{q}^1(x) = q(x) + \eta(x').
\end{align}
One may readily check that $(\overline{v}^1, \overline{q}^1, \zeta)$ is a solution to \eqref{eq:4} if and only if $(v,q,\eta)$ satisfies 
\begin{align} \label{eq:main unflattened}
	\begin{cases}
		\diverge S(q,v) -\gamma e_1 \cdot \nabla v - \gamma \kappa x_n \p_1 \eta e_1+ (v + W^1 + W^2 )\cdot \nabla (v+W^1 + W^2)  + (\nabla' \eta,0)  \\ 
		\hspace{0.3in} - \kappa x_n \Delta' \eta(x') e_1 - \kappa (x_n \nabla' \p_1 \eta, \p_1 \eta) = \mathfrak{f}, & \text{in} \; \Omega_{b + \eta} \\
		\diverge{v} + \kappa x_n \p_1 \eta(x') = 0 , & \text{in} \; \Omega_{b + \eta} \\
		 (-\gamma  + s(\eta+b) + \kappa (\eta+b)\eta)  \p_1 \eta = v \cdot \mathcal{N} ,  & \text{on} \; \Sigma_{b + \eta} \\	
		 S(q,v) \mathcal{N} =   [-\sigma \mathcal{H}(\eta)I  + \mathcal{T}  + \kappa (b+\eta) (e_1 \otimes (\nabla' \eta,0) + (\nabla' \eta,0)  \otimes e_1) ] \mathcal{N} ,  & \text{on} \; \Sigma_{b + \eta}	\\
		v= 0, & \text{on} \; \Sigma_0.
\end{cases}
\end{align}
We note that by using the perturbations introduced in \eqref{eq: mod pressure zeta}, we have replaced the $\eta$ terms in the fourth equation of \eqref{eq:4} by either derivatives of $\eta$ or products of $\eta$ and its derivatives, at the price of introducing additional terms in the first equation.

\subsection{Previous work}\label{sec:previous work}
The system \eqref{eq:1} and its variants have been studied extensively in the mathematical literature.  For a brief survey of the subject, we refer to Section 1.2 of Leoni-Tice \cite{leonitice} and Section 1.4 of Tice \cite{tice18}. For a more thorough review of the subject, we refer the surveys of Tolland \cite{tolland}, Groves \cite{groves}, Strauss \cite{strauss}, the recent paper by Strauss et al. \cite{straussONEPAS} in the inviscid case, and to the surveys of Zadrzy\'{n}ska \cite{zadrzynska} and Shibata-Shimizu \cite{S-S} in the viscous case.

When $\kappa = 0$, the small data theory for the free boundary problem \eqref{eq:1} over periodic domains is well-understood in dimension $n=3$. For the problem with surface tension, Nishida-Teramoto-Yoshihara \cite{N-T-Y} constructed global periodic solutions and proved that they decay exponentially fast to equilibrium. For the problem without surface tension, Hataya \cite{hataya} constructed global solutions decaying at a fixed algebraic rate, and later Guo-Tice \cite{guo-tice2013} constructed global solutions decaying almost exponentially. In the non-periodic setting, Beale \cite{beale} established the local-wellposedness of solutions without surface tension, and Beale \cite{beale84} proved the global existence of solutions with surface tension. Beale-Nishida \cite{bealenishida} later established that the aforementioned solutions decay at an algebraic rate. Tani-Tanaka \cite{tani} proved the global existence solutions with and without surface tension under milder assumptions on the initial data, but did not study their decay rates. Guo-Tice \cite{guo-tice2} proved that for the problem without surface tension, the global solutions decay at a fixed algebraic rate.

The investigation of the dynamics of viscous shear flows without free boundary is classical and dates back to the work of Orr \cite{orr} and Sommerfeld \cite{sommerfield}, where they noted the so-called viscous destabilization phenomenon. This was subsequently investigated formally by many authors in the physics literature, including Heisenberg \cite{heisenberg}, Lin \cite{lin}, and Tollmien \cite{tollmien}.  However, it wasn't until recently that a rigorous mathematical proof for the instability of viscous incompressible shear flows without free boundary appeared, in the work of Grenier-Guo-Nguyen \cite{G-G-N}.

When $\kappa \neq 0$, much less is known about the dynamics of the free boundary problem \eqref{eq:1}. Ueno \cite{ueno} studied the 2D problem with surface tension in the thin film regime and established uniform estimates of solutions with respect to the thinness parameter. Padula \cite{padula1} studied the 3D problem with surface tension and proved sufficient conditions for asymptotic stability under a priori assumptions on the global existence of solutions. Tice \cite{tice18} studied the asymptotic stability of shear flow solutions to the nonlinear problem with and without surface tension, and proved that solutions decay exponentially fast to equilibrium with surface tension and almost exponentially without surface tension.

The existence of traveling wave solutions without background shear flows to \eqref{eq:1} first appeared in the recent work of Leoni-Tice \cite{leonitice}, and the results therein were extended by Stevenson-Tice \cite{noahtice} to a multilayer configuration.  To the best of our knowledge, there are no known results on the existence of traveling wave solutions around shear flows in the case when $\kappa \neq 0$ or when the cross section $\Sigma$ is periodic.

\subsection{Reformulation in a fixed domain}\label{sec:flatten}
We note that in \eqref{eq:main unflattened}, the fluid domain of interest $\Omega_{b +\eta}$ is dependent on an unknown free surface $\eta$. To bypass the difficulty of working in such a domain, we proceed to reformulate the problem on a fixed domain $\Omega_b = \Sigma \times (0,b)$, at the cost of worsening the nonlinearities in the system. 

To do so we introduce the flattening map $\mathfrak{F}_\eta: \overline{\Omega_b} \to \overline{\Omega_{b+\eta}}$ associated to a continuous function $\eta: \Sigma_b \to (-b,\infty)$, defined via 
\begin{align}\label{eq:flattening}
	\mathfrak{F}_\eta(x',x_n) = x + \frac{x_n \eta(x')}{b} e_n.
\end{align}
We note that by construction, $\mathfrak{F}_\eta \rvert_{\Sigma_0} = \Sigma_0$ and $\mathfrak{F}_\eta(x',b) = \Sigma_{\eta + b}$ (see Figure \ref{fig:fig_2}). Moreover, $\mathfrak{F}_\eta$ is bijective with its inverse $\mathfrak{F}_\eta^{-1}: \overline{\Omega_{b+\eta}} \to \overline{\Omega_b}$ given by 
\begin{align}\label{eq:inv flattening}
	\mathfrak{F}_\eta^{-1} (y)= \left(y', \frac{by_n}{b+ \eta(x')} \right).
\end{align}

\begin{figure}[!h]
	\includegraphics[scale=0.8]{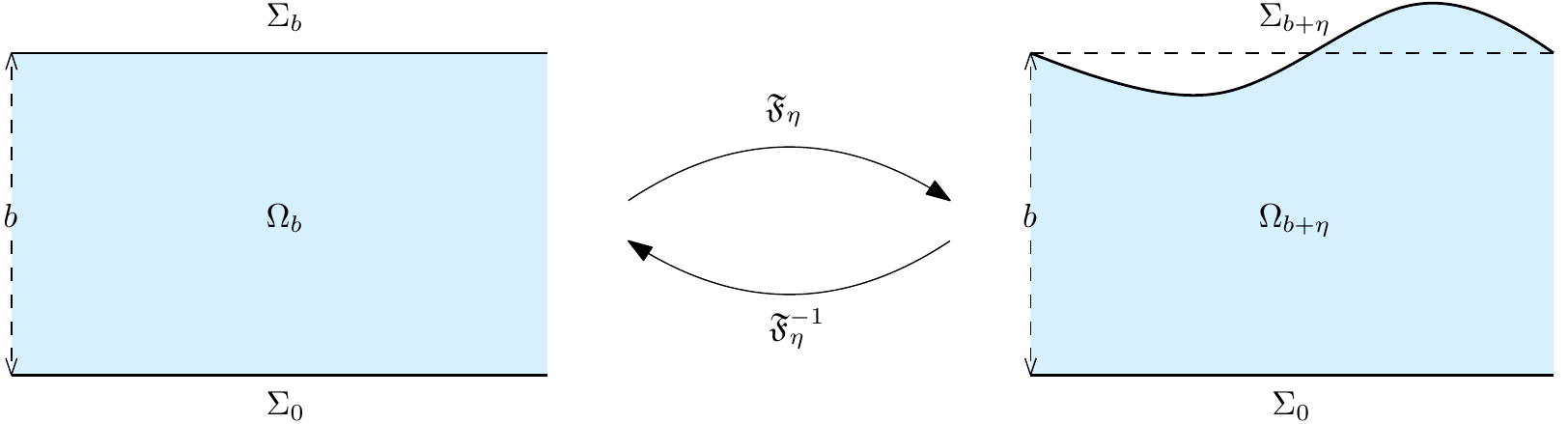}
	\caption{The flattening map $\mathfrak{F}_\eta$ and its inverse}
	\label{fig:fig_2}
\end{figure}

Throughout the rest of the paper we will typically suppress the dependence of the subsequent maps and their associated domains on $b$ and $\eta$, e.g. writing $\mathfrak{F}_\eta$ as $\mathfrak{F}$ and $\Omega_b$ as $\Omega$, unless these dependencies need to be emphasized. We also use the following abuse of notation throughout the paper: since the hypersurface $\Sigma_b \subset \Sigma \times \R$ is canonically diffeomorphic to $\Sigma$ via the projection map $\pi: \Sigma_b \to \Sigma$ given by $\pi(x',b) = x'$, we will use this to identify $H^s(\Sigma_b ; V)$ with $H^s(\Sigma;V)$ (and similar spaces) for any finite dimensional vector space $V$. This allows us to sometimes write $f(x')$ instead of $f(x',b)$, and allows us to apply the horizontal Fourier transform on $\Sigma_b$ in the natural way.

Following this convention, we note that $\mathfrak{F}$ is a homeomorphism inheriting the regularity of $\eta$, in the sense that if $\eta \in C^k(\Sigma;\R)$ then $\mathfrak{F}$ is a $C^k$ diffeomorphism. When $\eta$ is differentiable we may compute 
\begin{align}
\nabla \mathfrak{F}(x) = \begin{pmatrix}
I_{(n-1) \times (n-1)} & 0_{ (n-1) \times 1 } \\
\frac{x_n \nabla ' \eta(x')^T}{b} & 1 +\frac{\eta(x')}{b} 
\end{pmatrix}.
\end{align}
Thus, the Jacobian $\mathcal{J}$ and the inverse Jacobian $\mathcal{K}$ of $\mathfrak{F}$ are 
\begin{align}\label{eq:JandK}
\mathcal{J} = \det \nabla \mathfrak{F} = 1+ \frac{\eta}{b}, \quad 
\mathcal{K} = \frac{1}{J} = \frac{b}{b+\eta}.
\end{align}
We then introduce the matrix $\mathcal{A}: \Omega_b \to \R^{n\times n}$ defined via  
\begin{align}\label{eq:A}
\mathcal{A}(x) = (\nabla \mathfrak{F})^{-\intercal} =  \begin{pmatrix} 
	I_{(n-1) \times (n-1) } & \frac{-x_n \nabla' \eta(x')}{b+\eta(x')} \\
	0_{1 \times (n-1)} & \frac{b}{b+\eta(x')} 
\end{pmatrix} = \begin{pmatrix}
I_{(n-1) \times (n-1) } &   - x_n \mathcal{K} \frac{\nabla' \eta(x')}{b}\\
0_{1 \times (n-1)} & \mathcal{K}
\end{pmatrix},
\end{align}
the $\mathcal{A}$-dependent differential operators 
\begin{align} 
	(\nabla_{\mathcal{A}} f)_i = \sum_{j=1}^n \mathcal{A}_{ij} \p_j f, \quad (X \cdot \nabla_\mathcal{A} u)_i = \sum_{j,k=1}^n X_j \mathcal{A}_{jk} \p_k u_i, \quad \diverge_{\mathcal{A}} X = \sum_{i,j = 1}^n \mathcal{A}_{ij} \p_j X_i,
\end{align}
and also 
\begin{multline}
	(\mathbb{D}_{\mathcal{A}} u)_{ij} = \sum_{k=1}^n \mathcal{A}_{ik} \p_k u_j + \mathcal{A}_{jk} \p_k u_i,\quad S_{\mathcal{A}}(p,u) = pI - \mathbb{D}_\mathcal{A} u, \\ 
	 \diverge_{\mathcal{A}} S_{\mathcal{A}}(p,u) = \nabla_{\mathcal{A}} p - \Delta_{\mathcal{A}} u - \nabla_{\mathcal{A}} \diverge_{\mathcal{A}} u, 
	 \quad (\Delta_\mathcal{A} u)_i = \sum_{j,k,m=1}^n \mathcal{A}_{j,k} \p_k (\mathcal{A}_{jm} \p_m u_i).
\end{multline}
Now writing $u, U^1 ,U^2: \Omega \to \R^n, p: \Omega \to \R, f :\Omega \to \R^n$ via $u = v \circ \mathfrak{F}, U^1 =W^1 \circ \mathfrak{F},  U^2 = W^2 \circ \mathfrak{F}, p = q \circ \mathfrak{F}$, where $v,q$ satisfies \eqref{eq:3}, \eqref{eq:main unflattened} can be reformulated as the quasilinear system
\begin{align} \label{eq: main flattened}
	\begin{cases}	
		\diverge_\mathcal{A} S_\mathcal{A} (p, u) - \gamma  e_1 \cdot  \nabla_\mathcal{A} u  - \gamma \kappa (x_n + \eta \frac{x_n}{b})\p_1 \eta e_1+ \left( u + U^1 + U^2\right) \cdot  \nabla_\mathcal{A} \left( u + U^1 + U^2 \right) 
	 & \\
	 \hspace{0.3in} + (\nabla' \eta,0) - \kappa \left( x_n + \eta \frac{x_n}{b}\right) \Delta' \eta e_1   - \kappa ( (x_n + \eta \frac{x_n}{b}) \nabla' \p_1 \eta, \p_1 \eta)  = \mathfrak{f} \circ \mathfrak{F}_\eta
		, &  \text{in} \; \Omega \\
	\diverge_\mathcal{A} u+ \kappa \left( x_n + \eta(x') \frac{x_n}{b} \right) \p_1 \eta = 0, & \text{in} \; \Omega \\
	u\cdot \mathcal{N} = \left( - \gamma  + s(\eta + b) + \kappa (\eta + b) \eta \right) \p_1 \eta , & \text{on} \; \Sigma_b \\
	S_\mathcal{A}(p, u) \mathcal{N} = \left[ - \sigma \mathcal{H}(\eta) I + \mathcal{T}\circ \mathfrak{F}_\eta + \kappa (\eta + b) ( e_1 \otimes (\nabla' \eta,0) + (\nabla' \eta,0)  \otimes e_1)  \right] \mathcal{N}, & \text{on} \; \Sigma_b \\
	u = 0, & \text{on} \; \Sigma_0.
	\end{cases}
\end{align}

\subsection{Statement of main results}
We now state the main results obtained in this paper, though we delay a thorough discussion of them to the next subsection.  The first result concerns a key property of the anisotropic Sobolev space $X^s(\Sigma; \F)$, defined later in Section~\ref{sec: X^s} for $\F \in \{ \R, \C \}$, which serves as the container space for the free surface function $\eta$ when $\F = \R$.  

\begin{thm}[Proven later in Section~\ref{sec: X^s algebra}]\label{thm: 1}
	Suppose $s > d/2$, where $d = \dim \Sigma$ and $\Sigma$ is defined via \eqref{sigma}. Then the anisotropic space $X^s(\Sigma;\F)$ defined in Definition~\ref{def: X^s} is an algebra.
\end{thm}
The proof of this theorem utilizes anisotropic Littlewood-Paley techniques heavily inspired by \cite{benoit}. 
We note that this theorem is only non-trivial in the case when $1 \in R_\Sigma, \abs{R_\Sigma} \ge 1$, where the set $R_\Sigma$ is defined via \eqref{eq: R}. It turns out that in all the remaining possibilities of $\Sigma$, the anisotropic Sobolev $X^s(\Sigma ; \F)$ coincides with the standard Sobolev space $H^s(\Sigma ;\F)$, so the result follows directly from standard Sobolev theory. In Lemma~\ref{lem: Xs Hs equiv} we give the precise characterization of product domains $\Gamma$ for which $H^s(\Gamma ; \F) = X^s(\Gamma; \F)$ and $H^s(\Gamma ; \F) \subsetneq X^s(\Gamma; \F)$.  Moreover, we show in Proposition \ref{prop: Xs not complete case 3} that if $\Gamma = L_1 \T \times \R$, then the space $X^s(\Gamma; \F)$ is not complete, which is an initial indication of why we cannot allow for $\Sigma = L_1\T \times \R$ in our analysis.

With Theorem~\ref{thm: 1} in hand, we are able to prove the solvability of flattened system \eqref{eq: main flattened}, and by extension the solvability of the original system \eqref{eq:main unflattened}. Before stating the solvability results we make a quick comment on the forcing terms in \eqref{eq:main unflattened} and \eqref{eq: main flattened}. We first note that the bulk force $\mathfrak{f}$ in \eqref{eq:main unflattened} needs to be defined in the domain $\Omega_{b+\eta}$, which depends on the unknown free surface function $\eta$.  In order to work with $\mathfrak{f}$ without the implicit need to know $\eta$ first, we will assume a priori that $\mathfrak{f}$ is defined in all of $\Sigma \times \R$.  We will also need the map $(\mathfrak{f}, \eta) \mapsto \mathfrak{f} \circ \mathfrak{F}_\eta$ to be $C^1$, since we need to invoke the implicit function theorem in our analysis in Section~\ref{nonlinear}. It is known that this is possible, for example, if the domain of $\mathfrak{f}$ is $H^{s+1}$ and its codomain is $H^s$. Due to the strip-like structure of the fluid domain and the flattening map defined in \eqref{eq:flattening}, we can in fact allow for another type of bulk force which preserves the regularity of its domain. A detailed discussion of this can be found in Section 1.4 of \cite{leonitice}.  Moving forward, we consider the generalized bulk forces $\mathfrak{f} + L_{\Omega_{b+\eta}} f$ and the generalized surface forces $\mathcal{T} \rvert_{\Sigma_{b+\eta}} + S_{b+\eta}T$ in \eqref{eq:main unflattened}, where $L_{\Omega_{\eta+b}} f(x',x_n) = f(x')$ and  $S_{b+\eta} T(x',x_n) = T(x')$.
 
With this in mind, we are able to prove the solvability of \eqref{eq: main flattened}.
The spaces $C^k_b, C^k_0, \Hzeros{s}(\Omega;\R^n)$ referred to in the following results are defined in Section~\ref{sec:notation}, and the space $\mathcal{X}^s$ is defined in Definition~\ref{defn:Xsfordata}. We also write $\Sigma \ni x' = x'_{R_\Sigma} + x'_{T_\Sigma}$ in accordance with the notation introduced in \eqref{eq: xi R}.

\begin{thm}[Proved later in Section~\ref{sec: main flattened}]\label{thm:main1}
	Let $\Sigma$ be given by \eqref{sigma} and suppose that  $\N \ni s > \frac{n}{2}$.  Further suppose that either $\sigma > 0$ and $n \ge 2$ or else $\sigma = 0$ and $n = 2$. Then there exists open sets 
\begin{align} 
		\mathcal{U}^s \subset (\R \setminus \{0\}) \times \R \times H^{s+2}(\Sigma \times \R ; \R^{n \times n}_{\sym}) \times H^{s+\frac{1}{2}}(\Sigma ; \R^{n\times n}_{\sym}) \times H^{s+1}(\Sigma \times \R; \R^n) \times H^s(\Sigma ; \R^n) 
\end{align}
and $\mathcal{O}^s \subset \mathcal{X}^s$ such that the following hold.
\begin{enumerate}
	\item $(0,0,0) \in \mathcal{O}^s$, and for every $(u,p,\eta) \in \mathcal{O}^s$ we have that 
		\begin{align} 
		u \in C_b^{2 + \left\lfloor s-\frac{n}{2} \right\rfloor} (\Omega ; \R^n), p \in C^{1 + \left\lfloor s - \frac{n}{2} \right\rfloor}_b (\Omega; \R), \eta \in C^{3 + \left\lfloor s - \frac{n}{2} \right\rfloor}_0 (\Sigma; \R)		\end{align}
			and 
			\begin{align} 
				\lim_{|x_{R_\Sigma}'| \to \infty} \p^\alpha u(x) &= 0 \; \text{for all} \; \alpha \in \N^n \; \text{such that} \; \abs{\alpha} \le 2 + \left\lfloor s - \frac{n}{2} \right\rfloor \\
				\lim_{|x_{R_\Sigma}'| \to \infty} \p^\alpha p(x) &= 0 \; \text{for all} \; \alpha \in \N^n \; \text{such that} \; \abs{\alpha} \le 1 + \left\lfloor s - \frac{n}{2} \right\rfloor,
		\end{align}
		$\max_{\Sigma} \abs{\eta} \le \frac{b}{2}$, and the flattening map $\mathfrak{F}_\eta$ is a bi-Lipschitz homeomorphism and a $C^{3+\left\lfloor s-\frac{n}{2} \right\rfloor}$ diffeomorphism. 
	\item We have $(\R \setminus \{0\}) \times \{0\} \times \{ 0\} \times \{ 0 \} \times \{ 0 \} \times \{0 \} \subset \mathcal{U}^s$ 
	\item For each $(\gamma,\kappa, \mathcal{T}, T, \mathfrak{f}, f) \in \mathcal{U}^s$, there exists a unique $(u, p, \eta) \in \mathcal{O}^s$ classically solving
\begin{align} \label{eq: thm 1 main flattened}
	\begin{cases}	
		\diverge_\mathcal{A} S_\mathcal{A} (p, u) - \gamma  e_1 \cdot  \nabla_\mathcal{A} u  - \gamma \kappa (x_n + \eta \frac{x_n}{b})\p_1 \eta e_1+ \left( u + U^1 + U^2\right) \cdot  \nabla_\mathcal{A} \left( u + U^1 + U^2 \right) 
	 & \\
	 \hspace{0.3in} + (\nabla' \eta,0) - \kappa \left( x_n + \eta \frac{x_n}{b}\right) \Delta' \eta e_1   - \kappa ( (x_n + \eta \frac{x_n}{b}) \nabla' \p_1 \eta, \p_1 \eta)  = \mathfrak{f} \circ \mathfrak{F}_\eta + L_{\Omega_b} f
		, &  \text{in} \; \Omega \\
	\diverge_\mathcal{A} u+ \kappa \left( x_n + \eta(x') \frac{x_n}{b} \right) \p_1 \eta = 0, & \text{in} \; \Omega \\
	u\cdot \mathcal{N} = \left( - \gamma  + s(\eta + b) + \kappa (\eta + b) \eta \right) \p_1 \eta , & \text{on} \; \Sigma_b \\
	S_\mathcal{A}(p, u) \mathcal{N} = \left[ - \sigma \mathcal{H}(\eta) I + \mathcal{T}\circ \mathfrak{F}_\eta \vert_{\Sigma_b} + S_b T + \kappa (\eta + b) ( e_1 \otimes (\nabla' \eta,0) + (\nabla' \eta,0)  \otimes e_1)  \right] \mathcal{N}, & \text{on} \; \Sigma_b \\
	u = 0, & \text{on} \; \Sigma_0.
	\end{cases}
\end{align}

	\item The map $\mathcal{U}^s \ni (\gamma,\kappa, \mathcal{T}, T, \mathfrak{f}, f) \mapsto (u,p,\eta) \in \mathcal{O}^s$ is $C^1$ and locally Lipschitz. 
\end{enumerate}
\end{thm}

A few remarks are in order.  First, we note again that in the physically relevant case $n =3$, we cannot solve \eqref{eq: thm 1 main flattened} in the stated spaces if $\Sigma = L_1 \T \times \R$.  Second, Theorem \ref{thm:main1} is analogous to Theorem 1.1 in \cite{leonitice} for the problem without incline and periodicity, and Theorem 1 in \cite{noahtice} for the multilayer problem.  However, our choice of the function space $\mathcal{X}^s$ is slightly different from the one used there because we have formulated the problem \eqref{eq: thm 1 main flattened} in a slightly different manner, with the $\eta$ term shifted from the dynamic boundary condition into the bulk.  This results in the pressure belonging to a standard Sobolev space rather than a specialized anisotropic one, as in \cite{leonitice,noahtice}.  Third, we can say something about the set of $(\gamma,\kappa)$ parameters for which we can produce solutions,
\begin{equation}
\mathfrak{P}^s =  \{(\gamma,\kappa) \in \R \backslash \{0\} \times \R  \mid  (\gamma,\kappa,\mathcal{T},T,\mathfrak{f},f) \in \mathcal{U}^s \text{ for some } (\mathcal{T},T,\mathfrak{f},f) \}.
\end{equation}
Indeed, an examination of our proof of Theorem \ref{thm:main1} shows that for every $\gamma \in \R \backslash \{0\}$, there exists a $\kappa_0(\gamma) >0$, depending on $\gamma$ and the other physical parameters in a semi-explicit way (see Theorem \ref{thm:M}), such that 
\begin{equation}
 (-\kappa_0(\gamma),\kappa_0(\gamma)) \subseteq \{ \kappa \in \R \mid (\gamma,\kappa) \in \mathfrak{P}^s\}. 
\end{equation}
However, the estimates in Theorem \ref{thm:M} suggest that for each $\kappa \in \R\backslash \{0\}$, the set $\{\gamma \in \R \backslash \{0\} \mid (\gamma,\kappa) \in \mathfrak{P}^s\}$ is bounded, and possibly empty for large $\abs{\kappa}$.  We conjecture that this is indeed the case, but we do not have a complete proof here due to the complicated dependence the operator norm of $L_{0,\sigma}$, defined in \eqref{eq:Lkappa}, on $\gamma$.

Using solutions constructed in the flattened domain via Theorem~\ref{thm:main1}, we can then produce solutions to the unflattened system \eqref{eq:main unflattened}. This leads us to the final result of this paper.

\begin{thm}[Proved later in Section~\ref{sec: main unflattened}]\label{thm:main2}
	Let $\Sigma$ be given by \eqref{sigma}, and suppose that  $\N \ni s > \frac{n}{2}$.  Suppose that either $\sigma > 0$ and $n\ge 2$ or else $\sigma = 0$ and $n = 2$. Let 
	\begin{align} 
		\mathcal{U}^s \subset (\R \setminus \{0\}) \times \R \times H^{s+2}(\Sigma \times \R ; \R^{n \times n}_{\sym}) \times H^{s+\frac{1}{2}}(\Sigma ; \R^{n\times n}_{\sym}) \times H^{s+1}(\Sigma \times \R ; \R^n) \times H^s(\Sigma  ; \R^n) 
	\end{align}
	and $\mathcal{O}^s \subset \mathcal{X}^s$ be the same open sets as in Theorem~\ref{thm:main1}. Then for each $(\gamma,\kappa, \mathcal{T},T,\mathfrak{f},f) \in \mathcal{U}^s$, there exist 
	\renewcommand\labelitemi{\raisebox{0.25ex}{\tiny$\bullet$}}
	\begin{itemize}
	\item a free surface function $\eta \in X^{s+5/2}(\Sigma ; \R) \cap C_0^{\lfloor s - n /2 \rfloor + 2}(\Sigma; \R)$ such that $\max_{\Sigma} \abs{\eta} \le b/2$, and $\mathfrak{F}_\eta$ as defined in \eqref{eq:flattening} a bi-Lipschitz homeomorphism and a $C^{\lfloor s - n/2 \rfloor + 3}$ diffeomorphism, 
	\item a velocity field $v \in \Hzeros{s+2}(\Omega_{b+\eta} ; \R^n) \cap C^{\lfloor s - n/2 \rfloor + 2}(\Omega_{b+\eta} ; \R^n )$,
	\item a pressure $q \in H^{s+1}(\Omega_{b+\eta}; \R) \cap C_b^{\lfloor s - n/2 \rfloor + 1}(\Omega_{b+\eta}; \R)$,
	\item and constants $C, R > 0$, 
	\end{itemize}
	such that the following hold:
	\begin{enumerate}
	\item $(v,q,\eta)$ are classical solutions to 
\begin{align}  
	\begin{cases}
		\diverge S(q,v) -\gamma e_1 \cdot \nabla v - \gamma \kappa x_n \p_1 \eta e_1+ (v + W^1 + W^2 )\cdot \nabla (v+W^1 + W^2)  + (\nabla' \eta,0)  \\ 
		\hspace{0.3in} - \kappa x_n \Delta' \eta(x') e_1 - \kappa (x_n \nabla' \p_1 \eta, \p_1 \eta) = \mathfrak{f} + L_{\Omega_{b+\eta}}f, & \text{in} \; \Omega_{b + \eta} \\
		\diverge{v} + \kappa x_n \p_1 \eta(x') = 0 , & \text{in} \; \Omega_{b + \eta} \\
		 (-\gamma  + s(\eta+b) + \kappa (\eta+b)\eta)  \p_1 \eta = v \cdot \mathcal{N} ,  & \text{on} \; \Sigma_{b + \eta} \\	
		S(q,v) \mathcal{N} =   [-\sigma \mathcal{H}(\eta)I  + \mathcal{T}\vert_{\Sigma_{b+\eta}} + S_{b+\eta}T  + \kappa (b+\eta) (e_1 \otimes (\nabla' \eta,0) + (\nabla' \eta,0)  \otimes e_1) ] \mathcal{N} ,  & \text{on} \; \Sigma_{b + \eta}	\\
		v= 0, & \text{on} \; \Sigma_0.
\end{cases}
\end{align}

	\item $(v\circ \mathfrak{F}_\eta, q \circ \mathfrak{F}_\eta, \eta) \in \mathcal{O}^s \subset \mathcal{X}^s$, 
	\item If $(\gamma_*,\kappa_\ast, \mathcal{T}_*, T_*, \mathfrak{f}_*, f_*) \in U^s$ satisfy 
	\begin{align}
		 \abs{\gamma - \gamma_*} +\abs{\kappa-\kappa_\ast} + \norm{\mathcal{T} - \mathcal{T}_*}_{H^{s+2}} + \norm{T - T_*}_{H^{s+1/2}} + \norm{\mathfrak{f} - \mathfrak{f}_*}_{H^{s+1}} + \norm{f - f_*}_{H^{s}} < R,   
	\end{align}
	then the corresponding solution triple $(v_*, q_*, \eta_*)$ satisfy 
	\begin{multline}
		 \norm{(v\circ\mathfrak{F}_\eta,q\circ\mathfrak{F}_\eta,\eta) - (v_*\circ\mathfrak{F}_{\eta_\ast},q_*\circ\mathfrak{F}_{\eta_\ast},\eta_*)}_{\mathcal{X}^s} \\
		 \le C (\abs{\gamma - \gamma_*} +\abs{\kappa-\kappa_\ast} + \norm{\mathcal{T} - \mathcal{T}_*}_{H^{s+2}} + \norm{T - T_*}_{H^{s+1/2}} + \norm{\mathfrak{f} - \mathfrak{f}_*}_{H^{s+1}} + \norm{f - f_*}_{H^{s}}).  
	\end{multline} 
	\end{enumerate}
	 \end{thm}

\subsection{Discussion of techniques and plan of attack}\label{sec: discussion}

As in \cite{leonitice}, our strategy for producing solutions to \eqref{eq: main flattened} is to employ an implicit function theorem  argument built on the linearization of \eqref{eq: main flattened}, 
\begin{align} \label{eq:linear without decouple}
	\begin{cases}
		\diverge S(p,u)   - \gamma \p_1 u  - \gamma \kappa x_n \p_1 \eta e_1+ (\nabla' \eta, 0) + s(x_n) \p_1 u \\ \hspace{0.3in} + s'(x_n) u_n e_1 + \kappa x_n s(x_n) \p_1 \eta e_1   - \kappa x_n \Delta' \eta e_1  - \kappa ( x_n \nabla' \p_1 \eta, \p_1 \eta)  = f, & \text{in} \; \Omega \\
		\diverge u + \kappa x_n \p_1 \eta= g, & \text{in} \; \Omega \\
		u_n + (\gamma - \frac{\kappa b^2}{2} ) \p_1 \eta = h, & \text{on} \; \Sigma_b \\
		S(p,u) e_n + \sigma \Delta' \eta e_n = k, & \text{on} \; \Sigma_b \\
		u = 0, & \text{on} \; \Sigma_0.
	\end{cases}
\end{align}

When $\kappa = 0$ and $\Sigma = \R^{n-1}$, the strategy for producing solutions to \eqref{eq:main unflattened} and is discussed extensively in Section 1.5 of \cite{leonitice}. To succinctly summarize their approach, a viable strategy for producing solutions to \eqref{eq:2} is to use the implicit function theorem after first proving that  \eqref{eq:linear without decouple} induces an isomorphism $(u,p,\eta) \mapsto (f,g,h,k)$ between a pair of identified function spaces. Since \eqref{eq:linear without decouple} is not a standard elliptic system in the sense considered in Agmon-Douglis-Nirenberg \cite{a-d-n}, as the unknown free surface $\eta$ only appears in the equations imposed on the boundary, we follow the decoupling strategy in \cite{leonitice} and first study the overdetermined system 
\begin{align}\label{eq: intro overdet}
	\begin{cases}
		\diverge S(p,u) - \gamma \p_1 u= f & \text{in} \; \Omega \\
	\diverge u = g, & \text{in} \; \Omega \\
	S(p,u)e_n  = k, & \text{on} \; \Sigma_b \\
	u_n = h, & \text{on} \; \Sigma_b \\
	u= 0, & \text{on} \; \Sigma_{0}
\end{cases}
\end{align}
and its associated compatibility conditions. By the adjoint compatibility condition \eqref{eq: 2nd comp fourier}, the free surface function $\eta$ can be constructed from the data tuple $(f,g,h,k)$ by means of the pseudodifferential equation $\rho_\gamma(\xi) \hat{\eta}(\xi) = \psi(\xi)$ for $\xi \in \hat{\Sigma}$, where $\hat{\Sigma
}$ is the dual group of $\Sigma$ defined via \eqref{eq: hat gamma}, and $\psi, \rho_\gamma$ are defined in \eqref{eq:psi defn}, \eqref{eq:rho} as
\begin{multline}\label{eq: psi rho}
	 \psi(\xi) = \int_0^b \left( \hat{f}(\xi,x_n) \cdot \overline{V(\xi,x_n,-\gamma)} - \hat{g}(\xi,x_n) \overline{Q(\xi,x_n,-\gamma)} \right) \; dx_n - \hat{k}(\xi) \cdot \overline{V(\xi,b,-\gamma)} + \hat{h}(\xi), \\ \rho_\gamma(\xi) = 2\pi i \gamma \xi_1 + (1+4\pi^2 \abs{\xi}^2 \sigma)\overline{m(\xi,-\gamma)}.
\end{multline}
Here $V,Q,m$, defined via \eqref{eq:multiplers}, are symbols of special pseudodifferential operators. In order to solve the pseudodifferential equation for $\eta$, we need the precise asymptotics of $V,Q,m$ on $\hat{\Sigma}$. In particular, we note that when $\gamma \neq 0$ and $\Sigma$ is defined via \eqref{sigma}, the asymptotics of $m$ completely determine the asymptotics of $\rho_{\gamma}$ on $\hat{\Sigma}$, which for $\gamma \neq 0$ are given in Lemma~\ref{lem:rho} as 
\begin{align}\label{eq: intro rho asym}
	\begin{cases}
		\abs{\rho_\gamma(\xi)}^{2} \asymp ( \xi_1^2 + \abs{\xi}^4) \mathbbm{1}_{B(0,1)}(\xi) +(1 +  \abs{\xi}^2) \mathbbm{1}_{B(0,1)^c}(\xi), & \sigma > 0 \\
		\abs{\rho_\gamma(\xi)}^2 \asymp  \abs{\xi}^2\mathbbm{1}_{B(0,1)}(\xi) + (1 + \abs{\xi}^2) \mathbbm{1}_{B(0,1)^c}(\xi), & \sigma = 0, n= 2.
	\end{cases}
\end{align}
With the asymptotics of $\rho_\gamma$ in hand, we define $\hat{\eta}$ via $\hat{\eta}(\xi) = \psi(\xi) / \rho(\xi)$ for $\xi \neq 0$, which by \eqref{eq: intro rho asym} would imply 
\begin{multline}\label{eq: control}
	 \int_{\hat{\Sigma}} \frac{\xi_1^2+\abs{\xi}^4}{\abs{\xi}^2} \abs{\hat{\eta}(\xi)}^2 \mathbbm{1}_{B(0,1)} + (1+\abs{\xi}^2)^{s+5/2} \abs{\eta(\xi)}^2 \mathbbm{1}_{B(0,1)^c}\; d\xi \\ \asymp \int_{\hat{\Sigma}} \frac{\abs{\psi(\xi)}^2}{\abs{\xi}^2}\mathbbm{1}_{B(0,1)} + (1+\abs{\xi}^2)^{s+3/2} \abs{\psi(\xi)}^2 \mathbbm{1}_{B(0,1)} \; d\xi.
\end{multline}
Using the asymptotics of $V,Q$ in Theorem~\ref{thm:pseudo asymp}, and the functional framework built on the specialized anisotropic space $X^s(\Sigma;\R)$ that serves as the container space for the free surface function $\eta$ in Section~\ref{sec: X^s}, we can then utilize the equivalence \eqref{eq: control} to recover $\eta$ via Fourier reconstruction. We note that by the first items of Lemma~\ref{lem: Xs Hs equiv} and Theorem~\ref{thm:Xs}, the free surface function $\eta$ recovered through this process will be regular enough to be a classical function as opposed just being a tempered distribution, which is a crucial requirement for $\eta$ to be utilized in the subsequent nonlinear analysis. 

It is also worth mentioning that in the physically relevant dimension $n = 3$, we deliberately chose to ignore the configuration $\Sigma = (L_1 \T) \times \R$, as in this case we would have $\xi \in B_{\hat{\Sigma}}(0, r) \Longleftrightarrow \xi = (0,\xi_2), \abs{\xi_2} \le r$ for $r = \min(1,L_1^{-1})$. By \eqref{eq: psi rho}, this implies that 
$
	 \abs{\rho_{\gamma}(\xi)}^2 \mathbbm{1}_{B_{\hat{\Sigma}}(0,r)}\asymp \abs{\xi}^4 \implies \abs{\xi}^2 \abs{\hat{\eta}(\xi)}^2 \mathbbm{1}_{B_{\hat{\Sigma}}(0,r)} \asymp \frac{\abs{\psi(\xi)}^2}{\abs{\xi}^2}\mathbbm{1}_{B_{\hat{\Sigma}}(0,r)}.$
Unfortunately, in this scenario we only have $\dot{H}^1$ control over the low frequency modes of $\eta$, and by Proposition~\ref{prop: Xs not complete case 3} the corresponding anisotropic space $X^s((L_1 \T \times \R) ; \R)$ fails to be complete. If one were to consider the completion of this space, elements of the completion would be equivalence classes of tempered distributions modulo polynomials. Since we mandate elements of the container space for the free surface function $\eta$ to be classical functions, neither the space $X^s((L_1 \T \times \R) ; \R)$ nor its completion are suitable for the purposes of this paper. Practically speaking, this means that we cannot employ our techniques assuming a priori periodicity in the direction of incline. For the same reason, our framework also cannot produce stationary solutions in the case when $\gamma = 0$. 

By solving for $\eta$ through the aforementioned approach, we can then solve for $(u,p)$ by means of \eqref{eq:linear without decouple} and the linear isomorphism associated to \eqref{eq: intro overdet}. Fortunately, this is possible as the results from the linear analysis in \cite{leonitice} continue to hold over $\Sigma$ defined in \eqref{sigma}. This is mainly due to the fact that many results in \cite{leonitice} are proved by analyzing the low frequency behavior of functions in $H^s(\R^d;\R)$ and $X^s(\R^d ; \R)$. By Remark~\ref{rem: iso everything} and Remark~\ref{rem:zero fourier}, the analogous 
results for $H^s(\Sigma;\R)$ and $X^s(\Sigma;\R)$ can all be deduced from reducing a similar set of calculations over $\Sigma$ and $\hat{\Sigma}$ to the calculations in \cite{leonitice} over $\R^d$. As such, the results from \cite{leonitice} can be directly ported over with minimal modification of the proofs contained therein. 

Though, as seen in the case $\Sigma = L_1 \T \times \R$, the generalized space $X^s(\Gamma ; \R)$ can fail to be complete depending on the set $\Gamma$ over which it is defined. This is an initial indication that some care needs to be taken in generalizing the space $X^s(\R^d ; \R)$ introduced in \cite{leonitice} over general sets $\Gamma$. We devote the Appendix to developing the necessary tools to carefully define the space $X^s(\Gamma; \R)$ over product domains $\Gamma$ in Section~\ref{sec: X^s}, and also identify the ``good sets'' $\Sigma$ for which the spaces $X^s(\Sigma; \R)$ are compatible with our analysis, which led us to \eqref{sigma}.

Next, we discuss the role of the perturbations \eqref{eq:first perturb}, \eqref{eq: mod pressure zeta}, and in particular the role of \eqref{eq:modU2}. First, by renormalizing the pressure via \eqref{eq: mod pressure}, the vertical gravitational force $-e_n$ shifts from the bulk to the dynamic boundary conditions, appearing as the term $\zeta \nu$ on the right hand side of the fourth equation in \eqref{eq:3}; this term later appears as $\eta \mathcal{N}$ in the fourth equation of \eqref{eq:4}. By further modifying the pressure via \eqref{eq: mod pressure zeta}, we eliminate the term $\eta \mathcal{N}$ from \eqref{eq:4}, at the price of introducing $(\nabla' \eta, 0)$ to the bulk in \eqref{eq:main unflattened} and \eqref{eq: main flattened}. The advantage of this formulation is that the pressure $q$ in Theorem~\ref{thm:main2} lives in a standard Sobolev space $H^s(\Omega ; \R)$, as opposed to the alternative formulation in \cite{leonitice}, for which $q$ belongs to a specialized Sobolev space built on the anisotropic space $X^s(\Sigma; \R)$. 
 
A key difference between the system \eqref{eq:3} and the analogous system in \cite{leonitice} is that upon removing the  background shear flow \eqref{eq:modU1}, we are left with the term $\kappa \eta (e_1 \otimes e_n + e_n \otimes e_1) \mathcal{N}$ appearing in the fourth equation of \eqref{eq:4}. Since $\eta$ is expected to belong to a specialized Sobolev space $X^s(\Sigma; \R)$ that does not coincide with the standard Sobolev space $H^s(\Sigma ; \R)$ in general, attacking the problem at this level would require one to build specialized spaces for the data tuple $(f,g,h,k)$, and also prove that the associated linear maps remain to be isomorphisms in this modified functional framework. By introducing an additional perturbation \eqref{eq:modU2}, we were able to replace this term with terms that are all standard Sobolev in the regime $s > n/2$ thanks to the second item of Theorem~\ref{thm:Xs}, as the non-trivial terms are either derivatives of functions in $X^s(\Sigma; \R)$ or products of functions in $X^s(\Sigma; \R)$ and derivatives of functions in $X^s(\Sigma; \R)$. This approach allows us to directly employ the functional framework from \cite{leonitice} at the price of introducing worse nonlinearities in \eqref{eq:main unflattened} and \eqref{eq: main flattened}.

To construct solutions to \eqref{eq:main unflattened}, we use the associated linear isomorphism in conjunction with the implicit function theorem around the trivial solution. In order to invoke the implicit function theorem, the nonlinear maps associated to \eqref{eq: main flattened} first need to be well-defined on the same spaces used in the linear analysis. This requires some analysis in Section~\ref{sec: product est} to understand the mapping properties of functions in the specialized Sobolev space $X^s(\Sigma ; \R)$, and also of the mean-curvature operator $\mathcal{H}$ defined in \eqref{eq:H}. 

In addition, we also need a divergence-trace type compatibility condition \eqref{eq:divtrace} built into the container space $\mathcal{Y}^s$ defined in \eqref{defn:Ys} to hold; in Theorem~\ref{thm:smoothnessnonlinear}, we show that the compatibility condition requires
\begin{multline}\label{eq: intro est}
	\left[-s(\eta+b)\p_1 \eta -\kappa (\eta+b)\eta \p_1 \eta - \int_0^b J(x',x_n) \kappa \left( x_n + \eta(x') \frac{x_n}{b} \right) \p_1 \eta(\cdot) \; dx_n \right]_{\dot{H}^{-1}}\\ = \left[-\kappa (b^2 \p_1 \eta + b \p_1 \eta^2 + \frac{1}{3} \p_1 \eta^3)\right]_{\dot{H}^{-1}} < \infty,
\end{multline}
where the $\dot{H}^{-1}$ seminorm is defined in \eqref{eq:seminorm}. 

A major obstacle in proving \eqref{eq: intro est} is that low-frequency control of powers of $\eta$ is not immediately evident from the inclusion $\eta \in X^s(\Sigma ; \R)$.  One could in principle attempt to obtain control of powers of $\eta$ by way of Young's convolution inequality applied on the Fourier side.  However, in the model case $\Sigma = \R^d$ with $s > d/2$, this only leads to the inclusion $\eta^k \in H^s(\R^d; \R)$ for $k \ge 2 + \lceil 4/(2d-2) \rceil, d\ge 2$.   Unfortunately, this means that in the physically relevant case $n=3$ (so that $d=n-1 = 2$), this elementary argument only provides control over quartic or higher powers of $\eta$.  However, by the fourth item of Theorem~\ref{thm:Xs}, we know that if $\zeta \in X^s(\Sigma ;\R)$, then $\p_1 \zeta \in \dot{H}^{-1}(\Sigma; \R)$. Thus, a viable strategy for proving \eqref{eq: intro est} is to show that $X^s(\Sigma; \R)$ is an algebra. By Lemma~\ref{lem: Xs Hs equiv}, there are configurations of $\Sigma$ for which $X^s(\Sigma ; \R) = H^s(\Sigma; \R)$, in which case we know $X^s(\Sigma;\R)$ is an algebra for $s > \dim \Sigma /2$. In general, though, we only know that $H^s(\Sigma; \R) \hookrightarrow X^s(\Sigma; \R)$, so further analysis is required to show that $X^s(\Sigma;\R)$ is an algebra.  Fortunately, we are able to establish this in Theorem~\ref{thm: 1}, which is proved later in Section~\ref{sec: X^s algebra}.

As the linearization of \eqref{eq: main flattened} depends on $\kappa$, we also need to identify a parameter regime for $\kappa$ for which the associated linear map remain an isomorphism. To that end, in Section~\ref{sec:kappa est} we study the map 
\begin{multline} \label{eq:intro Lkappa}
	L_{\kappa,\sigma} (u,p,\eta)=( \diverge S(p,u) -  \gamma\kappa x_n \p_1 \eta e_1 - \gamma \p_1 u + (\nabla' \eta, 0) + s(x_n) \p_1 u + s'(x_n) u_n e_1 + \kappa x_n s(x_n) \p_1 \eta e_1 \\- \kappa x_n \Delta' \eta e_1  - \kappa(x_n \p_1 \nabla' \eta , \p_1 \eta), \diverge u + \kappa x_n \p_1 \eta, u_n\rvert_{\Sigma_b} + (\gamma - \frac{\kappa b^2}{2} ) \p_1 \eta,  S(p,u) e_n \rvert_{\Sigma_b}+ \sigma \Delta' \eta e_n )
\end{multline} 
induced by the linearization associated to \eqref{eq: main flattened}, and in Theorem~\ref{thm:M} we show that via a perturbative argument around $\kappa = 0$ that for fixed $\gamma$ and other physical parameters, there exists a $\kappa_0 > 0$ for which $L_{\kappa,\sigma}$ is an isomorphism over appropriate spaces for all $\kappa \in (-\kappa_0, \kappa_0)$. Synthesizing the aforementioned results and employing our strategy of invoking the implicit function theorem leads us to the solvability of \eqref{eq: main flattened} and \eqref{eq:main unflattened}.

Finally, we discuss the strategy for producing solutions to the unflattened system \eqref{eq:main unflattened} using solutions constructed in the flattened domain via Theorem~\ref{thm:main1}. To that end, we use the free surface function $\eta$ to build the flattening map and its inverse defined in \eqref{eq:flattening} and \eqref{eq:inv flattening} to undo the reformulation outlined in Section~\ref{sec:flatten}. This requires some results on the regularity of these maps, which is recorded in Section~\ref{sec: product est}. Fortunately, the same analysis can be adapted from \cite{leonitice} with minimal modification.

\subsection{Notational conventions and outline of article}\label{sec:notation}
In this subsection we discuss the notational conventions adopted throughout this paper. In this paper, $\N$ denotes the natural numbers including 0.  We always use $n \ge 1$ to denote the dimension of the flattened fluid domain $\Omega = \Sigma \times \R$. For $d \ge 1$, we will consider spaces defined over domains of the form 
\begin{align}\label{eq: Gamma f}
	\Gamma = \prod_{i=1}^{d} \Gamma_i, \;\Gamma_i = \R \; \text{or} \; L_i \T, \; L_i > 0, 
\end{align}
where $\Gamma$ is endowed with the natural group, topological, and smooth structures. In fact, the topology on $\Gamma$ is metrizable, and in this paper we equip $\Gamma$ with the metric 
\begin{align}
	 d_{\Gamma}(x,y) = \left( \sum_{i=1}^d d_i (x_i,y_i)^2 \right)^{1/2} \text{ for } d_i(x_i,y_i) = \begin{cases}
		 d_i(x_i, y_i) = \abs{x_i - y_i}, & i \in R_{\Gamma} \\
		 d_i(x_i, y_i) =  \inf \{ \abs{r-s} \mid r \in [x_i], \; s \in [y_i] \}, & i \in T_\Gamma.
	 \end{cases}
\end{align}
We write $\hat{\Gamma}$ to denote the dual group associated to $\Gamma$, defined via 
\begin{align}\label{eq: hat gamma}
	\hat{\Gamma} = \widehat{\prod_{i=1}^{d} \Gamma_i} = \prod_{i=1}^{d} \hat{\Gamma}_i, \; \hat{\Gamma}_i = \begin{cases}
		 \R, & \Gamma_i = \R \\
		 L_i^{-1} \Z, & \Gamma_i = L_i \T,
	\end{cases}
\end{align}
where $\hat{\Gamma}$ is also endowed with the obvious group, topological, and smooth structures. We also endow $\hat{\Gamma}$ with the metric induced by inclusion in $\R^d$:
\begin{align}
	 d_{\hat{\Gamma}}(x,y) = \sum_{i=1}^d \left( \abs{x_i-y_i}^2 \right)^{1/2}.
\end{align}
If $X$ is a metric space, we write $B_X(x,r)=\{ y \in X \mid d_X(x,y) < r \}$ for the open ball.

For $X \in \{\Gamma, \hat{\Gamma} \}$, we write $\mathscr{S}(X ; \C)$ to denote the Schwartz class of complex valued functions over $X$ and $\mathscr{S}'(X; \C)$ to denote the space of complex valued tempered distributions over $X$; a detailed treatment of how to define these space can be found in Appendix~\ref{tempered}. For $f \in \mathscr{S}(X ; \C)$ or $f \in \mathscr{S}'(X; \C)$, we denote its unitary Fourier and inverse Fourier transforms by 
\begin{align}
	 \hat{f}(\xi) = \mathscr{F}_X^{+} \{ f \} (\xi), \quad \check{f}(\xi) = \mathscr{F}_X^{-} \{ f \} (\xi),
\end{align}
where $\mathscr{F}_X^{\pm}$ is defined in \eqref{eq: Fourier def}. Sometimes we will also write 
\begin{align}
	 \mathscr{F}[f] (\xi)= \mathscr{F}_X^{+} \{ f \} (\xi), \quad \mathscr{F}^{-1} [f] (\xi) = \mathscr{F}_X^{-} \{ f \} (\xi).
\end{align}
For $\R \ni s > 1/2$ and a Lipschitz $\upzeta : \Sigma \to \R$ satisfying $\inf \upzeta > 0$, we define 
\begin{align}
	 \Hzeros{s}(\Omega_\upzeta ; \R^n) = \{ u \in H^s(\Omega_\upzeta ; \R^n) \mid u = 0 \; \text{on} \; \Sigma_0\},
\end{align}
where the equality is taken in the trace sense. In the case when $s = 1$ and $\zeta = b$, we endow $\Hzero(\Omega_b; \R^n)$ with the inner product 
\begin{align}
	 (u,v)_{\Hzero} = \frac{1}{2} \int \mathbb{D} u : \mathbb{D} v \; dx,
\end{align}
which generates the same topology as the standard $H^1$-norm by Korn's inequality (see Lemma 2.7 of \cite{beale}). For $k \in \N$, a real Banach space $V$, and a nonempty open set $U \subseteq \Gamma$ we define the space
\begin{align}
	 C^k(U ; V) = \{ f : U \to V \mid Lf \in C^k(\R^d ; V), \; Lf(x) = f([x]) \}.
\end{align}
We define 
\begin{align}
	 C^k_b(U;V) = \{ f : U \to V \mid f \in C^k(U;V) \; \text{and} \; \norm{f}_{C^k_b} < \infty \}, 
\end{align}
where 
\begin{align}
	\norm{f}_{C^k_b} = \sum_{\alpha \in \N^d, \abs{\alpha} \le k} \sup_{x \in U} \norm{\p^\alpha f(x)}_{V}. 
\end{align}
We also define the space $C_0^k(\Gamma;V) \subset C^k_b(\Gamma;V)$ to be the closed subspace 
\begin{align}
	C^k_0(\Gamma;V) = \{ f \in C^k_b(\Gamma;V) \mid \lim_{|x_{R_\Gamma}| \to \infty} \p^\alpha f(x) = 0 \; \text{for all} \; \alpha \in \N^d \; \text{such that} \; \abs{\alpha} \le k \},
\end{align}
where $x_{R_\Gamma} \in \Gamma$ is defined via \eqref{eq: R}. Lastly, we use $L^0(\Gamma; [0,\infty])$ to denote the set of non-negative measurable functions over $\Gamma$. 
	
We conclude this section by giving outline of the article. In Section~\ref{sec:functional}, we introduce the anisotropic Sobolev space $X^s(\Gamma; \R)$ and characterize the space based on the underlying product structure of the domain $\Gamma$. We then state its essential properties, and in Section~\ref{sec: X^s algebra} we prove that for $\Sigma$ defined via \eqref{sigma}, $X^s(\Sigma;\R)$ is an algebra in the supercritical regime $s > \dim \Sigma/2$. 

In Section~\ref{sec:linear analysis}, we first record the isomorphism associated to the overdetermined system \eqref{eq: intro overdet} and the asymptotics of the special pseudodifferential symbols used in the construction of the free surface function $\eta$. This allows us to prove the isomorphism associated to the linear system \eqref{eq:linear without decouple}. We then establish the parameter regime for $\kappa$ for which the flattened system \eqref{eq: main flattened} induces an isomorphism.

In Section~\ref{nonlinear}, we first record some key mapping properties of the anisotropic space $X^s(\Sigma;\R)$ and various nonlinear maps used in the subsequent analysis. Using these preliminary results, we show that the maps induced by \eqref{eq: main flattened} are well-defined and smooth, and we use this in conjunction with the implicit function theorem to produce solutions to \eqref{eq: main flattened}. We conclude the paper by using the solutions from \eqref{eq: main flattened} to produce solutions to \eqref{eq:main unflattened}.

\section{The anisotropic space \texorpdfstring{$X^s(\Gamma; \F)$}{Xs}} \label{sec:functional}

In this section we aim to generalize the anisotropic Sobolev space $X^s(\R^d;\R)$ introduced in Section 5 of \cite{leonitice}. First, in order to define such a space on more general product domains, we need to define a suitable notion of tempered distributions on $\Gamma$ and its Pontryagin dual $\hat{\Gamma}$; a detailed treatment of this can be found in Appendix~\ref{tempered}. Second, in order to understand how these spaces depend on the structure of the product domain $\Gamma$, we study how they behave under permutations of the factors of $\Gamma$; this is treated in Appendix~\ref{appendix: permutation}

\subsection{Definition of $X^s(\Gamma ; \F)$ and its general properties} \label{sec: X^s}
In this subsection we let $\R \ni s \ge 0$, $\N \ni d \ge 1$, and we consider the domain $\Gamma$ defined via \eqref{eq: Gamma f}. Recall that the Japanese bracket $\langle \cdot \rangle: \R^d \to [0,\infty)$ is defined via $\langle \xi \rangle = (1+\abs{\xi}^2)^{1/2}$. In \cite{leonitice}, the anisotropic space $X^s(\R^d; \R)$ is defined in terms of the Fourier multiplier $\omega_s : \R^d \setminus \{0\} \to \R$ given by
\begin{align}\label{eq: defn omega}
	 \omega_s(\xi) = 
		\mathbbm{1}_{B_{\hat{\Gamma}}(0,r)}(\xi)\left( \frac{\abs{\xi_1}^2}{\abs{\xi}^2} + \abs{\xi}^2  \right) + \mathbbm{1}_{\R^d \setminus B_{\hat{\Gamma}}(0,r)}(\xi) \br{\xi}^{2s}.
\end{align}
In this paper, we use the same formula in \eqref{eq: defn omega} to define $\omega_s$ on $\hat{\Gamma}$; in the purely toroidal case when $\hat{\Gamma} = \prod_{i=1}^d L_i^{-1} \Z$, we also define $\omega_s(0) = 1$ so that $\omega_s$ takes on the same value as the standard Sobolev multiplier $\langle \cdot \rangle^{2s}$ at $\xi = 0$.  The function $\omega_s$ is introduced due to its close relation to the symbol $\rho_\gamma$  from \eqref{eq: psi rho}.  To see this, note that 
\begin{align}\label{eq: low high equiv}
	\begin{cases}
		1 \asymp_r \br{\xi}^{s} \asymp_r \br{\xi}^{s-1} &\text{for } \xi \in B_{\hat{\Gamma}}(0,r) \\
		\br{\xi} \asymp_r \abs{\xi} \asymp_r \frac{\abs{\xi_1}}{\abs{\xi}} + \abs{\xi} &\text{for } \xi \in \hat{\Gamma} \setminus B_{\hat{\Gamma}}(0,r),
	\end{cases}
 \end{align}
and so this and \eqref{eq: intro rho asym} show that for $\xi \neq 0$ we have the equivalences
\begin{multline}\label{eq: rho to mu}
	\frac{\abs{\rho_\gamma(\xi)}^2}{\abs{\xi}^2}  \br{\xi}^{2s} \asymp \mathbbm{1}_{B_{\hat{\Gamma}}(0,r)}(\xi)\left( \frac{\abs{\xi_1}^2}{\abs{\xi}^2} + \abs{\xi}^2  \right) + \mathbbm{1}_{\hat{\Gamma} \setminus B_{\hat{\Gamma}}(0,r)}(\xi) \br{\xi}^{2s}  = \omega_s(\xi) \\
	\asymp \mathbbm{1}_{B_{\hat{\Gamma}}(0,r)}(\xi)\left(\frac{\abs{\xi_1}}{\abs{\xi}}+\abs{\xi}\right)^2 \br{\xi}^{2(s-1)}    + \mathbbm{1}_{\hat{\Gamma}\setminus B_{\hat{\Gamma}}(0,r)}(\xi) \br{\xi}^2 \br{\xi}^{2(s-1)} \asymp \left(\frac{\abs{\xi_1}}{\abs{\xi}}+\abs{\xi}\right)^2 \br{\xi}^{2(s-1)},
\end{multline}
where the implicit constants depend only on $r,n, \gamma, \sigma, b$. 

The equivalence \eqref{eq: rho to mu} suggests that we can give equivalent definitions of the space by either using $\omega_s$ or the multiplier $\upmu:\R^d \to\R$ defined by
\begin{align}\label{eq:upmu}
    \upmu(\xi)=\begin{cases}
		\f{|\xi_1|}{|\xi|}+|\xi| &\text{for } \xi \neq 0\\
		1 &\text{for } \xi = 0.
	\end{cases} 
\end{align}
We note that defining $\upmu(0) = 1$ is only relevant in the purely toroidal case when $\hat{\Gamma} = \prod_{i=1}^d L_i^{-1} \Z$, and we do so that the restriction of $\upmu$ over $\hat{\Gamma}$ is well-defined and it takes on the same value as $\langle \cdot \rangle^{2s}$ at $\xi = 0$. We then define the specialized Sobolev space $X^s(\Gamma; \F)$ of order $s\ge 0$ in terms of $\upmu$, to be the function space
\begin{align}\label{def: X^s}
    X^s(\Gamma;\F)=
	\begin{cases}
		\tcb{f\in\mathscr{S}'(\Gamma;\C) \mid \hat{f}\in L^1_{\loc}(\hat{\Gamma};\C), \;f= \overline{f},\;\tnorm{\upmu\tbr{\cdot}^{s-1}\hat{f}}_{L^2}<\infty}, & \F = \R \\
		\tcb{f\in\mathscr{S}'(\Gamma;\C) \mid \hat{f}\in L^1_{\loc}(\hat{\Gamma};\C),\;\tnorm{\upmu\tbr{\cdot}^{s-1}\hat{f}}_{L^2}<\infty}, & \F = \C,
	\end{cases}
\end{align}
with the norm associated to $X^s(\Gamma ; \F)$ defined by 
\begin{align}\label{eq: Xs norm defn}
	\norm{f}_{X^s} = \tnorm{\upmu\tbr{\cdot}^{s-1}\hat{f}}_{L^2}.
\end{align}
The definitions of the class of tempered distributions $\mathscr{S}'(\Gamma ; \C)$ on $\Gamma$ and the Fourier transform on $\mathscr{S}'(\Gamma ; \C)$ are contained in Section \ref{tempered} of the appendix. 

First, we show that if a permutation $g \in S_d$ leaves the first factor of $\Gamma$ fixed, then the restriction of the induced map $\sigma_{g}$ on $X^s(\Gamma; \F)$ is an isometric isomorphism between $X^s(\Gamma; \F)$ and $X^s(g \Gamma ; \F)$. The reordered Cartesian product $g\Gamma$ and the map $\sigma_{g}$ are defined in Definition~\ref{defn: reordering notation}.

\begin{lem}\label{lem: X^s permutation}
	Let $\Gamma$ be defined as in \eqref{eq: Gamma f} and let $S_d$ denote the symmetric group of permutations. Suppose $g \in S_d$ satisfies $g(1) = 1$, and let $g\Gamma$ be the reordered Cartesian product defined as in the first item of Definition~\ref{defn: reordering notation}, $\mathcal{P}_{g}: \Gamma \to g\Gamma$ be the reordering map defined as in the second item of Definition~\ref{defn: reordering notation}. Then the map $\sigma_{g}: X^s(\Gamma ; \F) \to  X^s(g\Gamma; \F)$ defined via $\sigma_{g}(f) = f(\mathcal{P}_{g}^{-1}\cdot)$ is an isometric isomorphism, with its inverse given by the map $\sigma_{g,\Gamma}^{-1}: X^s(g \Gamma ; \F) \to  X^s(\Gamma ; \F)$ defined via $\sigma_{g}^{-1}(f) = f(\mathcal{P}_g \cdot)$. 
	\end{lem}
\begin{proof}
We note that $\langle \cdot \rangle = \langle \mathcal{P}_g \cdot \rangle$, and if $g(1) = 1$, then $\upmu(\cdot) = \upmu(\mathcal{P}_g \cdot)$. As a consequence, the Fourier multiplier $\upmu^2(\cdot) \langle \cdot \rangle^{2(s-1)}$ is invariant under $\mathcal{P}_g$, therefore the desired conclusion follows from Lemma~\ref{lem: gen X^s permutation}.
\end{proof}

Next, we characterize the anisotropic space $X^s(\Gamma; \F)$ in relation to the set $R_\Gamma$ defined in the first item of Definition~\ref{defn: RT notation}.

\begin{lem}\label{lem: Xs Hs equiv}
Let $\Gamma$ be defined as in \eqref{eq: Gamma f} with $\dim \Gamma = d \ge 1$ and let $R_\Gamma$ be defined as in \eqref{eq: R}. Let $\Gamma_\ast, \hat{\Gamma}_\ast$ be the canonical reordering of $\Gamma$ and $\hat{\Gamma}$ defined in \eqref{eq: canonical Gamma}, and let $\mathcal{P}_\Gamma: \Gamma \to \Gamma_\ast, \mathcal{P}_{\hat{\Gamma}}: \hat{\Gamma} \to \hat{\Gamma}_\ast$ be the canonical reordering maps defined in \eqref{eq: canonical pGamma}. Then the following hold. 

\begin{enumerate}
	 \item If $R_\Gamma = \varnothing$ or $R_\Gamma = \{1\}$, then $X^s(\Gamma ; \F) = H^s(\Gamma ; \F)$ and $\norm{\cdot}_{X^s}$ and $\norm{\cdot}_{H^s}$ are equivalent norms. In particular, this implies that if $s > d/2$, then $X^s(\Gamma ; \F)$ is an algebra.
	 \item If $1 \in R_\Gamma$ and $\abs{R_\Gamma} \ge 2$, then $X^s(\Gamma ; \F)$ is not closed under rotations in the sense that for any orthogonal matrix $Q \in O(\abs{R_\Gamma})$ is such that $\abs{Q e_1 \cdot e_1 } < 1$, there exists a function $f \in X^s(\Gamma ; \F) \setminus L^2(\Gamma ; \F)$ such that $f(\mathcal{P}_\Gamma^{-1} (Q \oplus I) \mathcal{P}_\Gamma \cdot) \notin X^s(\Gamma ; \F)$, where $Q \oplus I:  \Gamma_{*,R} \oplus \Gamma_{*,T} \to \Gamma_{*,R} \oplus \Gamma_{*,T}$ maps $(x,y)$ to $(Qx,y)$. In particular, this implies that $H^s( \Gamma ; \F) \subsetneq X^s(  \Gamma ; \F)$. 

	 \item If $1 \notin R_\Gamma$ and $R_\Gamma \neq \varnothing$, then $H^s( \Gamma ; \F) \subsetneq X^s(  \Gamma ; \F)$. Furthermore, for all $f \in X^s(\Gamma; \F)$ we have 
	 \begin{align}\label{eq: Xs equiv norm}
		  \norm{f}_{X^s}^2 \asymp \int_{\hat{\Gamma}} \abs{\xi}^2 \br{\xi}^{2(s-1)} \abs{\hat{f}(\xi)}^2 \; d\xi.
	 \end{align} 
\end{enumerate}

\end{lem}
\begin{proof}
We note that since $\upmu(\xi) \le 1 + \abs{\xi}$, we always have $H^s(\Gamma ; \R) \subseteq X^s(\Gamma;\R)$. To prove the first item, we note that if $R_\Gamma = \varnothing$, then we are in the purely toroidal case when $\Gamma = \prod_{i=1}^d L_i \T$. Here, the only low frequency mode $\xi \in B_{\hat{\Gamma}}(0,r)$ is $\xi = 0$. By the definition of $\upmu$ at $\xi = 0$ in \eqref{eq:upmu}, we have $\upmu(0) = 1 = \langle 0 \rangle^{2s}$. So by \eqref{eq: low high equiv} we find that for all $\xi \in \hat{\Gamma}, \upmu(\xi) \asymp \langle \xi \rangle^{2s}$ and the desired conclusion follows. If $R_\Gamma = \{1\}$, then $\Sigma = \R \times \prod_{i=1}^{d-1} L_i \T$. In this case, $\xi \in B_{\hat{\Gamma}}(0,r) \Longleftrightarrow \xi = (\xi_1,0)$ for $\abs{\xi_1} < r$. Then for $\xi \in B_{\hat{\Gamma}}(0,r)$, we have $\upmu(\xi) = 1 + \abs{\xi_1}$ and $\langle \xi \rangle = (1 + \abs{\xi_1}^2)^{1/2}$, which implies that $\upmu(\xi) \asymp_r \langle \xi \rangle \mathbbm{1}_{B_{\hat{\Gamma}}(0,r)} + \langle \xi \rangle^{2s}  \mathbbm{1}_{B^c(0,r)} \asymp \langle \xi \rangle^{2s}$ for all $\xi \in \hat{\Gamma}$, and the desired conclusion follows. 

Next we proceed to prove the second item. Let the set $T_\Gamma$ be defined as in \eqref{eq: T}. The case when $T_\Gamma = \varnothing$ follows from the third item of Proposition 5.2 in \cite{leonitice}, so we assume here that $T_\Gamma \neq \varnothing$.  

Since $\abs{R_\Gamma} \ge 2$, by the third item of Proposition 5.2 in \cite{leonitice}, for every $Q \in O(\abs{R_\Gamma})$ such that $\abs{Q e_1 \cdot e_1} < 1$ there exists a function $G \in (X^s(\R^{\abs{R_\Gamma}}; \F) \cap C^\infty_0(\R^{\abs{R_\Gamma}}; \F) ) \setminus L^2(\R^{\abs{R_\Gamma}}; \F)$ such that $G(Q \cdot ) \notin X^s(\R^{\abs{R_\Gamma}}; \F)$. Let $G$ be such a function, $\mathcal{P}_\Gamma$ be the canonical reordering map defined via \eqref{eq: canonical pGamma}, $\hat{\Gamma}_R$ be the set defined via \eqref{eq: Gamma_R}, and $\uppi_{\hat{\Gamma}}$ be the surjective map on $\hat{\Gamma}$ defined via \eqref{eq: uppi_R}. We then consider the measurable function $F: \hat{\Gamma} \to \C$ defined via 
	\begin{align}
			F(\xi) = \begin{cases}
				\hat{G}( \uppi_{\hat{\Gamma}} \xi), & \xi \in \hat{\Gamma}_R\\
				0, & \text{otherwise},
			\end{cases}
	\end{align}
where $\hat{\Gamma}_R$ is defined as in \eqref{eq: Gamma_R}. We note that by Lemma~\ref{rem: iso everything}, $\uppi_{\hat{\Gamma}}$ is an isometric measure-preserving group isomorphism between $\hat{\Gamma}_{*,R}$ and $\R^{|R_{\Gamma_*}|} = \R^{\abs{R_\Gamma}}$. Then by Fubini's theorem, it follows that 
	\begin{multline}
		 \int_{\hat{\Gamma}} \abs{F(\xi)} \; d\xi = \int_{\hat{\Gamma}_{*,R}} \abs{\hat{G}(\uppi_{\hat{\Gamma}_\ast}\xi)} \; d\xi = \int_{\R^{\abs{R_\Gamma}}} \abs{\hat{G}(\omega)} \; d\omega < \infty, 
		 \; \int_{\hat{\Gamma}} \abs{F(\xi)}^2 \; d\xi = \int_{\R^{\abs{R_\Gamma}}} \abs{\hat{G}(\omega)}^2 \; d\omega = \infty.
	\end{multline}
We also note that the Fourier multiplier $\upmu^2(\cdot) \langle \cdot \rangle^{2(s-1)}$ is invariant under $P_{\hat{\Gamma}}$, and since $\uppi_{\hat{\Gamma}} = \uppi_{\hat{\Gamma}_\ast} \circ \mathcal{P}_{\hat{\Gamma}}$, by the same calculations in \eqref{eq: switching calculation} and by Lemma~\ref{rem: iso everything} we have 
	\begin{multline}
		\int_{\hat{\Gamma}} \upmu^2(\xi) \langle \xi \rangle^{2(s-1)} \abs{F(\xi)}^2 \; d\xi = \int_{\hat{\Gamma}_{*,R}} \upmu^2(\xi) \langle \xi \rangle^{2(s-1)} \abs{\hat{G}(\uppi_{\hat{\Gamma}_\ast}\xi)}^2 \; d\xi = \int_{\R^{\abs{R_\Gamma}}} \upmu(\omega)^2 \langle \omega \rangle^{2(s-1)} \abs{\hat{G}(\omega)}^2 \; d\xi < \infty.
	\end{multline}
Hence, the function $f := \check{F} \in X^s(\Gamma ; \F) \setminus L^2(\Gamma; \F)$. In particular, $f \notin H^s(\Gamma ; \F)$. On the other hand, we have 
	\begin{multline}
		\mathscr{F}^+_{\Gamma}[f(\mathcal{P}_{\Gamma}^{-1} (Q \oplus I) \mathcal{P}_{\Gamma}\cdot)](\xi) = \int_{\Gamma} f(\mathcal{P}_{\Gamma}^{-1} (Q \oplus I) \mathcal{P}_{\Gamma} x) e^{-2\pi i \xi \cdot x} \; dx = \int_{\Gamma_\ast} f(\mathcal{P}_{\Gamma}^{-1} (Q \oplus I) x) e^{-2\pi i \mathcal{P}_{\hat{\Gamma}}\xi \cdot x} \; dx \\
		= \int_{\Gamma_\ast} f(\mathcal{P}_{\Gamma}^{-1} x) e^{-2\pi i ((Q \oplus I) \mathcal{P}_{\hat{\Gamma}} \xi) \cdot x} \; dx = \int_{\Gamma} f(x) e^{-2\pi i (\mathcal{P}_{\hat{\Gamma}}^{-1}(Q \oplus I) \mathcal{P}_{\hat{\Gamma}} \xi) \cdot x} \; dx = F(\mathcal{P}_{\hat{\Gamma}}^{-1} (Q \oplus I) \mathcal{P}_{\hat{\Gamma}} \xi).
	\end{multline}
Then by Lemma~\ref{rem: iso everything} and Fubini's theorem,
	\begin{multline}
		\norm{f(\mathcal{P}_{\Gamma}^{-1} (Q \oplus I) \mathcal{P}_{\Gamma} \cdot)}_{X^s(\Gamma)}^2 = \int_{\hat{\Gamma}} \upmu(\xi) \langle \xi \rangle^{2(s-1)} \abs{F(\mathcal{P}_{\hat{\Gamma}}^{-1}(Q \oplus I)  \mathcal{P}_{\hat{\Gamma}} \xi)}^2 \; d\xi = \int_{\hat{\Gamma}_R} \upmu(\xi) \langle \xi \rangle^{2(s-1)} \abs{\hat{G}(\uppi_{\hat{\Gamma}_\ast} (Q \oplus I) \mathcal{P}_{\hat{\Gamma}}  \xi)}^2 \; d\xi \\
		= \int_{\hat{\Gamma}_{*,R}} \upmu(\xi) \langle \xi \rangle^{2(s-1)} \abs{\hat{G}(\uppi_{\hat{\Gamma}_\ast}(Q \oplus I) \xi)}^2 \; d\xi = \int_{\R^{\abs{R_{\Gamma}}}} \upmu(\omega) \langle \omega \rangle^{2(s-1)} \abs{\hat{G} ( Q \omega)} \; d \omega = \infty.
	\end{multline}
This proves the second item.

To prove the last item, we note that if $1 \notin R_\Gamma$ and $R_\Gamma \neq \varnothing$, then $\Gamma_1 = L_1 \T$ and $\Gamma \neq \prod_{i=1}^d L_i \T$. In this setting, $\xi \in B_{\hat{\Gamma}}(0,r) \Longleftrightarrow \xi \in \Gamma_R$ as $\xi_T = 0$ for the low frequencies and $B_{\hat{\Gamma}}(0,r) \setminus\{0\} \neq \varnothing$. In particular, for $\xi \in B_{\hat{\Gamma}}(0,r)$ we have $\xi_1 = 0$, which implies that $\upmu(\xi) = |\xi_{R_{\hat{\Gamma}}}| = \abs{\xi}$ for $\xi \in B_{\hat{\Gamma}}(0,r) \setminus\{ 0\}$, so by \eqref{eq: low high equiv}, we arrive at \eqref{eq: Xs equiv norm}. By \eqref{eq: Xs equiv norm}, it also follows that $H^s( \Gamma ; \F) \subsetneq X^s(  \Gamma ; \F)$.
\end{proof}

Recall that our ultimate aim in introducing $X^s(\Gamma;\F)$ is to use it as the container space for the free surface function $\eta$ in our study of the traveling wave problem. The upshot of Lemma~\ref{lem: Xs Hs equiv} is that the precise structure of the space $X^s(\Gamma; \F)$ is heavily dependent on the form of the domain $\Gamma$, and in particular, the properties of the set $R_\Gamma$. In the first case considered in the lemma, $X^s(\Gamma;\F)$ is the standard Sobolev space $H^s(\Gamma;\F)$, and therefore we may employ standard Sobolev tools in our subsequent analysis. In the second case, even though $X^s(\Gamma;\F)$ is not a standard Sobolev space, we will be able to prove that it enjoys many of the same properties as $H^s(\Gamma;\F)$. However, in the third case, which includes the physically relevant case when $\Gamma = L_1 \T \times \R$, the space $X^s(\Gamma;\F)$ is unfortunately unusable for our subsequent analysis due to a failure of completeness. To justify the last claim, we prove the following proposition. 

\begin{prop}\label{prop: Xs not complete case 3}
	Let $\Gamma$ be defined as in \eqref{eq: Gamma f} and let $R_\Gamma$ be defined as in \eqref{eq: R}. If $1 \notin R_\Gamma$ and $R_\Gamma \neq \varnothing$, then $X^s(\Gamma; \F)$ is complete if and only if $\abs{R_\Gamma} > 2$. 
\end{prop}
\begin{proof}
The proof is a modification of Proposition 1.34 in \cite{chemin}, which characterizes when the homogeneous Sobolev space $\dot{H}^s(\R^d)$ is complete. We define the measure $\mu_s := \abs{\xi}^2 \langle \xi \rangle^{2(s-1)} \; d\xi$, and denote by $L^2_s(\hat{\Gamma} ; \C)$ the complex valued functions in $L^2(\hat{\Gamma} ; \mu_s)$.  

First assume $\abs{R_\Gamma} > 2$, and suppose $\{u_n\}_{n=1}^{\infty}$ is a Cauchy sequence in $X^s(\Gamma; \F)$. Then by \eqref{eq: Xs equiv norm}, $\{\hat{u}_n\}_{n=1}^{\infty}$ is a Cauchy sequence in $L^2_s(\hat{\Gamma}; \C )$, and hence there exists a function $f \in L^2_s(\hat{\Gamma}; \C )$ such that $u_n \to f$ in $L^2_s(\hat{\Gamma}; \C )$. Recall that by \eqref{eq: low freq r}, we have $\xi \in B_{\hat{\Gamma}}(0,r) \Longleftrightarrow \xi \in \hat{\Gamma}_R$, where $\hat{\Gamma}_R$ is defined in \eqref{eq: Gamma_R}. Using Lemma~\ref{rem: iso everything} and the assumption $\abs{R_\Gamma} > 2$, we have 
\begin{align}\label{eq: f temp}
	 \int_{B_{\hat{\Gamma}}(0,r)} \abs{f(\xi)} \; d\xi \le \left( \int_{B_{\hat{\Gamma}}(0,r)} \abs{\xi}^2 \abs{f(\xi)}^2 \; d\xi \right)^{1/2} \left( \int_{B_{\abs{R_\Gamma}}(0,r)} \abs{\omega}^{-2} \; d\omega \right)^{1/2} < \infty,
\end{align}
where in the last equality we have used the isometric measure-preserving isomorphism $\uppi_{\hat{\Gamma}}$ between $B_{\hat{\Gamma}}(0,r)$ and $B_{\R^{\abs{R_\Gamma}}} (0,r)$. Since $\mathbbm{1}_{\hat{\Gamma} \setminus B_{\hat{\Gamma}}(0,r)} f$ belongs to $H^s(\hat{\Gamma} ; \C)$, with \eqref{eq: f temp} we can infer that $f$ defines a tempered distribution. Then $u := \check{f} \in X^s(\Gamma; \F)$ and $u_n \to u$ in $X^s(\Gamma; \F)$. This shows that $X^s(\Gamma; \F)$ is complete. 

Now assume $\abs{R_\Gamma} \le 2$, and suppose for the sake of contradiction that $X^s(\Gamma; \F)$ is complete with respect to the $X^s$ norm defined via \eqref{eq: Xs norm defn}. Consider a new norm on $X^s(\hat{\Gamma} ; \F)$ given by $\norm{u}_{*} = \norm{u}_{X^s} + \norm{\hat{u}}_{L^1(B_{\hat{\Gamma}}(0,r))}$; this norm is well-defined since the definition of $X^s(\Gamma;\F)$ requires $\hat{u} \in L^1_{\loc}(\hat{\Gamma} ; \C)$ for any function $u \in X^s(\Gamma; \F)$. We then claim that $X^s(\Gamma; \F)$ is also complete endowed with the $\norm{\cdot}_{*}$ norm. Indeed, suppose $\{u_n\}_{n=1}^{\infty}$ is a Cauchy sequence in $(X^s(\Gamma; \F), \norm{\cdot}_{*})$, then $\{u_n\}_{n=1}^{\infty}$ is a Cauchy sequence in $(X^s(\Gamma; \F), \norm{\cdot}_{X^s})$ and $\{\hat{u}_n\}_{n=1}^\infty$ is a Cauchy sequence in $L^1(B_{\hat{\Gamma}}(0,r); \C)$. By the assumed completeness of $(X^s(\Gamma, \F), \norm{\cdot}_{X^s})$, there exists a function $u \in (X^s(\Gamma, \F), \norm{\cdot}_{X^s})$ for which $u_n \to u$ in $(X^s(\Gamma, \F), \norm{\cdot}_{X^s})$. Similarly, there exists a function $g \in L^1(B_{\hat{\Gamma}}(0,r); \C)$ for which $\hat{u}_n \to g$ in $L^1(B_{\hat{\Gamma}}(0,r); \C)$. Clearly, $g = \hat{u}$ a.e. in $L^1(B_{\hat{\Gamma}}(0,r); \C)$, therefore we can conclude that $u_n \to u$ in $(X^s(\hat{\Gamma}; \F), \norm{\cdot}_{*})$. This completes the proof of the claim.

We now know that $X^s(\hat{\Gamma}; \F)$ is complete with respect to both $\norm{\cdot}_{X^s}$ and $\norm{\cdot}_{*}$. The identity map $I : (X^s(\hat{\Gamma}; \F), \norm{\cdot}_{*} )$ $\to  (X^s(\hat{\Gamma}; \F), \norm{\cdot}_{X^s} )$ is trivially continuous, thus we can invoke the bounded inverse theorem to deduce the existence of a universal constant $C > 0$ such that $\norm{u}_{*} \le C \norm{u}_{X^s}$ for all functions $u \in X^s(\Gamma;\F)$. In turn, this implies that 
\begin{align}\label{eq: bad ineq}
	\norm{\hat{u}}_{L^1(B_{\hat{\Gamma}}(0,r))} \le C \norm{u}_{X^s} \; \text{for all} \;  u \in X^s(\Gamma;\F).
\end{align}

To derive a contradiction, we construct an explicit function $f \in X^s(\Gamma; \F)$ for which \eqref{eq: bad ineq} is violated. For any $\xi \in \hat{\Gamma}$, we adopt the convention in \eqref{eq: xi R} and write $\xi = \xi_{R_{\hat{\Gamma}}} + \xi_{T_{\hat{\Gamma}}}$. Let $\mathcal{C} = \{ \xi \in \hat{\Gamma} \mid \xi_{T_{\hat{\Gamma}}} = 0 \; \text{and} \; r/2 < |\xi_{R_{\hat{\Gamma}}}| < r \}$, which in particular implies that $2^{-i} \mathcal{C} \cap 2^{-j} \mathcal{C} = \varnothing$ for any $i,j \in \N$ such that $i \neq j$. Now for every $n\ge 1$ we then consider $G_n \in L^1(\hat{\Gamma}; \C)$ defined via 
	\begin{align}
		 G_n(\xi) = \sum_{q=1}^{n} \frac{2^{q(1+\abs{R_\Gamma}/2)}}{q} \mathbbm{1}_{2^{-q} \mathcal{C}} ( \xi).
	\end{align}
	We note in particular that $\overline{G_n(\xi)} = G_n(-\xi)$ and $\supp G_n \subseteq B_{\hat{\Gamma}}(0,r) \subset \hat{\Gamma}$, and so we may define the smooth and bounded function $g_n : \Gamma_\ast \to \R$ via $g_n = \mathscr{F}^{-1}[G_n]$. Let $\mathscr{H}$ denote the Hausdorff measure over $\hat{\Gamma}$. Now we may readily calculate 
	\begin{align}\label{eq: ghat blows up}
		 \norm{\hat{g}_n}_{L^1(B_{\hat{\Gamma}}(0,r))} =  \sum_{q=1}^{n} \frac{2^{q(1+\abs{R_\Gamma}/2)}}{q} 2^{-q \abs{R_\Gamma}} \mathscr{H}^{\abs{R_\Gamma}}(\mathcal{C} ) = \mathscr{H}^{\abs{R_\Gamma}}(\mathcal{C} )  \sum_{q=1}^{n} \frac{2^{q(1-\abs{R_\Gamma}/2)}}{q} \to \infty \; \text{as} \; n \to \infty, 
	\end{align}
since $\abs{R_\Gamma} \le 2$. On the other hand, for any $n \ge 1$ we have 
	\begin{multline}
		 \norm{g_n}_{X^s}^2 \asymp \int_{B_{\hat{\Gamma}}(0,r)} \abs{\xi}^{2} \abs{G_n(\xi)}^2 \; d\xi =  \sum_{q=1}^n \int_{2^{-q} \mathcal{C}} \abs{\omega}^{2} \left( \frac{2^{q(1+\abs{R_\Gamma}/2)}}{q} \right)^2 d \omega \\
		 \asymp \mathscr{H}^{\abs{R_\Gamma}}(\mathcal{C}) \sum_{q=1}^n 2^{-2q} 2^{-q\abs{R_\Gamma}} \left( \frac{2^{q(1+\abs{R_\Gamma}/2)}}{q} \right)^2 = \mathscr{H}^{\abs{R_\Gamma}}(\mathcal{C}) \sum_{q=1}^n \frac{1}{q^2}. 
	\end{multline}
	Therefore, $\sup_{n \ge 1} \norm{g_n}_{X^s} < \infty $ and by \eqref{eq: bad ineq} and \eqref{eq: ghat blows up}, we arrive at a contradiction. Thus, $X^s(\Gamma;\F)$ cannot be complete for $\abs{R_\Gamma} \le 2$. 
\end{proof}

We now proceed to study the space $X^s(\Gamma ; \F)$ in the second case of Lemma~\ref{lem: Xs Hs equiv}. We begin by stating a preliminary result. 

\begin{lem}\label{lem: fhat L1}
	Let $\Gamma$ be defined as in \eqref{eq: Gamma f} and $r$ be defined as in \eqref{eq: low freq r}. Suppose $\R \ni s \ge 0$, $1 \in R_\Gamma$, and $\abs{R_\Gamma} \ge 2$. The following hold.
\begin{enumerate}
	\item For $\upmu: \hat{\Gamma} \to \R$ defined in \eqref{eq:upmu}, we have 
	\begin{align}\label{eq: m reciprocal is in L2 loc near zero}
		 \int_{B_{\hat{\Gamma}}(0,r)} \frac{1}{\upmu^2(\xi)} \; d\xi < \infty. 
	\end{align}
	 \item For $f \in X^s(\Gamma; \F)$ we have the estimate 
	 \begin{align}\label{eq: function}
		  \int_{B_{\hat{\Gamma}}(0,r)} \abs{\hat{f}(\xi) } d \xi + \left(  \int_{\hat{\Gamma} \setminus B_{\hat{\Gamma}}(0,r)} (1+\abs{\xi}^2)^s \abs{\hat{f}(\xi) } \right)^{1/2} \lesssim_{d,s} \norm{f}_{X^s}. 
	 \end{align}
	 In particular, if $s > d/2$ then 
	 \begin{align}\label{eq: X^s super}
		  \Vert\hat{f}\Vert_{L^1(\hat{\Gamma})} \lesssim_{d, s} \norm{f}_{X^s}.  
	 \end{align} 
\end{enumerate}
\end{lem}
\begin{proof}
For the first item, we note that by Lemma~\ref{rem: iso everything}, 
\begin{align}
	\int_{B_{\hat{\Gamma}}(0,r)} \frac{1}{\upmu^2(\xi)} \; d\xi \asymp \int_{B_{\hat{\Gamma}}(0,r)} \frac{\abs{\xi}^2}{\xi_1^2 + \abs{\xi}^4} \; d\xi = \int_{B_{\hat{\Gamma}}(0,r)}  \frac{\abs{\xi_{R_{\hat{\Gamma}}}}^2}{\xi_1^2 + \abs{\xi_{R_{\hat{\Gamma}}}}^4} \; d\xi = \int_{B_{\R^{|R_{\hat{\Gamma}}|}}(0,r)}  \frac{\abs{\omega}^2}{\omega_1^2 + \abs{\omega}^4} \; d\omega,
\end{align}
where in the last equality we used the isometric measure-preserving group isomorphism $\uppi_{\hat{\Gamma}}$ between $B_{\hat{\Gamma}}(0,r)$ and $B_{\R^{|R_{\hat{\Gamma}}|}}(0,r)$. The latter integral can readily be verified to be finite; see for instance, Proposition 5.2 of \cite{leonitice}. 

For the second item, we note that by \eqref{eq: low high equiv}, we have
\begin{align}\label{eq: Xs equiv mu and bracket}
	 \norm{f}_{X^s}^2 \asymp \int_{B_{\hat{\Gamma}}(0,r)} \upmu^2(\xi) \abs{\hat{f}(\xi)}^2 \; d\xi + \int_{\hat{\Gamma} \setminus B_{\hat{\Gamma}}(0,r)} \langle \xi \rangle^{2s} \abs{\hat{f}(\xi)}^2 \; d\xi.
\end{align}
Then by Cauchy-Schwarz, \eqref{eq: m reciprocal is in L2 loc near zero}, and \eqref{eq: Xs equiv mu and bracket} we have 
\begin{align}\label{eq: hat f low frequency}
	 \int_{B_{\hat{\Gamma}}(0,r)} \abs{\hat{f}(\xi)} \; d\xi \le \left( \int_{B_{\hat{\Gamma}}(0,r)} \frac{1}{\upmu^2(\xi)} \right)^{1/2} \left(  \int_{B_{\hat{\Gamma}}(0,r)}\upmu^2(\xi) \abs{\hat{f}(\xi)}^2 \; d\xi \right)^{1/2} \lesssim \norm{f}_{X^s}.  
\end{align}
Then \eqref{eq: function} follows immediately from \eqref{eq: Xs equiv mu and bracket} and \eqref{eq: hat f low frequency}.

\end{proof}

The next theorem records the fundamental completeness and embedding properties of the space $X^s(\Gamma; \F)$ in the first and second cases considered in Lemma~\ref{lem: Xs Hs equiv}. 

\begin{thm}\label{thm: X^s gen}
	Suppose $\R \ni s \ge 0$. Let $\Gamma$ be defined as in \eqref{eq: Gamma f} and let $R_\Gamma$ be defined as in \eqref{eq: R}. Assume $R_\Gamma = \varnothing, R_\Gamma = \{1\}$, or $1 \in R_\Gamma$ and $\abs{R_\Gamma} \ge 2$. Then the following hold.
	\begin{enumerate}
		\item $X^s(\Gamma; \F)$ is a Hilbert space.  
		
		\item The Schwartz space $\mathscr{S}(\Gamma ; \F)$, as defined in Definition~\ref{defn:schwartz}, is dense in $X^s(\Gamma ; \F)$.  
		
		\item If $t \in \R$ and $s < t$, then we have the continuous inclusion $X^t(\Gamma; \F) \hookrightarrow X^s(\Gamma\; \F)$.

		\item We have the continuous inclusion $H^s(\Gamma; \F) \hookrightarrow X^s(\Gamma ;\F)$. 
		
		\item If $s \ge 1$, then there exists a constant $c > 0$ depending on $d,s$, and in the toroidal cases on $L_i$, such that 
		\begin{align}\label{eq: grad Xs}
			 \norm{\nabla f}_{H^{s-1}} \le c \norm{f}_{X^s}.  
		\end{align}
		In particular, this implies that the map $\nabla: X^s(\Gamma; \F) \to H^{s-1}(\Gamma; \F^d)$ is continuous. 
	\end{enumerate}
\end{thm}
\begin{proof}
	In the case when $R_\Gamma = \varnothing$ or $R_\Gamma = \{1\}$, by the first item of Lemma~\ref{lem: Xs Hs equiv}, all five items follow from standard Sobolev theory as $X^s(\Gamma; \F)$ is the standard Sobolev space $H^s(\Gamma; \F)$.

	Next suppose that $1 \in R_\Gamma$ and $\abs{R_\Gamma} \ge 2$. If a sequence $\{f_n\}_{n=1}^\infty \subset X^s(\Gamma;\F)$ is Cauchy, then there exists $F \in L^2_s(\Gamma; \F)$ for which $\hat{f}_n \to F$ in $L^2_s(\Gamma;\F)$ as $n \to \infty$. The second item of Lemma~\ref{lem: fhat L1} guarantees that $F \subseteq L^1(B_{\hat{\Gamma}}(0,r); \C) + L^2(\hat{\Gamma} \setminus B_{\hat{\Gamma}}(0,r) ; \C)$, therefore $f := \check{F} \in X^s(\Gamma; \F)$ is well-defined and it is easy to verify that $f$ is real-valued in the case when $\F = \R$, as realness is preserved in the limit. This implies $f_n \to f$ in $X^s(\Gamma;\F)$ as $n \to \infty$, and it follows then that $X^s(\Gamma ; \F)$ is complete. This proves the first item. Following the arguments of Theorem 5.6 of \cite{leonitice}, the first item in turn implies the other fundamental properties listed in the second to fifth items.  
\end{proof}

We conclude this subsection by summarizing some additional properties of the space $X^s(\Sigma; \R)$, where $\Sigma$ is defined as in \eqref{sigma}.

\begin{thm}\label{thm:Xs}
	Let $\R \ni s \ge 0$ and let $\Sigma$ be defined as in \eqref{sigma}. Assume $R_\Sigma = \varnothing, R_\Sigma = \{1\}$, or $1 \in R_\Sigma$ and $\abs{R_\Sigma} \ge 2$. Then following hold.
	\begin{enumerate}
		\item (Low-high frequency decomposition) For every $f \in X^s(\Sigma ; \R)$ and $t > 0$, we can write $f=f_{l,t}+f_{h,t}$, where $f_{l,t}=\mathscr{F}^{-1} [\mathbbm{1}_{B(0,t)}\mathscr{F}[f] ] \in C^\infty_0(\Sigma ; \C)$ and $f_{h,t}=\mathscr{F}^{-1} [\mathbbm{1}_{\hat{\Gamma}\setminus B(0,t)}\mathscr{F}[f] ]\in H^s(\Sigma ; \C)$. Furthermore, we have the estimates
		\begin{align}\label{eq: low high Xs estimates}
			 \norm{f_{l,t}}_{C^k_b} = \sum_{\abs{\alpha} \le k} \norm{\p^\alpha f_{l,R}}_{L^\infty} \lesssim \norm{f_{l,R}}_{X^s} \; \text{and} \; \norm{f_{h,t}}_{H^s} \lesssim \norm{f_{h,t}}_{X^s}.     
		\end{align}
		\item  (Supercritical specialized Sobolev times standard Sobolev is standard Sobolev) If $s > d/2$, then for any $f \in X^s(\Gamma; \F), g \in H^s(\Sigma; \F)$ we have $fg \in H^s(\Sigma ; \F)$, and there exists a constant $c = c(d,s) > 0$ for which 
		\begin{align}\label{eq: supercrit XsHs prod}
			 \norm{fg}_{H^s} \le c \norm{f}_{X^s} \norm{g}_{H^s} \; \text{for all } f \in X^s(\Sigma; \F) \;\text{and} \; g \in H^s(\Sigma; \F).
		\end{align}
		\item ($e_1$-derivatives of specialized Sobolev are $\dot{H}^{-1}$ bounded) If $s \ge 1$, then there exists a constant $c = c(d,s) > 0$ such that 
		\begin{align}
			 [\p_1 \eta]_{\dot{H}^{-1}} \le c \norm{f}_{X^s} \; \text{for all } f \in X^s(\Sigma ; \F).
		\end{align}
		This implies that the map $\p_1 : X^{s}(\Sigma; \F) \to \dot{H}^{-1}(\Sigma; \F) \cap H^{s-1}(\Sigma ; \F)$ is continuous and injective.
	\end{enumerate}
\end{thm}

\begin{proof}
If $R_\Sigma = \varnothing$ or $R_\Sigma = \{ 1 \}$, then by the first item of Lemma~\ref{lem: Xs Hs equiv}, $X^s(\Sigma; \R)$ is the standard Sobolev space $H^s(\Sigma;\R)$. All three items then follow from standard results on $H^s(\Sigma ; \R)$. Suppose $1 \in R_\Sigma$ and $\abs{R_\Sigma} \ge 2$. We note that by Lemma~\ref{rem: iso everything} and Remark~\ref{rem:zero fourier}, the calculations performed over $\R^d$ in \cite{leonitice} are also valid over the low frequencies belonging to general $\hat{\Gamma}$, therefore all three items follow from minimal modifications of Theorem 5.5, the last item of Theorem 5.6, and Theorem 5.12 of \cite{leonitice}.
\end{proof}

\subsection{The anisotropic space \texorpdfstring{$X^s(\Sigma; \F)$}{Xs} as an algebra} \label{sec: X^s algebra}

Let $\Sigma$ be defined as in \eqref{sigma}. In this subsection we prove that $X^s(\Sigma; \F)$ is an algebra for $s > \frac{d}{2}$.  To prove this we will first adapt the anisotropic Littlewood-Paley techniques used in \cite{benoit} to prove that $X^s(\R^d ; \C)$ is an algebra.  This special case turns out to be sufficient for deducing the result in the general case when $\R^d$ is replaced by $\Sigma$ and $\C$ is replaced by $\R$.

First, recall that the multiplier $\upmu$ defined in \eqref{eq:upmu} satisfies $1/\upmu \in L^2(B(0,1))$ by the first item of Lemma~\ref{lem: fhat L1}. Next, we consider the functional $I: (L^0(\R^d; [0, \infty]))^3 \to [0,\infty]$ defined via
\begin{align}\label{eq: I}
	I[F,G,H]=\int_{B(0,1)^2}\f{\upmu(\xi+\eta)}{\upmu(\xi)\upmu(\eta)}F(\xi)G(\eta)H(\xi+\eta)\;d\xi\;d\eta,
\end{align}
where $L^0(\R^d; [0, \infty])$ denote the non-negative measurable functions on $\R^d$. Our goal is to use the same formula \eqref{eq: I} to define a trilinear functional over $(L^2(\R^d; \C))^3$, but for now we only define $I$ over non-negative measurable functions so that $I$ is clearly well-defined.

The next lemma shows that $I$ induces a bounded trilinear functional over $(L^2(\R^d; \C))^3$ into $\C$ as long as $I$ is bounded over $(L^0(\R^d; [0,\infty]))^3$ into $\R$.

\begin{lem}\label{lem: I over nonnegative}
Suppose there exists a constant $c > 0$ such that
\begin{align}\label{eq: I est over nonneg real}
	I[F,G,H] \le c \norm{F}_{L^2} \norm{G}_{L^2} \norm{H}_{L^2} \; \text{for all} \; F,G,H \in L^2(\R^d;[0,\infty]),
\end{align}
where $I$ is the functional defined in \eqref{eq: I}. Then $I$ induces a bounded trilinear map defined over $(L^2(\R^d; \C))^3$ into $\C$ via the same formula in \eqref{eq: I}. Furthermore, there exists a constant $C > 0$ such that 
\begin{align}\label{eq: I est over C}
	\abs{I[F,G,H]} \le C \norm{F}_{L^2} \norm{G}_{L^2} \norm{H}_{L^2}  \; \text{for all} \;  F,G,H \in L^2(\R^d;\C). 
\end{align}
\end{lem}
\begin{proof}
We first consider the case when $F,G,H \in (L^2(\R^d ; \R))^3$. Then by decomposing $F = F_0 - F_1$ where $F_0 = F^+$ and $F_1 = F^-$ are the positive and negative parts of $F$, and similarly for $G,H$, we have 
\begin{align}
	 FGH = \sum_{j,k,l \in \{0,1\} } (-1)^{j+k+l}F_j G_k H_l.
\end{align}
By linearity, 
\begin{align}\label{eq: I decomp pos neg}
	 I[F,G,H] = \sum_{j,k,l \in \{0,1\} } (-1)^{j+k+l} I[F_j, G_k, H_l ].
\end{align}
By assumption, $\abs{I[F_j, G_k, H_l ]} \le c \norm{F_j}_{L^2} \norm{G_k}_{L^2} \norm{H_l}_{L^2}$ for all $j,k,l \in \{0,1\}$, thus by \eqref{eq: I decomp pos neg} $I$ is bounded over $F,G,H \in (L^2(\R^d;\R))^3$ and $\abs{I[F,G,H]} \lesssim \norm{F}_{L^2} \norm{G}_{L^2} \norm{H}_{L^2}$ for all $F,G,H \in L^2(\R^d;\R)$. 

In the general case when $F,G,H \in (L^2(\R^d ; \C))^3$, we write $F = F_0 + i F_1$ where $F_0, F_1 \in L^2(\R^d; \R)$, and similarly for $G,H$. Then 
\begin{align}
	 FGH = \sum_{j,k,l \in \{0,1\} } (i)^{j+k+l} F_j G_k H_l,
\end{align}
and by linearity again we have 
\begin{align}
	 I[F,G,H] = \sum_{j,k,l \in \{0,1\} } (i)^{j+k+l} I[F_j, G_k, H_l].
\end{align}
By the first case and following the same line of reasoning, $I$ is bounded over $(L^2(\R^d;\C))^3$ into $\C$ and there exists a constant $C > 0$ for which \eqref{eq: I est over C} holds. 
\end{proof}

Next, we claim that supercritical specialized Sobolev space is an algebra if and only if the functional defined via  \eqref{eq: I} is bounded over $(L^2(\R^d; [0,\infty]))^3$ into $[0,\infty]$. 
\begin{prop}\label{proposition on reduction to boundedness of a trilinear functional} Assume $s>d/2$. Then $X^s(\R^d;\C)$ is an algebra if and only if for the mapping $I$ defined via \eqref{eq: I}, there exists a constant $c > 0$ such that for all $F,G,H \in L^2(\R^d; [0,\infty])$ we have the estimate $I[F,G,H]\le c \norm{F}_{L^2}\norm{G}_{L^2}\norm{H}_{L^2}$.
\end{prop}
\begin{proof}
Assume first that $I$ is bounded over $(L^2(\R^d; [0,\infty]))^3$. By Lemma~\ref{lem: I over nonnegative}, $I$ induces a bounded trilinear functional over $(L^2(\R^d; \C))^3$ into $\C$ via the same formula \eqref{eq: I}, and there exists a constant $C > 0$ such that 
\begin{align}\label{eq: I bounded over L2 C}
	\abs{I[F,G,H]} \le C \norm{F}_{L^2} \norm{G}_{L^2} \norm{H}_{L^2},
\end{align} 
for all $F,G,H \in L^2(\R^d;\C)$. 

Now suppose that $f,g\in X^s(\R^d ; \C)$. By the second item in Theorem~\ref{thm:Xs}, we may write $f=f_0+f_1$ where $f_0=\mathscr{F}^{-1} [\mathbbm{1}_{B(0,1)}\mathscr{F}[f] ] \in  C^\infty_0(\R^d ; \C)$ and $f_1=\mathscr{F}^{-1} [\mathbbm{1}_{\R^d\setminus B(0,1)}\mathscr{F}[f] ]\in H^s(\R^d ; \C)$, and similarly for $g$. Then we have the decomposition
\begin{align}
fg = f_0 g_0 + f_0 g_1 + f_1 g_0 + f_1 g_1.
\end{align}
If $1\le i+j$, by the fact that we are in the supercritical regime $s > d/2$, $f_1, g_1 \in H^s(\R^d ; \C)$, \eqref{eq: low high Xs estimates}, and \eqref{eq: supercrit XsHs prod} we have $f_i g_j \in H^s$ with the estimate 
\begin{align}\label{eq: high low/high Xs est}
	\norm{f_ig_j}_{H^s} \lesssim\norm{f}_{X^s}\norm{g}_{X^s}.
\end{align}
Thus, it remains to use the boundedness of $I$ to understand the product $f_0 g_0$. Note that by the first item of Theorem~\ref{thm:Xs}, we have the inclusions $\hat{f}_0, \hat{g}_0 \in L^1(\R^d;\C)$, and hence Young's inequality implies that $\hat{f}_0 \ast \hat{g}_0 \in L^1(\R^d;\C)$; also, $\supp(\hat{f}_0 \ast \hat{g}_0) \subseteq B(0,2)$.

Now let $\varphi\in L^2(\R^d;\C) \cap L^\infty(\R^d;\C)$.  Since $\hat{f}_0$ and $\hat{g}_0$ are supported in $B(0,1)$ and $\upmu$ is locally bounded, we may employ Tonelli's theorem and a change of variables to see that 
\begin{multline}
\int_{\R^d}\upmu(\hat{f_0}\ast \hat{g_0})\varphi=\int_{\R^d}\int_{\R^d}\upmu(\xi+\eta)\hat{f_0}(\xi)\hat{g_0}(\eta)\varphi(\xi+\eta)\;d\xi\;d\eta\\=\int_{B(0,1)^2}\upmu(\xi+\eta)\hat{f_0}(\xi)\hat{g_0}(\eta)\varphi(\xi+\eta)\;d\xi\;d\eta=I[\upmu \hat{f_0},\upmu \hat{g_0},\varphi].
\end{multline}
Therefore, by \eqref{eq: I bounded over L2 C} we have 
\begin{align}
\babs{\int_{\R^d}\upmu(\hat{f_0}\ast \hat{g_0})\varphi}\le C\tnorm{\upmu \hat{f_0}}_{L^2}\tnorm{\upmu \hat{g_0}}_{L^2}\tnorm{\varphi}_{L^2},
\end{align}
but by the density of $L^2(\R^d;\C) \cap L^\infty(\R^d;\C)$ in $L^2(\R^d;\C)$, the left side of this expression extends to define a bounded linear functional on $L^2(\R^d;\C)$ obeying the same estimate.  
Upon invoking Riesz's representation theorem, we deduce that $\upmu(\hat{f_0}\ast \hat{g_0})\in L^2(\R^d;\C)$ with
\begin{align}\label{eq: low/low Xs est}
\tnorm{\upmu(\hat{f_0}\ast \hat{g_0})}_{L^2}\le C\tnorm{\upmu \hat{f_0}}_{L^2}\tnorm{\upmu \hat{g_0}}_{L^2}\le C\norm{f}_{X^s}\norm{g}_{X^s}.
\end{align}
As $\m{supp}(\hat{f_0}\ast \hat{g_0})\subseteq B(0,2)$, we have that $\norm{\tbr{\cdot}^{s-1}\upmu(\hat{f_0}\ast \hat{g_0})}_{L^2}  \le 5^{\f{s-1}{2}}\norm{\upmu(\hat{f_0}\ast \hat{g_0})}_{L^2}$. Thus, using  \eqref{eq: high low/high Xs est}, \eqref{eq: low/low Xs est}, and the fourth item of Theorem \ref{thm:Xs}  we find that
\begin{multline}\label{eq: Xs product bound}
\norm{fg}_{X^s} \le \sum_{0 \le i + j \le 2} \norm{f_i g_j}_{X^s}   
= \tnorm{\upmu\tbr{\cdot}^{s-1}(\hat{f_0}\ast \hat{g_0})}_{L^2}+ 
\sum_{1\le i+j\le2} \norm{f_i g_j}_{X^s} \\
\lesssim
\tnorm{\upmu(\hat{f_0}\ast \hat{g_0})}_{L^2}+\sum_{1\le i+j\le2} \norm{f_i g_j}_{H^s}
\lesssim\norm{f}_{X^s}\norm{g}_{X^s}.
\end{multline}
Thus, $X^s(\R^d;\C)$ is an algebra.

Conversely, assume that $X^s(\R^d;\C)$ is an algebra. Let $F,G,H\in L^2(\R^d;[0,\infty])$ and observe first that $I[F,G,H]=I[F\mathbbm{1}_{B(0,1)}, G\mathbbm{1}_{B(0,1)},H]$ and then that $\mathscr{F}^{-1}[F\mathbbm{1}_{B(0,1)}/\upmu],\mathscr{F}^{-1}[G\mathbbm{1}_{B(0,1)}/\upmu]\in X^s(\R^d;\C)$. Therefore, by Cauchy-Schwarz and the boundedness of products in $X^s$ we have
\begin{multline}\label{eq: Xs implies I bounded}
I[F,G,H]=I[F\mathbbm{1}_{B(0,1)}, G\mathbbm{1}_{B(0,1)},H]=\int_{\R^d}\upmu\sp{(F\mathbbm{1}_{B(0,1)}/\upmu)\ast (G\mathbbm{1}_{B(0,1)}/\upmu)}H\\
\le\tnorm{\upmu\tp{(F\mathbbm{1}_{B(0,1)}/\upmu)\ast (G\mathbbm{1}_{B(0,1)}/\upmu)}}_{L^2}\norm{H}_{L^2}=\tnorm{\mathscr{F}^{-1}[F\mathbbm{1}_{B(0,1)}/\upmu]\mathscr{F}^{-1}[G\mathbbm{1}_{B(0,1)}/\upmu]}_{X^s}\norm{H}_{L^2}\\
\lesssim\tnorm{\mathscr{F}^{-1}[F\mathbbm{1}_{B(0,1)}/\upmu]}_{X^s}\tnorm{\mathscr{F}^{-1}[G\mathbbm{1}_{B(0,1)}/\upmu]}_{X^s}\norm{H}_{L^2}\lesssim\norm{F}_{L^2}\norm{G}_{L^2}\norm{H}_{L^2}.
\end{multline}
Hence, $I$ is bounded  over $(L^2(\R^d ; [0,\infty]))^3$ and the proof is complete.
\end{proof}

Thus, by Lemma~\ref{lem: I over nonnegative} and Proposition~\ref{proposition on reduction to boundedness of a trilinear functional}, to prove that $X^s(\R^d ; \C)$ is an algebra it remains to show that the functional $I$ defined over $(L^0(\R^d;[0,\infty]))^3$ via \eqref{eq: I} satisfies \eqref{eq: I est over nonneg real}. For the rest of this subsection we assume that $F,G,H \in L^2(\R^d; [0,\infty])$, and we now estimate the $L^2$-boundedness of the operator $I$. 

First, we introduce a decomposition of the frequency space $\R^d \times \R^d$. 

\begin{defn}[Squared frequency space and functional decomposition]\label{definition of squared frequency space decomposition}
We write $I$ as the sum of two operators $I=I_0+I_1$, where each $I_i$ is accounting for the contribution from a special portion of squared frequency space.
\begin{enumerate}
\item We identify the following `good' and `bad' sets. First we define the `good' set $E_0$ via 
\begin{align}\label{eq:good E}
E_0=\tcb{(\xi,\eta)\in B(0,1)^2\;:\;|\xi|+|\eta|\le3||\xi|-|\eta||}.
\end{align}
Next we define the `bad' set $E_1$ via
\begin{align}\label{eq:bad E}
E_1=\tcb{(\xi,\eta)\in B(0,1)^2\;:\;|\xi|+|\eta|>3||\xi|-|\eta||}.
\end{align}
Note that $B(0,1)^2=E_0\cup E_1$.
\item For $i \in \{0,1\}$, we define the functionals $I_i : (L^0(\R^d; [0,\infty]))^3 \to [0,\infty]$ via 
\begin{align}\label{eq: Is}
I_i[F,G,H]=\int_{E_i}\f{\upmu(\xi+\eta)}{\upmu(\xi)\upmu(\eta)}F(\xi)G(\eta)H(\xi+\eta)\;d\xi\;d\eta. 
\end{align}
Clearly, $I_i$ is well-defined, and we have the identity $I=I_0+I_1$.
\end{enumerate}
\end{defn}
Next, we analyze the set $E_1$ as defined in \eqref{eq:bad E}.

\begin{lem}\label{lemma on equivalence in E1} 
The inclusion $(\xi,\eta)\in E_1$ is equivalent to the bounds $\f12|\eta| < |\xi| < 2|\eta|$, and for $(\xi,\eta) \in E_1$ we have the estimate $|\xi+\eta| < 3|\xi|$.
\end{lem}
	\begin{proof}
	The second bound follows from the equivalence and the triangle inequality, so it suffices to prove the equivalence. For this we note that 
	\begin{align}
	(\xi,\eta) \in E_1 \Longleftrightarrow	\abs{\xi} + \abs{\eta} > 3\abs{\abs{\xi}-\abs{\eta}} \Longleftrightarrow \abs{\xi}^2 - \frac{5}{2} \abs{\xi} \abs{\eta} + \abs{\eta}^2 < 0 \Longleftrightarrow \f12|\eta| < |\xi| < 2|\eta|  .
	\end{align}
	\end{proof}

	By Lemma~\ref{lemma on equivalence in E1}, one can get a good geometric understanding of $E_0$ and $E_1$ by examining the plot in $(\abs{\xi},\abs{\eta})$ space given in Figure \ref{fig_1}.

\begin{figure}[h]
	\centering
	\includegraphics[width=0.3\textwidth]{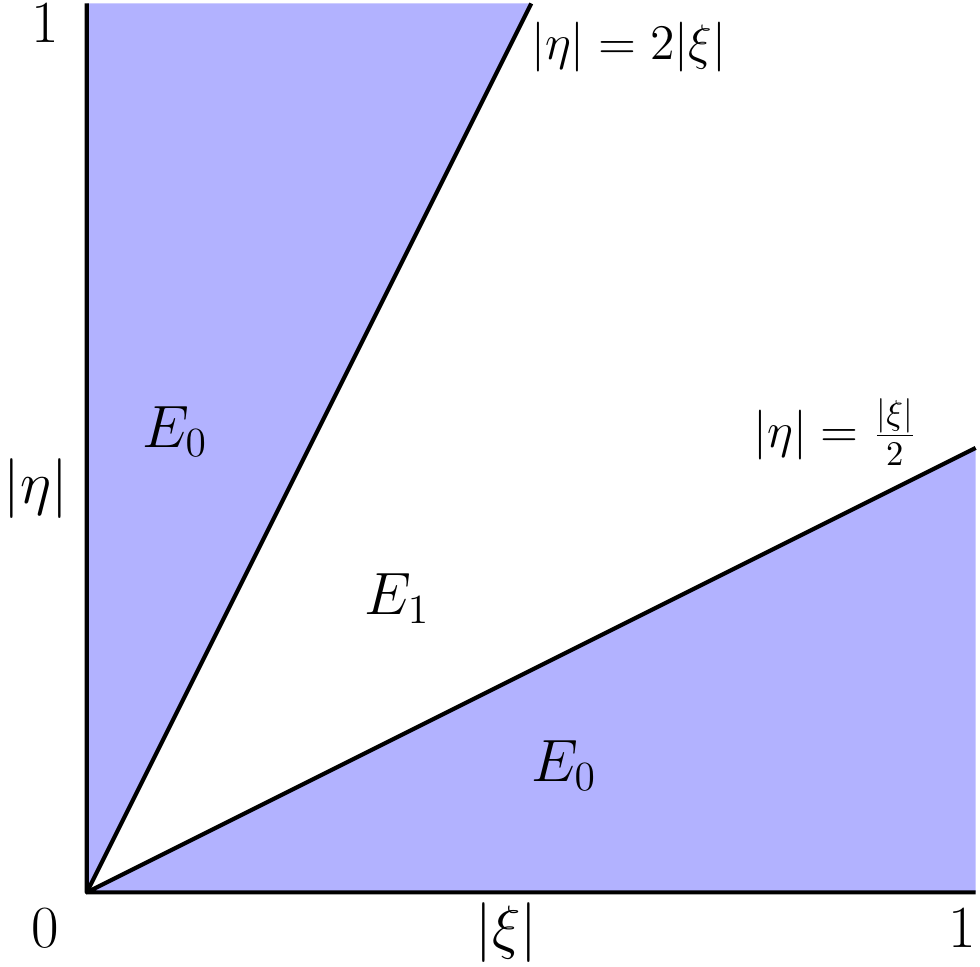} 
	\caption{A schematic breakdown of the sets $E_0$ and $E_1$}\label{fig_1}
\end{figure} 
The region $E_0$ is dubbed the good region because of the following lemma.
\begin{lem}[Multiplier subadditivity in $E_0$]\label{lemma on multiplier subadditivty}
If $(\xi,\eta)\in E_0$, then $\upmu(\xi+\eta)\le 3(\upmu(\xi)+\upmu(\eta))$.
\end{lem}
\begin{proof}
If $|\xi|+|\eta|\le 3||\xi|-|\eta||$, then the reverse triangle inequality $|\xi+\eta|\ge||\xi|-|\eta||$ allows us to deduce
\begin{align}
\upmu(\xi+\eta)=\f{|\xi_1+\eta_1|}{|\xi+\eta|}+|\xi+\eta|\le\f{|\xi_1|+|\eta_1|}{||\xi|-|\eta||}+|\xi|+|\eta|
\le 3\f{|\xi_1|+|\eta_1|}{|\xi|+|\eta|}+|\xi|+|\eta|\le 3(\upmu(\xi)+\upmu(\eta)).
\end{align}
\end{proof}
This gives us the boundedness of $I_0$.

\begin{prop}\label{proposition on boundedness of I0}
For all $F,G,H\in L^2(\R^d;[0,\infty))$ we have the estimate
\begin{align}
I_0[F,G,H]\le6\norm{1/\upmu}_{L^2(B(0,1))}\norm{F}_{L^2}\norm{G}_{L^2}\norm{H}_{L^2}.
\end{align}
\end{prop}
\begin{proof}
By applying Lemma~\ref{lemma on multiplier subadditivty}, Tonelli's theorem, and Cauchy-Schwarz, we find
\begin{multline}
I_0[F,G,H]\le3\int_{E_0}\f{1}{\upmu(\xi)}F(\xi)G(\eta)H(\xi+\eta)\;d\xi\;d\eta+3\int_{E_0}\f{1}{\upmu(\eta)}F(\xi)G(\eta)H(\xi+\eta)\;d\xi\;d\eta\\\le3\int_{B(0,1)^2}\f{1}{\upmu(\xi)}F(\xi)G(\eta)H(\xi+\eta)\;d\xi\;d\eta + 3\int_{B(0,1)^2}\f{1}{\upmu(\eta)}F(\xi)G(\eta)H(\xi+\eta)\;d\xi\;d\eta\\
\le 3\norm{H}_{L^2}\bp{\norm{G}_{L^2}\int_{B(0,1)}\f{1}{\upmu(\xi)}F(\xi)\;d\xi + \norm{F}_{L^2}\int_{B(0,1)}\f{1}{\upmu(\eta)}G(\eta) \; d\eta}\\\le6\norm{1/\upmu}_{L^2(B(0,1))}\norm{F}_{L^2}\norm{G}_{L^2}\norm{H}_{L^2}.
\end{multline}
\end{proof}

We now further decompose the set $E_1$ and the functional $I_1$, as defined in \eqref{eq:bad E} and \eqref{eq: Is}.
\begin{defn}
We make the following definitions for $m,n\in\N$.
\begin{enumerate}
\item We define the set
\begin{align}
	E_{m,n}=\tcb{(\xi,\eta)\in E_1\;:\;2^{-m-1}\le|\xi|\le2^{-m}\;\;\text{and}\;\;2^{-n+1}\le|\xi+\eta|\le2^{-n+2}}.
\end{align}
\item We define the functional $I_{m,n} : (L^0(\R^d;\R))^3 \to [0,\infty]$ via 
\begin{align}
	I_{m,n}[F,G,H]=\int_{E_{m,n}}\f{\upmu(\xi+\eta)}{\upmu(\xi)\upmu(\eta)}F(\xi)G(\eta)H(\xi+\eta)\;d\xi\;d\eta.
\end{align}
\end{enumerate}
\end{defn}
The dyadic decomposition of $E_1$ gives us the following.
\begin{lem}\label{labels are difficult}The following hold.
\begin{enumerate}
\item We have the equalities
\begin{align}
\bigcup_{m=0}^\infty\bigcup_{n=m}^\infty E_{m,n}=E_1\quad\text{and}\quad\sum_{m=0}^\infty\sum_{n=m}^\infty I_{m,n}=I_1.
\end{align}
\item Let $F,G,H \in L^2(\R^d; [0,\infty])$. Then 
\begin{align}
I_{m,n}[F,G,H]\le I_1[F_m,G_m,H_n],
\end{align}
for the functions
\begin{align}
F_m=F\mathbbm{1}_{2^{-m-1}A},\quad G_m=G\mathbbm{1}_{2^{-m-1}A},\quad\text{and}\quad H_n=H\mathbbm{1}_{2^{-n}A},
\end{align}
where $A$ is the annulus $A=\tcb{x\in\R^d\;:\;1/2\le|x|\le4}$.
\item
We have $\mathcal{L}^{d}$-almost everywhere the inequality
\begin{align}
\sum_{m\in\N}\mathbbm{1}_{2^{-m-1}A}\le5\mathbbm{1}_{B(0,2)}
\end{align}
\end{enumerate}
\end{lem}
\begin{proof}
For the first item we note that the second estimate of Lemma~\ref{lemma on equivalence in E1} implies that if $n < m$ then $E_{m,n} = \es$, and hence $I_{m,n}=0$. The second item follows from the first item of Lemma~\ref{lemma on equivalence in E1} and nonnegativity. To prove the third item, we note that if $x\in  B(0,2)\setminus\tcb{0}$, there exists a unique $n\in\N$ such that $|x|\in2^{-n}[1,2)$. It is clear that $\mathbbm{1}_{2^{-m-1}A}(x)=0$ whenever $m\ge0$ and $|m-n|>2$. The inequality now follows.
\end{proof}
Using the preceding lemma we arrive at the following proposition. 
\begin{prop}\label{case 1}Let $d\in\N$. The following hold:
\begin{enumerate}
\item If $d\ge 4$, there exists a constant $ C_d > 0$ such that for all $F,G,H\in L^2(\R^d; [0,\infty])$ we have the estimate $I_1[F,G,H]\le C_d\norm{F}_{L^2}\norm{G}_{L^2}\norm{H}_{L^2}$.
\item If $d\in\tcb{2,3}$, there exists a constant $ C_{1,d} > 0$ such that for all $F,G,H\in L^2(\R^d; [0,\infty])$ we have the estimate
\begin{align}
	\sum_{m=0}^\infty\sum_{n=2m+1}^\infty I_{m,n}[F,G,H]\le C_{1,d}\norm{F}_{L^2}\norm{G}_{L^2}\norm{H}_{L^2}.
\end{align}
\end{enumerate}
\end{prop}
\begin{proof}
Let $m,n,d\in\N$ with $2\le d$ and $m\le n$. According to the second item of Lemma~\ref{labels are difficult}, we have
\begin{align}
I_{m,n}[F,G,H]\le I_1[F_m,G_m,H_n]=\int_{E_1}\f{\upmu(\xi+\eta)}{\upmu(\xi)\upmu(\eta)}F_m(\xi)G_m(\eta)H_n(\xi+\eta)\;d\xi\;d\eta.
\end{align}
For $(\xi,\eta)\in E_1$, the right hand integrand vanishes except possibly when the following inequalities hold: $2^{-m-2}\le|\xi|\le 2^{-m+1}$, $2^{-m-2}\le|\eta|\le 2^{-m+1}$, and $2^{-n-1}\le|\xi+\eta|\le2^{-n+2}$. This provides the elementary estimate
\begin{align}
\f{\upmu(\xi+\eta)}{\upmu(\xi)\upmu(\eta)}\le\f{1+|\xi+\eta|}{|\xi||\eta|}\le 2^{2m+4}\tp{1+2^{-n+2}}\lesssim 2^{2m},
\end{align}
and hence
\begin{equation}\label{whoa}
I_{m,n}[F,G,H]\lesssim 2^{2m}\int_{E_1}F_m(\xi)G_{m}(\eta)H_n(\xi+\eta)\;d\xi\;d\eta.
\end{equation}
Now, for $\ell\in\Z^d$ let $Q_\ell$ denote the closed cube centered at $2^{-n}\ell$ of side length $2^{-n}$ and $\tilde{Q}_\ell$ denote the closed cube centered at $-2^{-n}\ell$ of side length $9\cdot 2^{-n}$. By construction, almost everywhere we have that
\begin{equation}\label{fish}
1=\sum_{\ell\in\Z^d}\mathbbm{1}_{Q_\ell}\quad\text{and}\quad 9^d=\sum_{
\ell\in\Z^d}\mathbbm{1}_{\tilde{Q}_\ell},
\end{equation}
and we may compute 
\begin{equation}\label{fish2}
\max_{\ell\in\Z^d}\tnorm{\mathbbm{1}_{Q_\ell}}_{L^2} = \tnorm{\mathbbm{1}_{Q_0}}_{L^2}=2^{-nd/2}.
\end{equation}
Notice that if $(\xi,\eta)\in E_1$ is such that the integrand on the right hand side of~\eqref{whoa} is nonzero and $\xi\in Q_\ell$, then
\begin{align}
|\eta+2^{-n}\ell|_{\infty}\le|\eta+\xi|+ \abs{-\xi+2^{-n}\ell}_{\infty}\le 2^{-n+2}+2^{-n-1} = \f{9}{2}\cdot 2^{-n},
\end{align}
which in particular implies that $\eta\in\tilde{Q}_\ell$. Hence, 
\begin{equation}\label{whoa2}
\int_{E_1}F_m(\xi)G_m(\eta)H_n(\xi+\eta)\;d\xi\;d\eta\le\sum_{\ell\in\Z^d}\int_{B(0,1)^2}F_m(\xi)\mathbbm{1}_{Q_\ell}(\xi)G_m(\eta)\mathbbm{1}_{\tilde{Q}_\ell}(\eta)H_n(\xi+\eta)\;d\xi\;d\eta.
\end{equation}
By Tonelli's theorem, repeated applications of Cauchy-Schwarz, and the identities in \eqref{fish} and \eqref{fish2} we then find that
\begin{multline}\label{whoa3}
\sum_{\ell\in\Z^d}\int_{B(0,1)^2}F_m(\xi)\mathbbm{1}_{Q_\ell}(\xi)G_m(\eta)\mathbbm{1}_{\tilde{Q}_\ell}(\eta)H_n(\xi+\eta)\;d\xi\;d\eta\\=\sum_{\ell\in\Z^d}\int_{B(0,1)}F_m(\xi)\mathbbm{1}_{Q_\ell}(\xi)\int_{B(0,1)}G_m(\eta)\mathbbm{1}_{\tilde{Q}_\ell}(\eta)H_n(\xi+\eta)\;d\eta\;d\xi\\\le\sum_{\ell\in\Z^d}\int_{B(0,1)}F_m(\xi)\mathbbm{1}_{Q_\ell}(\xi)\tnorm{G_m\mathbbm{1}_{\tilde{Q}_\ell}}_{L^2}\norm{H_n}_{L^2}\;d\xi\le\tnorm{H_n}_{L^2}\tnorm{\mathbbm{1}_{Q_0}}_{L^2}\sum_{\ell\in\Z^d}\tnorm{F_m\mathbbm{1}_{Q_\ell}}_{L^2}\tnorm{G_m\mathbbm{1}_{\tilde{Q}_\ell}}_{L^2}\\
\le2^{-nd/2}\tnorm{H_n}_{L^2}\bp{\int_{\R^d}|F_m(\xi)|^2\sum_{\ell\in\Z^d}\mathbbm{1}_{Q_\ell}(\xi)\;d\xi}^{1/2}\bp{\int_{\R^d}|G_m(\eta)|^2\sum_{\ell\in\Z^d}\mathbbm{1}_{\tilde{Q}_\ell}(\eta)\;d\eta}^{1/2}\\
\lesssim 2^{-nd/2}\tnorm{F_m}_{L^2}\tnorm{G_m}_{L^2}\tnorm{H_n}_{L^2}.
\end{multline}
Next we combine equations~\eqref{whoa}, \eqref{whoa2}, and~\eqref{whoa3} to deduce that
\begin{equation}\label{cat}
 I_{m,n}[F,G,H]\lesssim 2^{2m-nd/2}\tnorm{F_m}_{L^2}\tnorm{G_m}_{L^2}\tnorm{H_n}_{L^2}.
\end{equation}

Now we break to cases based on the dimension. First, if $d\ge 4$, then there is no issue in summing over the set $\tcb{(m,n)\in\N^2\;:\;m\le n}$.  Indeed, we use Cauchy-Schwarz along with the third item of Lemma~\ref{labels are difficult} to find
\begin{multline}
\sum_{m=0}^\infty\sum_{n=m}^\infty I_{m,n}[F,G,H]\lesssim\sum_{m=0}^\infty\sum_{n=m}^\infty2^{2m-nd/2}\tnorm{F_m}_{L^2}\tnorm{G_m}_{L^2}\tnorm{H_n}_{L^2}\\\le\sum_{m=0}^\infty2^{2m}\tnorm{F_m}_{L^2}\tnorm{G_m}_{L^2}\bp{\sum_{n=m}^\infty2^{-nd}}^{1/2}\bp{\int_{\R^d}|H(\zeta)|^2\sum_{n=m}^\infty\mathbbm{1}_{2^{-n}A}(\zeta)\;d\zeta}^{1/2}\\
\lesssim\tnorm{H}_{L^2}\sum_{m=0}^\infty2^{(2-d/2)m}\tnorm{F_m}_{L^2}\tnorm{G_m}_{L^2}\lesssim\tnorm{F}_{L^2}\tnorm{G}_{L^2}\tnorm{H}_{L^2}.
\end{multline}
This completes the proof of the first item.

For the second item, we let $d\in\tcb{2,3}$ in equation~\eqref{cat} and sum over the set $\tcb{\tp{m,n}\in\N^2\;:\;2m+1\le n}$. Arguing as above, we arrive at the bound
\begin{multline}
\sum_{m=0}^\infty\sum_{n=2m+1}^\infty I_{m,n}[F,G,H]\lesssim\tnorm{H}_{L^2}\sum_{m=0}^\infty2^{(2-d)m}\bp{\sum_{n=2m+1}^\infty2^{-d(n-2m)}}^{1/2}\tnorm{F_m}_{L^2}\tnorm{G_m}_{L^2}\\\lesssim\tnorm{F}_{L^2}\tnorm{G}_{L^2}\tnorm{H}_{L^2}.
\end{multline}
This completes the proof of the second item.
\end{proof}

It remains to study the functional $\sum_{m=0}^\infty\sum_{n=m}^{2m}I_{m,n}$ in the case $d\in\tcb{2,3}$. We do this in the subsequent proposition.
\begin{prop}\label{case 2}
Let $d\in\tcb{2,3}$. There exists a constant $ C_{2,d} > 0$ such that for all $F,G,H\in L^2(\R^d; [0,\infty))$  we may estimate
\begin{align}
\sum_{m=0}^{\infty}\sum_{n=m}^{2m}I_{m,n}[F,G,H]\le C_{2,d}\norm{F}_{L^2}\norm{G}_{L^2}\norm{H}_{L^2}.
\end{align}
\end{prop}
\begin{proof}
Let $m,n\in\N$ with $m\le n\le 2m$. We begin by estimating the right hand side of the following inequality of Lemma~\ref{labels are difficult}, $I_{m,n}[F,G,H]\le I_1[F_m,G_m,H_n]$, by decomposing into a rectangular grid. To this end we define the following family of rectangles indexed by $p\in\Z$, $\pi\in\Z^{d-1}$, and $\al\in[1,\infty)$, via 
\begin{align}
R_{p,\pi}(\al)=2^{-2m}[-\al/2+p,\al/2+p]\times 2^{-n}\prod_{\nu=1}^{d-1}[-\al/2+\pi_\nu,\al/2+\pi_\nu].
\end{align}
Suppose that $(\xi,\eta)\in E_1$ are such that the integrand of $I_1[F_m,G_m,H_n]$ is nonzero and that additionally $\xi\in R_{p,\pi}(1)$ and $\eta\in R_{q,\sig}(1)$ for some $p,q\in\Z$ and $\pi,\sig\in\Z^{d-1}$. Then we obtain the inequalities
\begin{align}
2^{-n}(\pi_\nu+\sig_\nu-1/2)\le\xi_{\nu+1}+\eta_{\nu+1} \le 2^{-n}(\pi_\nu+\sig_\nu+1/2) \text{ for all }1\le\nu\le d-1,
\end{align}
and upon pairing these with the bound $|\xi+\eta|\le2^{-n+2}$ we deduce that
\begin{align}
|\pi_\nu+\sig_\nu|\le|\pi_\nu+\sig_\nu-2^{n}(\xi_{\nu+1}+\eta_{\nu+1})|+2^{n}|\xi_{\nu+1}+\eta_{\nu+1}|\le9/2  \text{ for all }1\le\nu\le d-1.
\end{align}
In particular, this implies that $\eta\in R_{q,-\pi}(9)$, which when combined with the inclusion $\xi\in R_{p,\pi}(1)$ further implies that $\xi+\eta\in R_{p+q,0}(10)$.  Therefore, we may estimate
\begin{equation}\label{blue}
I_{m,n}[F,G,H]\le\sum_{p,q\in\Z}\sum_{\pi\in\Z^{d-1}}\int_{E_1}\f{\upmu(\xi+\eta)}{\upmu(\xi)\upmu(\eta)}F_m(\xi)\mathbbm{1}_{R_{p,\pi}(1)}(\xi)G_m(\eta)\mathbbm{1}_{R_{q,-\pi}(9)}(\eta)H_n(\xi+\eta)\mathbbm{1}_{R_{p+q,0}(10)}(\xi+\eta)\;d\xi\;d\eta.
\end{equation}

For fixed $p$, $q$, and $\pi$, we now want to estimate the factor $\upmu(\xi+\eta)/\upmu(\xi)\upmu(\eta)$ appearing in the integrand in \eqref{blue}. First we note that if $\xi \in R_{p,\pi}(1) \cap 2^{-m-1}A$, then $|2^{2m}\xi_1-p|\le 1/2$, which implies that $2^{2m}|\xi_1|\ge\max\tcb{0,|p|-1/2}$ and hence 
\begin{equation}\label{hello}
\upmu(\xi)=\f{|\xi_1|}{|\xi|}+|\xi|\ge2^{-m-1}(2^{2m}|\xi_1|+1/2)\gtrsim 2^{-m}\max\tcb{1,|p|}.
\end{equation}
Similarly, for $\eta \in  R_{q,-\pi}(9) \cap 2^{-m-1}A$, we have the bound 
\begin{equation}\label{goodbye}
\upmu(\eta) \ge 2^{-m-1} (\max\{|q|-9/2,0\} +1/2)  \gtrsim 2^{-m}\max\tcb{1,|q|}.
\end{equation}
Finally, for $\xi+ \eta \in R_{p+q,0}(10) \cap 2^{-n}A$, we estimate $\upmu(\xi+\eta)$ from above, recalling that $m\le n$:
\begin{multline}\label{hello goodbye}
\upmu(\xi+\eta)\le2^{n+1}|\xi_1+\eta_1|+2^{-n+2}=2^{n-2m+1}(2^{2m}|\xi_1+\eta_1|+2^{2(m-n)+1}) \\
\le 2^{n-2m+1}(|p+q|+5+2)
\lesssim 2^{n-2m}\tp{\max\tcb{1,|p|}+\max\tcb{1,|q|}}.
\end{multline}
Upon piecing equations~\eqref{hello}, \eqref{goodbye}, and~\eqref{hello goodbye} all together, we obtain the bound
\begin{equation}\label{hound}
\f{\upmu(\xi+\eta)}{\upmu(\xi)\upmu(\eta)}\lesssim 2^n\bp{\f{1}{\max\tcb{1,|p|}}+\f{1}{\max\tcb{1,|q|}}}
\end{equation}
for $\xi \in R_{p,\pi}(1) \cap 2^{-m-1}A$ and $\eta \in  R_{q,-\pi}(9) \cap 2^{-m-1}A$.

Next, we combine~\eqref{blue} and~\eqref{hound} and then use Tonelli's theorem and Cauchy-Schwarz to obtain
\begin{multline}\label{nobody}
2^{-n}I_{m,n}[F,G,H]\le\\\sum_{p,q\in\Z}\sum_{\pi\in\Z^{d-1}}\f{1}{\max\tcb{1,|p|}}\int_{B(0,1)}F_m(\xi)\mathbbm{1}_{R_{p,\pi}(1)}(\xi)\int_{B(0,1)}G_m(\eta)\mathbbm{1}_{R_{q,-\pi}(9)}(\eta)H_n(\xi+\eta)\mathbbm{1}_{R_{p+q,0}(10)}(\xi+\eta)\;d\eta\;d\xi\\
+\sum_{p,q\in\Z}\sum_{\pi\in\Z^{d-1}}\f{1}{\max\tcb{1,|q|}}\int_{B(0,1)}F_m(\xi)\mathbbm{1}_{R_{p,\pi}(1)}(\xi)\int_{B(0,1)}G_m(\eta)\mathbbm{1}_{R_{q,-\pi}(9)}(\eta)H_n(\xi+\eta)\mathbbm{1}_{R_{p+q,0}(10)}(\xi+\eta)\;d\eta\;d\xi\\
\le\sum_{p,q\in\Z}\sum_{\pi\in\Z^{d-1}}\bp{\f{1}{\max\tcb{1,|p|}}+\f{1}{\max\tcb{1,|q|}}}\int_{B(0,1)}F_m(\xi)\mathbbm{1}_{R_{p,\pi}(1)}(\xi)\tnorm{G_m\mathbbm{1}_{R_{q,-\pi}(9)}}_{L^2}\tnorm{H_n\mathbbm{1}_{R_{p+q,0}(10)}}_{L^2}\;d\xi\\
\le\sum_{p,q\in\Z}\sum_{\pi\in\Z^{d-1}}\bp{\f{1}{\max\tcb{1,|p|}}+\f{1}{\max\tcb{1,|q|}}}\tnorm{\mathbbm{1}_{R_{p,\pi}(1)}}_{L^2}\tnorm{F_m\mathbbm{1}_{R_{p,\pi}(1)}}_{L^2}\tnorm{G_m\mathbbm{1}_{R_{q,\-\pi}(9)}}_{L^2}\tnorm{H_n\mathbbm{1}_{R_{p+q,0}(10)}}_{L^2}\\
=2^{-m-n(d-1)/2}\sum_{p,q\in\Z}\sum_{\pi\in\Z^{d-1}}\bp{\f{1}{\max\tcb{1,|p|}}+\f{1}{\max\tcb{1,|q|}}}\tnorm{F_m\mathbbm{1}_{R_{p,\pi}(1)}}_{L^2}\tnorm{G_m\mathbbm{1}_{R_{q,\-\pi}(9)}}_{L^2}\tnorm{H_n\mathbbm{1}_{R_{p+q,0}(10)}}_{L^2}.
\end{multline}
What remains, before we sum over $m$ and $n$, is a few more applications of Cauchy-Schwarz, one for each sum above. First we handle the sum over $\pi\in\Z^{d-1}$ with
\begin{multline}\label{uncle albert}
\sum_{\pi\in\Z^{d-1}}\tnorm{F_m\mathbbm{1}_{R_{p,\pi}(1)}}_{L^2}\tnorm{G_m\mathbbm{1}_{R_{q,\-\pi}(9)}}_{L^2}\\\le\bp{\int_{\R^{d}}|F_m(\xi)|^2\sum_{\pi\in\Z^{d-1}}\mathbbm{1}_{R_{p,\pi}(1)}(\xi)\;d\xi}^{1/2}\bp{\int_{\R^{d}}|G_m(\eta)|^2\sum_{\pi\in\Z^{d-1}}\mathbbm{1}_{R_{q,-\pi}(9)}(\eta)\;d\eta}^{1/2}\\
\lesssim\tnorm{F_m\mathbbm{1}_{\bigcup_{\pi\in\Z^{d-1}}R_{p,\pi}(1)}}_{L^2}\tnorm{G_m\mathbbm{1}_{\bigcup_{\pi\in\Z^{d-1}}R_{q,-\pi}(9)}}_{L^2}.
\end{multline}
Next we consider the sums over $p,q\in\Z$. First, we consider the one containing $1/\max\tcb{1,|p|}$,
\begin{multline}\label{three legs}
\sum_{p,q\in\Z}\f{1}{\max\tcb{1,|p|}}\tnorm{F_m\mathbbm{1}_{\bigcup_{\pi\in\Z^{d-1}}R_{p,\pi}(1)}}_{L^2}\tnorm{G_m\mathbbm{1}_{\bigcup_{\pi\in\Z^{d-1}}R_{q,-\pi}(9)}}_{L^2}\tnorm{H_n\mathbbm{1}_{R_{p+q,0}(10)}}_{L^2}\\
\lesssim\tnorm{G_m\mathbbm{1}_{\bigcup_{q\in\Z}\bigcup_{\pi\in\Z^{d-1}}R_{q,-\pi}(9)}}_{L^2}\sum_{p\in\Z}\f{1}{\max\tcb{1,|p|}}\tnorm{F_m\mathbbm{1}_{\bigcup_{\pi\in\Z^{d-1}}R_{p,\pi}(1)}}_{L^2}\tnorm{H_n\mathbbm{1}_{\bigcup_{q\in\Z}R_{p+q,0}(10)}}_{L^2}\\
\le\tnorm{G_m}_{L^2}\tnorm{H_n}_{L^2}\bp{\sum_{p\in\Z}\f{1}{\max\tcb{1,p^2}}}^{1/2}\tnorm{F_m\mathbbm{1}_{\bigcup_{p\in\Z}\bigcup_{\pi\in\Z^{d-1}}R_{p,\pi}(1)}}\lesssim\tnorm{F_m}_{L^2}\tnorm{G_m}_{L^2}\tnorm{H_n}_{L^2}.
\end{multline}
Similarly, we consider the one containing $1/\max\tcb{1,|q|}$, 
\begin{multline}\label{smile away}
\sum_{p,q\in\Z}\f{1}{\max\tcb{1,|q|}}\tnorm{F_m\mathbbm{1}_{\bigcup_{\pi\in\Z^{d-1}}R_{p,\pi}(1)}}_{L^2}\tnorm{G_m\mathbbm{1}_{\bigcup_{\pi\in\Z^{d-1}}R_{q,-\pi}(9)}}_{L^2}\tnorm{H_n\mathbbm{1}_{R_{p+q,0}(10)}}_{L^2}\\\lesssim\tnorm{F_m}_{L^2}\tnorm{H_n}_{L^2}\sum_{q\in\Z}\f{1}{\max\tcb{1,|q|}}\tnorm{G_m\mathbbm{1}_{\bigcup_{\pi\in\Z^{d-1}}R_{q,-\pi}(9)}}_{L^2}\lesssim\tnorm{F_m}_{L^2}\tnorm{G_m}_{L^2}\tnorm{H_n}_{L^2}.
\end{multline}
Synthesizing equations~\eqref{nobody}, \eqref{uncle albert}, \eqref{three legs}, and~\eqref{smile away} we find that
\begin{align}
I_{m,n}[F,G,H]\lesssim 2^{-m+n(3-d)/2}\tnorm{F_m}_{L^2}\tnorm{G_m}_{L^2}\tnorm{H_n}_{L^2}.
\end{align}
Hence, we may sum over $m\le n\le 2m$ to bound
\begin{multline}
\sum_{n=m}^{2m}I_{m,n}[F,G,H]\lesssim2^{-m}\bp{\sum_{n=m}^{2m}2^{n(3-d)}}^{1/2}\tnorm{H}_{L^2}\tnorm{F_m}_{L^2}\tnorm{G_m}_{L^2}\\
\le 2^{-m}\tnorm{H}_{L^2}\tnorm{F_m}_{L^2}\tnorm{G_m}_{L^2} \cdot \begin{cases}
	\sqrt{m+1} & \text{if} \; d = 3 \\
	 2^{m+1/2} & \text{if} \;d=2
\end{cases}\lesssim\tnorm{H}_{L^2}\tnorm{F_m}_{L^2}\tnorm{G_m}_{L^2},
\end{multline}
and then finally sum over $m \ge 0$ to see that
\begin{align}
\sum_{m=0}^\infty\sum_{n=m}^{2m}I_{m,n}[F,G,H]\lesssim\tnorm{H}_{L^2}\sum_{m\in\N}\tnorm{F_m}_{L^2}\tnorm{G_m}_{L^2}\lesssim\tnorm{F}_{L^2}\tnorm{G}_{L^2}\tnorm{H}_{L^2},
\end{align}
which is the stated estimate.
\end{proof}

By synthesizing the previous results, we immediately arrive at the following result.

\begin{thm}[Boundedness of $I$]\label{thm:boundedness of I}
The functional $I: L^2(\R^d ; [0,\infty))^3 \to [0,\infty]$ defined in \eqref{eq: I} is bounded and satisfies the estimate \eqref{eq: I est over C}. 
\end{thm}
\begin{proof}
We simply combine Lemma~\ref{labels are difficult} and Propositions ~\ref{proposition on boundedness of I0}, \ref{case 1}, \ref{case 2}, and~\ref{proposition on reduction to boundedness of a trilinear functional}.
\end{proof}

With Theorem~\ref{thm:boundedness of I} in hand, we are ready to prove Theorem~\ref{thm: 1} as a corollary.

\begin{proof}[Proof of Theorem~\ref{thm: 1}]
	By Lemma~\ref{lem: I over nonnegative}, Proposition~\ref{proposition on reduction to boundedness of a trilinear functional}, and
	Theorem~\ref{thm:boundedness of I}, the anisotropic space $X^s(\R^k ; \F)$ is an algebra for $s > k/2$ for $\F = \C$ and $\F = \R$. Let $\Sigma$ be defined as in \eqref{sigma} and let $d = \dim \Sigma$. By the first item of Lemma~\ref{lem: Xs Hs equiv}, it suffices show that $X^s(\Sigma; \F)$ is an algebra for $s > d/2$ for the case when $1 \in R_\Sigma$ and $\abs{R_\Sigma} \ge 2$, as the anisotropic space coincides with the standard Sobolev space otherwise. To show this we first argue similarly as in Proposition~\ref{proposition on reduction to boundedness of a trilinear functional} to show that $X^s(\Sigma; \F)$ is an algebra if and only if for any $F,G \in X^s(\Sigma; \F)$, the estimate 
	\begin{align}\label{eq: f_0g_0 Lambda}
		 \norm{F_0 G_0 }_{X^s(\Sigma)} \lesssim   \norm{F_0 }_{X^s(\Sigma)}\norm{G_0 }_{X^s(\Sigma)}
	\end{align}
	holds for $F_0 = \mathscr{F}^{-1}[\mathbbm{1}_{B_{\hat{\Sigma}}(0,r)}\mathscr{F}[F]], G_0 = \mathscr{F}^{-1}[\mathbbm{1}_{B_{\hat{\Sigma}}(0,r)}\mathscr{F}[G]]$. 
	
	Let $\Sigma_\ast, \hat{\Sigma}_\ast$ be the reordered Cartesian products of $\Sigma$ defined via \eqref{eq: canonical pGamma}, and $\mathcal{P}_\Sigma, \mathcal{P}_{\hat{\Sigma}}$ be the canonical reordering maps defined by \eqref{eq: canonical pGamma}. We first note that since $F_0, G_0 \in X^s(\Sigma; \F)$, by Lemma~\ref{lem: X^s permutation} the functions $F_{0,\Sigma_\ast} = F_0(\mathcal{P}_{\Sigma}^{-1} \cdot), G_{0,\Sigma_\ast} =G_0(\mathcal{P}_{\Sigma}^{-1} \cdot)$ belong to $X^s(\Sigma_\ast ; \F)$ with $\norm{F_{0}}_{X^s(\Sigma)} = \norm{F_{0,\Sigma_\ast}}_{X^s(\Sigma_\ast)}, \norm{G_{0}}_{X^s(\Sigma)} = \norm{G_{0,\Sigma_\ast}}_{X^s(\Sigma_\ast)}$, and by \eqref{eq: appendix fourier comp} we have 
	\begin{align}\label{eq: fourier f lambda}
		 \mathscr{F}_{\Sigma_\ast}^+[F_{0,\Sigma_\ast}](\cdot) = \mathscr{F}^+_{\Sigma}[F_0](\mathcal{P}_{\hat{\Gamma}}^{-1} \cdot) , \; \mathscr{F}_{\Sigma_\ast}^+[G_{0,\Sigma_\ast}](\cdot) = \mathscr{F}^+_{\Sigma}[G_0](\mathcal{P}_{\hat{\Gamma}}^{-1} \cdot).
	\end{align}
	We also have $\norm{F_0 G_0}_{X^s(\Sigma)} = \norm{\sigma_\Sigma F_0 G_0 }_{X^s(\Sigma_\ast)} = \norm{F_{0,\Sigma_\ast} G_{0, \Sigma_\ast}}_{X^s(\Sigma_\ast)}$, where $\sigma_\Sigma$ is defined via \eqref{eq: canonical sigma}. So to establish \eqref{eq: f_0g_0 Lambda} it suffices to establish the estimate 
	\begin{align}\label{eq: X^s algebra desired est}
		\norm{F_{0,\Sigma_\ast} G_{0, \Sigma_\ast}}_{X^s(\Sigma_\ast)} \lesssim \norm{F_{0,\Sigma_\ast}}_{X^s(\Sigma_\ast)} \norm{G_{0,\Sigma_\ast}}_{X^s(\Sigma_\ast)}.  
	\end{align}
	By \eqref{eq: fourier f lambda}, we have $\supp \mathscr{F}^+_{\Sigma_\ast}[F_{0,\Sigma_\ast}], \supp \mathscr{F}^+_{\Sigma_\ast}[G_{0,\Sigma_\ast}] \subseteq B_{\hat{\Sigma}_\ast}(0,r)$. Since $\xi \in B_{\hat{\Sigma}_\ast}(0,r) \Longleftrightarrow \xi \in  \hat{\Sigma}_{*,R}$, by Fubini's theorem we then have 
	\begin{multline}
		\mathscr{F}^+_{\Sigma_\ast}[F_{0,\Sigma_\ast} G_{0,\Sigma_\ast}] (\omega, k) = \int_{\R^{|R_{\hat{\Sigma}_\ast}|}} \sum_{m \in \prod_{i \in T_{\hat{\Sigma}_\ast}} L_i^{-1} \Z} \mathscr{F}^+_{\Sigma_\ast}[F_{0,\Sigma_\ast}](\omega - \eta, k - m) \mathscr{F}^+_{\Sigma_\ast}[G_{0,\Sigma_\ast}](\eta,m) d\eta \\
		= \int_{\R^{|R_{\hat{\Sigma}_\ast}|}} \mathscr{F}^+_{\Sigma_\ast}[F_{0,\Sigma_\ast}](\omega - \eta, k) \mathscr{F}^+_{\Sigma_\ast}[G_{0,\Sigma_\ast}] (\eta, k) \; d\eta  = \delta_{k0}\int_{\R^{|R_{\hat{\Sigma}_\ast}|}} \mathscr{F}^+_{\Sigma_\ast}[F_{0,\Sigma_\ast}](\omega - \eta, 0) \mathscr{F}^+_{\Sigma_\ast}[G_{0,\Sigma_\ast}](\eta, 0) \; d\eta,
	\end{multline}
	where $\delta_{k0}$ is the Kronecker delta. This shows that $\supp \mathscr{F}^+_{\Sigma_\ast}[F_{0,\Sigma_\ast} G_{0,\Sigma_\ast}] \subseteq 
	B_{\hat{\Sigma}_\ast}(0,r)$. By Lemma~\ref{rem: iso everything}, $\uppi_{\hat{\Sigma}_\ast}$ is an isometric measure-preserving group isomorphism between $B_{\hat{\Sigma}_\ast}(0,r)$ and $B_{\R^{|R_{\hat{\Sigma}_\ast}|}}(0,r)$, and therefore 
	\begin{multline}
		 \norm{F_{0,\Sigma_\ast} G_{0,\Sigma_\ast}}_{X^s(\Sigma_\ast)}^2 = \int_{B_{\hat{\Sigma}_\ast}(0,r)} \upmu^2(\xi) \langle \xi \rangle^{2(s-1)} \abs{\mathscr{F}^+_{\Sigma_\ast}[F_{0,\Sigma_\ast} G_{0,\Sigma_\ast}](\xi)}^2 \; d\xi  \\ = \int_{B_{\R^{|R_{\hat{\Sigma}_\ast}|}}(0,r)} \upmu^2(\omega) \langle \omega \rangle^{2(s-1)}\abs{\mathscr{F}^+_{\Sigma_\ast}[F_{0,\Sigma_\ast} G_{0,\Sigma_\ast}](\omega, 0)}^2 \; d\omega 
		 = \norm{F_{0,\Sigma_\ast}(\cdot, 0) G_{0,\Sigma_\ast}(\cdot, 0)}^2_{X^s(\R^{|R_{\hat{\Sigma}_\ast}|})}. 
	\end{multline}
	Next we note that since $F_{0,\Sigma_\ast} \in X^s(\Sigma_\ast ; \F)$, we have the inclusions 
\begin{equation}
	F_{0,\Sigma_\ast}(\cdot, 0) \in \mathscr{S}'(\R^{|R_{\Sigma_\ast}|}; \C) \text{ and } \mathscr{F}^+_{\R^{|R_{\Sigma_\ast}|}}[F_{0,\Sigma_\ast}](\cdot, 0) \in L^1_{\loc}(\R^{|R_{\Sigma_\ast}|}; \C), 
\end{equation}
and we know that $\overline{F_{0,\Sigma_\ast}(\cdot,0)} = F_{0,\Sigma_\ast}(\cdot,0)$, and $\supp \mathscr{F}^+_{\R^{|R_{\Sigma_\ast}|}}[F_{0,\Sigma_\ast}](\cdot, 0) \subseteq B_{\R^{|R_{\hat{\Sigma}_\ast}|}}(0,r)$. By Lemma~\ref{rem: iso everything} again, we have 
	\begin{multline}\label{eq: X^s est same}
		 \norm{F_{0,\Sigma_\ast}(\cdot,0)}_{X^s(\R^{|R_{\Sigma_\ast}|})}^2 = \int_{B_{\R^{|R_{\hat{\Sigma}_\ast}|}}(0,r)} \upmu^2(\omega) \langle \omega \rangle^{2(s-1)}\abs{\mathscr{F}^+_{\Sigma_\ast}[F_{0,\Sigma_\ast}](\omega, 0)}^2 \; d\omega \\
		 = \int_{B_{\hat{\Sigma}_\ast}(0,r)} \upmu^2(\xi) \langle \xi \rangle^{2(s-1)} \abs{\mathscr{F}^+_{\Sigma_\ast}[F_{0,\Sigma_\ast}](\xi)}^2 \; d\xi = \norm{F_{0,\Sigma_\ast}}_{X^s(\Sigma_\ast)}^2.  
	\end{multline}
	This shows that $F_{0,\Sigma_\ast}(\cdot, 0) \in X^s(\R^{|R_{\Sigma_\ast}|}; \F)$, and similarly we can conclude that $G_{0,\Sigma_\ast}(\cdot, 0) \in X^s(\R^{|R_{\Sigma_\ast}|}; \F)$. Since $s > d/2 \ge \abs{R_{\Sigma_\ast}}/2$, $X^s(\R^{\abs{R_{\Sigma_\ast}}}; \F)$ is an algebra, and therefore we have the estimate 
	\begin{align}\label{eq: X^s est lowdim}
		\norm{F_{0,\Sigma_\ast}(\cdot, 0) G_{0,\Sigma_\ast}(\cdot, 0)}_{X^s(\R^{|R_{\hat{\Sigma}_\ast}|})} \lesssim \norm{F_{0,\Sigma_\ast}(\cdot, 0)}_{X^s(\R^{|R_{\hat{\Sigma}_\ast}|})} \norm{G_{0,\Sigma_\ast}(\cdot, 0)}_{X^s(\R^{|R_{\hat{\Sigma}_\ast}|})}.
	\end{align}
	Combining \eqref{eq: X^s est same} and \eqref{eq: X^s est lowdim} then gives us \eqref{eq: X^s algebra desired est} and in turn \eqref{eq: f_0g_0 Lambda}. The desired conclusion then follows. 
\end{proof}

\section{Linear analysis}\label{sec:linear analysis}
In this section we record the linear analysis associated to the flattened system \eqref{eq: main flattened}. First, we introduce the definition of the $\dot{H}^{-1}$ seminorm.

\begin{defn}[The $\dot{H}^{-1}$ seminorm]\label{defn: H -1}
	For $s \ge 0$, in the case when $R_\Gamma = \varnothing$, for $f \in H^s(\Gamma; \R)$ we define the $\dot{H}^{-1}$ seminorm of $f$ via 
	\begin{align}\label{eq:torusseminorm}
		 [f]_{\dot{H}^{-1}} = \begin{cases}
			  0, \; \text{if} \; \hat{f}(0) = 0 \\
			  \infty, \; \text{otherwise}.
		 \end{cases}
	\end{align}
In the remaining cases, we define the $\dot{H}^{-1}$ seminorm of $f$ via 
\begin{align}\label{eq:seminorm}
	[f]_{\dot{H}^{-1}} = \left( \int_{B_{\R^{|R_{\hat{\Gamma}}|}}(0,r)} \abs{\omega}^{-2} \abs{\hat{f}(\mathcal{P}_{\hat{\Gamma}}^{-1} l_{\hat{\Gamma}_\ast}\omega)}^2 \; d\omega  \right)^{1/2},
\end{align}
where $\Gamma_\ast$ is defined in \eqref{eq: canonical Gamma}, 
$\mathcal{P}_{\Gamma_\ast}: \Gamma \to \Gamma_\ast$ is defined in \eqref{eq: canonical pGamma}, and $l_{\hat{\Gamma}_\ast}: \R^{\abs{R_\Gamma}} \to \Gamma_\ast$ is defined via $l_{\hat{\Gamma}_\ast}(\omega) = (\omega,0)$. For $f \in H^s(\Gamma ; \R)$, we write $f \in \dot{H}^{-1}(\Gamma;\R)$ to mean that $[f]_{\dot{H}^{-1}} < \infty$.

\end{defn}

\begin{rem}\label{rem:zero fourier}
	We note that our definition of the $\dot{H}^{-1}$ seminorm coincides with the standard definition of the $\dot{H}^{-1}$ seminorm for functions $H^s(\Gamma;\R)$. First, if a function $f$ already belongs to $H^s(\Gamma ; \R)$, then 
	\begin{align}\label{eq: Hdot1 lower freq}
		 \int_{\hat{\Gamma}} \abs{\xi}^{-2} \abs{\hat{f}(\xi)}^2 \; d \xi \asymp \int_{B_{\hat{\Gamma}}(0,r)} \abs{\xi}^{-2} \abs{\hat{f}(\xi)}^2 \; d \xi.
	\end{align}
	Second, we note that the product measure on $\hat{\Gamma}$ is a product of Lebesgue measures and counting measures, corresponding to the $\R$ and $L_i^{-1}\Z$ factors of $\hat{\Gamma}$. If $R_{\hat{\Gamma}} = \varnothing$, then the only low frequency mode is the 0 mode, and thus the integral in \eqref{eq: Hdot1 lower freq} takes values in $\{0, \infty\}$ depending on the value of $\hat{f}(0)$.  If $R_{\hat{\Gamma}}  \neq \varnothing$, then we note that by Lemma~\ref{rem: iso everything}, using the natural isometric and measure-preserving identification between $B_{\hat{\Gamma}}(0,r)$ and $B_{\R^{|R_{\Gamma}|}}(0,r)$  we have 
	\begin{align}\label{eq: Hdot-1 equiv}
		\int_{B_{\hat{\Gamma}}(0,r)} \abs{\xi}^{-2} \abs{\hat{f}(\xi)}^2 \; d \xi \asymp  \int_{B_{\R^{|R_{\hat{\Gamma}}|}}(0,r)} \abs{\omega}^{-2} \abs{\hat{f}(\mathcal{P}_{\Gamma_\ast}^{-1} l_{\hat{\Gamma}_\ast}\omega)}^2 \; d\omega  .
	\end{align}
\end{rem}

\subsection{Asymptotics of the normal stress to solution map associated to \eqref{eq: intro overdet}}\label{sec:asymptotics}

In this subsection we record the asymptotics of some special functions associated to the normal stress to solution map associated to \eqref{eq: intro overdet}, which will play a crucial role in the construction of the free surface function $\eta$. In the case where $\hat{\Sigma} = \R^{n-1}$, the precise asymptotics of these symbols were derived using ODE techniques developed in \cite{leonitice}. For the other possibilities of $\Sigma$ considered in \eqref{sigma}, these estimates will continue to hold as $\hat{\Sigma}$ remains a subgroup of $\R^{n-1}$ and the asymptotics remain the same upon subsampling. 

First we record an auxiliary result. 

\begin{thm}[Hilbert isomorphism for the $\gamma$-Stokes system]\label{thm: strong}
	Suppose $s \ge 0$. The map  
			\begin{align}
				 \Phi_\gamma : \Hzeros{s+2}(\Omega ; \R^n) \times H^{s+1}(\Omega;\R) \to H^{s} (\Omega; \R^n) \times H^{s+1}(\Omega;\R) \times H^{s+\frac{1}{2}}(\Sigma_b ; \R^n)
			\end{align}
			defined via 
			\begin{align}\label{eq:Phi}
				 \Phi_\gamma(u,p) = \left( \diverge S(p,u) - \gamma \p_1 u, \diverge u, S(p,u) e_n \right)
			\end{align}
			is a Hilbert isomorphism for all $\gamma \in \R$. 
\end{thm}
\begin{proof}
The case $\Sigma = \R^{n-1}$ follows from Theorem 2.6 in \cite{leonitice}, and the arguments therein can be adapted with minimal modification to handle the cases when $\Sigma \neq \R^{n-1}$. 
\end{proof}
Next, we define the normal stress to velocity and pressure maps induced by $\Phi_\gamma$.

\begin{defn}[Normal stress to velocity and pressure maps]
	Let $\gamma \in \R$ and $s \ge 0$. We define the normal stress to velocity and pressure maps to be the linear maps $\mathcal{U}_\gamma : H^{s+\frac{1}{2}}(\Omega; \R) \to H^{s+2}\left(\Omega;\R^n \right)$ and $\mathcal{P}_\gamma : H^{s+\frac{1}{2} }(\Omega ; \R) \to H^{s+1}(\Omega; \R)$ defined via \begin{align}\label{eq: UP}
		\mathcal{U}_\gamma (\psi) = u, \; \mathcal{P}_\gamma (\psi) = p,
	\end{align}
		where $(u,p)$ is the unique solution to the $-\gamma$-Stokes system
		\begin{align} \label{eq:adjoint}
			\begin{cases}
			\diverge S(p,u) + \gamma \p_1 u = 0, &\text{in} \; \Omega \\
			\diverge u = 0, &\text{in} \; \Omega\\
			S(p,u) e_n  =\psi e_n, &\text{on} \; \Sigma_b \\
			u = 0, &\text{on} \; \Sigma_{0}.	
			\end{cases}
		\end{align}
	The existence and uniqueness of $(u,p)$ and the boundedness of $\mathcal{U}_\gamma, \mathcal{P}_\gamma$ are guaranteed by Theorem~\ref{thm: strong}.
\end{defn}

Now we may record properties of the pseudodifferential operator associated to the normal stress to solution maps \eqref{eq: UP}. 

\begin{thm}[Symbols associated to the pseudodifferential operator and their asymptotics]\label{thm:pseudo asymp}
	Suppose $s \ge 0$. The linear maps $\mathcal{U}_\gamma, \mathcal{P}_\gamma$ defined in \eqref{eq: UP} are well-defined and bounded. Moreover, there exists bounded and measurable functions $V: \hat{\Sigma} \times [0,b] \times \R \to \C^n, Q:\hat{\Sigma} \times [0,b] \times \R \to \C^n$, and $m: \Sigma \times \R \to \C$ such that for all $\psi \in H^{s+\frac{1}{2}}(\Sigma_b ; \R)$, we have 
	\begin{align} \label{eq:multiplers}
		\widehat{\mathcal{U}_\gamma(\psi)}(\xi,x_n) = V(\xi, x_n, -\gamma) \widehat{\psi}(\xi), \quad 
		\widehat{\mathcal{P}_\gamma(\psi)} (\xi,x_n) = Q(\xi, x_n, -\gamma) \widehat{\psi}(\xi), \quad
		m(\xi,\gamma) = V_n(\xi,b,\gamma).
	\end{align}
 Furthermore, the following hold.
	\begin{enumerate}
		\item $V,Q,m$ are continuous, with $V(0,x_n,\gamma) = 0, Q(0,x_n,\gamma) =1$, and $m(0,\gamma) = 0$.
		\item $\overline{V(\xi,x_n,\gamma)} = V(-\xi,x_n,\gamma), \overline{Q(\xi,x_n,\gamma)} = Q(-\xi,x_n,\gamma), \overline{m(\xi,\gamma)} = m(-\xi, \gamma)$ for all $\xi \in \hat{\Sigma}$.
		\item For each $\xi \in \hat{\Sigma}$, $V(\xi,\cdot, \gamma), Q(\xi,\cdot,\gamma)$ solve 
		\begin{align} \label{eq:VQ ode}
			\begin{cases}
			\left( - \p_n^2 + 4\pi^2 \abs{\xi}^2 \right) V' + 2\pi i \xi Q - 2\pi i \xi_1 \gamma V' = 0, & \text{in} \; (0, b) \\
			\left( - \p_n^2 + 4\pi^2 \abs{\xi}^2 \right) V_n + \p_n Q - 2\pi i \xi_1 \gamma V_n = 0, & \text{in} \; (0, b)  \\
			2\pi i \xi \cdot V' + \p_n V_n = 0, & \text{in} \; (0, b)  \\
			-\p_n V' - 2\pi i \xi V_n = 0,  & \text{for} \; x_n = b \\
			Q - 2 \p_n V_n =  1, & \text{for} \; x_n = b \\
			V = 0, &\text{for} \; x_n = 0.
			\end{cases}
			\end{align}
		\item For $\abs{\xi} \ll 1$, we have 
	    \begin{multline}\label{eq: asymp dev near 0}
			V'(\xi,x_n,\gamma) = - \pi \xi(2x_n b - x_n^2) + O(\abs{\xi}^2), \quad V_n(\xi, x_n,\gamma) = 2\pi^2 \abs{\xi}^2 x_n^2 \left( \frac{x_n}{3} - b \right) + O(\abs{\xi}^3), \\ Q(\xi,x_n,\gamma) = 1 + O(\abs{\xi}^2), \quad m(\xi,\gamma) = - \frac{4\pi^2 \abs{\xi}^2 b^3}{3} + O(\abs{\xi}^3), \quad \text{as} \; \abs{\xi} \to 0
		\end{multline}
		where $F(\xi, x_n) = O(\abs{\xi}^k)$ as $\abs{\xi} \to 0$ means that 
		\begin{align}
			 \limsup_{\abs{\xi} \to 0} \sup_{0 \le x_n \le b} \frac{\abs{F(\xi,x_n)}}{\abs{\xi}^k} < \infty. 
		\end{align}

		\item For each $\gamma \in \R$, there exists a constant $c = c(\gamma, b) > 0$ and $R = R(\gamma, b) > 0$ such that for $x_n \in [0,b]$ and $\abs{\xi} > R$, we have the point-wise estimates
		\begin{align}\label{eq: asymp dev to infty}
		\abs{V'(\xi,x_n,\gamma)} &\le c \left( \frac{\abs{\gamma}}{\abs{\xi}^2} + (b-x_n) \right) e^{-2\pi \abs{\xi} (b-x_n)} + c e^{-2\pi \abs{\xi}b}, \\
		 \abs{V_n(\xi,x_n,\gamma)} &\le c \left( \frac{1}{\abs{\xi}} + (b-x_n) \right) e^{-2\pi \abs{\xi} (b-x_n)} + c e^{-2\pi \abs{\xi}b}, \\
		 \abs{m(\xi,\gamma) + \frac{1}{4\pi \abs{\xi}}} &\le c \frac{1}{\abs{\xi}^2}, \\
		 \abs{Q(\xi,x_n,\gamma)} &\le c e^{-2\pi\abs{\xi}(b-x_n)} + c e^{-2\pi \abs{\xi} b}.
		\end{align}
		\item For each $\gamma \in \R$, there exists a constant $c =c(n,b,\gamma)$ such that 
		\begin{align}\label{eq:psi asymp 1}
			 \sup_{\abs{\xi} \le1, \xi \neq 0} \frac{1}{\abs{\xi}^2} \left( \int_0^b \abs{V(\xi,x_n,\gamma)}^2 + \abs{Q(\xi,x_n,\gamma)- 1}^2 \; dx_n + \abs{V(\xi,b,\gamma)}^2 \right) \le c
		\end{align}
		and 
		\begin{align}\label{eq:psi asymp 2}
			 (1+\abs{\xi}^2)^{3/2} \int_0^b \abs{V(\xi,x_n,\gamma)}^2 \; dx_n + (1+\abs{\xi}^2)^{1/2} \int_0^b \abs{Q(\xi,x_n,\gamma)}^2 \; dx_n + (1+\abs{\xi}^2) \abs{V(\xi,b,\gamma)}^2 \le c,
		\end{align}
		for all $\xi \in \hat{\Sigma}$. 
	\end{enumerate}
\end{thm}

\begin{proof}
If $\hat{\Sigma} = \R^{n-1}$, all six items follow from Theorem 4.5, Theorem 4.10, and Corollary 4.11 in \cite{leonitice}. In the other cases, we have $\hat{\Sigma} \subseteq \widehat{\R^{n-1}} = \R^{n-1}$. Therefore, the same conclusions will follow as we are sampling the frequencies on subgroups of $\R^{n-1}$. 
\end{proof}

We conclude this subsection by recording the properties of an auxiliary function defined in terms of $m$. 

\begin{lem}\label{lem:rho}
	Suppose $\gamma \in \R \setminus \{0\}$, and define 
	\begin{align}\label{eq:rho}
		\rho_\gamma(\xi) = 2\pi i \gamma \xi_1 + (1+4\pi^2 \abs{\xi}^2) \sigma \overline{m(\xi,-\gamma)}.
   \end{align}
   Then the following hold.
   \begin{enumerate}
	    \item $\rho_\gamma(\xi) = 0$ if and only if $\xi = 0$, and $\overline{\rho_\gamma(\xi)} = \rho_\gamma(-\xi)$ for all $\xi \in \hat{\Sigma}$. 
		\item For $\sigma > 0$, there exists a constant $C = C(n,\gamma,\sigma, b)> 0$  such that for all $\xi \in \hat{\Sigma}$, we have 
		\begin{align}
		  C^{-1} \abs{ \rho_\gamma(\xi)}^2 \le (\xi_1^2 + \abs{\xi}^4) \mathbbm{1}_{B(0,1)}(\xi) + (1+\abs{\xi}^2)  \mathbbm{1}_{B(0,1)^c}(\xi) \le C \abs{\rho_\gamma(\xi)}^2. 
		\end{align}
		\item For $\sigma = 0$ and $n=2$, there exists a constant $C = C(\gamma, b)> 0$  such that for all $\xi \in \hat{\Sigma}$, we have 
		\begin{align}
			C^{-1}\abs{ \rho_\gamma(\xi)}^2 \le \abs{\xi}^2 \mathbbm{1}_{B(0,1)}(\xi) + (1 + \abs{\xi}^2) \mathbbm{1}_{B(0,1)^c}(\xi) \le C \abs{\rho_\gamma(\xi)}^2. 
		  \end{align}
   \end{enumerate}
\end{lem}

\begin{proof}
This follows from the proof for Lemma 6.1 in \cite{leonitice} and the asymptotics of $m(\xi,-\gamma)$ on general $\hat{\Sigma}$ recorded as the third and fourth items of Theorem~\ref{thm:pseudo asymp}.
\end{proof}

\subsection{Compatibility conditions and the Hilbert isomorphism associated to the overdetermined \texorpdfstring{$\gamma$}{gamma}-Stokes system}

In this subsection we record the Hilbert isomorphism that completely characterize the solvability of the overdetermined system \eqref{eq: intro overdet}. First we introduce a pair of function spaces for the data tuple $(f,g,h,k)$ that encodes the compatibility conditions associated to \eqref{eq: intro overdet}.

\begin{defn}
Let $s \ge 0$ and $\Sigma$ be defined as in \eqref{sigma}. We say that a data tuple $(f,g,h,k) \in H^{s}(\Omega ; \R^n) \times H^{s+1}(\Omega ; \R) \times H^{s+3/2}(\Sigma ; \R)  \times H^{s+1/2}(\Sigma; \R^n) $ satisfy the divergence-trace compatibility condition if 
\begin{align}\label{eq:divtrace}
	h - \int_{0}^b g(\cdot, x_n) \; dx_n \in \dot{H}^{-1} (\Sigma; \R).
\end{align}
We define the Hilbert space $\mathcal{Y}^s$ to be 
	\begin{multline}\label{defn:Ys}
		\mathcal{Y}^s = \{ (f,g,h,k) \in H^{s}(\Omega ; \R^n) \times H^{s+1}(\Omega ; \R) \times  H^{s+3/2}(\Sigma ; \R)  \times H^{s+1/2}(\Sigma; \R^n) \mid \\ (f,g,h,k) \; \text{satisfy the divergence-trace condition} \; \eqref{eq:divtrace} \},
	\end{multline}
    with the associated norm defined via 
	\begin{align}
		 \norm{(f,g,h,k)}_{\mathcal{Y}^s}^2 = \norm{f}_{H^{s}}^2 + \norm{g}_{H^{s+1}}^2 + \norm{k}_{H^{s+1/2}}^2 + \norm{h}_{H^{s+3/2}}^2  + \left[ h - \int_0^b g(\cdot, x_n) \; dx_n \right]_{\dot{H}^{-1}}^2.
	\end{align}
	\end{defn}

	\begin{defn}\label{defn:Zs}
	Let $s \ge 0$. We say that the data tuple $(f,g,h,k)\in H^{s}(\Omega ; \R^n) \times H^{s+1}(\Omega ; \R) \times H^{s+3/2}(\Sigma ; \R)  \times H^{s+1/2}(\Sigma; \R^n) $ satisfy the adjoint compatibility condition if for every $\psi \in H^{s+1/2}(\Sigma;\R)$, 
	\begin{align}\label{eq: 2nd comp fourier}
		 \int_0^b \hat{f}(\xi,x_n) \cdot \overline{V(\xi,x_n,-\gamma)} - \hat{g}(\xi,x_n) \overline{Q(\xi,x_n,-\gamma)} \; dx_n - \hat{k}(\xi) \cdot \overline{V(\xi,b,-\gamma)} + \hat{h}(\xi) = 0,
	\end{align}
	where $V,Q$ are defined in terms of $\psi$ in \eqref{eq:VQ ode}. We define the closed subspace $\mathcal{Z}^s$ of $\mathcal{Y}^s$ to be 
		\begin{align}
			 \mathcal{Z}^s = \{ (f,g,h,k) \in \mathcal{Y}^s \mid  (f,g,h,k) \; \text{satisfy}\; \eqref{eq: 2nd comp fourier} \; \forall \psi \in H^{s+1/2}(\Sigma;\R) \}.
		\end{align}
	\end{defn}	
Now we are ready to record the main result of this subsection. 
\begin{thm}\label{thm:overlinear}
Let $s \ge 0, \gamma \in \R$, and $\mathcal{Z}^s$ the Hilbert space defined in \eqref{defn:Zs}. Then the bounded linear operator $\Psi_\gamma: \Hzeros{s+2}(\Omega ; \R^n) \times H^{s+1}(\Omega ; \R^n) \to \mathcal{Z}^s$ given by 
\begin{align}
	 \Psi_\gamma(u,p) = (\diverge S(p,u) - \gamma \p_1 u, \diverge u , u_n \rvert_{\Sigma_b}, S(p,u) e_n \rvert_{\Sigma_b})
\end{align}
is an isomorphism.
\end{thm}

\begin{proof}
The case $\Sigma = \R^{n-1}$ follows from Theorem 3.4 of \cite{leonitice}. We note that by Remark~\ref{rem:zero fourier}, the calculations performed in $\R^d$ in \cite{leonitice} are also valid over the low frequency regime in $\hat{\Gamma}$, therefore the arguments therein can be adapted with minimal modification to handle the cases when $\Sigma \neq \R^{n-1}$. 
\end{proof}

\subsection{Linear analysis with $\eta$ and $\kappa$} 
In this section we would like to establish the solvability of the $\gamma$-Stokes system with gravity capillary boundary conditions 
\begin{align} \label{eq:gammacapillary}
	\begin{cases}
		\diverge S(p,u) - \gamma \p_1 u + (\nabla' \eta, 0)= f & \text{in} \; \Omega \\
	\diverge u = g, & \text{in} \; \Omega \\
	u_n + \gamma \p_1 \eta= h, & \text{on} \; \Sigma_b \\
	S(p,u)e_n + \sigma \Delta' \eta e_n = k, & \text{on} \; \Sigma_b \\
	u= 0, & \text{on} \; \Sigma_{0},
\end{cases}
\end{align}
for data tuples belonging to the space $\mathcal{Y}^s$. First, we introduce the container space for the solution tuple $(u,p,\eta)$.

\begin{defn}\label{defn:Xsfordata}
	For $s \ge 0$, we define the Hilbert space 
	\begin{align}
		 \mathcal{X}^s = \begin{cases}
			\{ (u,p,\eta) \in \Hzeros{s+2}(\Omega ;\R^n) \times H^{s+1}(\Omega; \R) \times X^{s+5/2}(\Sigma ; \R)  \}, &R_\Sigma \neq \varnothing \\
			\{ (u,p,\eta) \in \Hzeros{s+2}(\Omega ;\R^n) \times H^{s+1}(\Omega; \R) \times \mathring{X}^{s+5/2}(\Sigma;\R) \}, &R_\Sigma = \varnothing,
		 \end{cases}
	\end{align}
where
\begin{align}
	\mathring{X}^{s+5/2}(\Sigma;\R) = \{ \eta \in X^{s+5/2}(\Sigma;\R)  \mid \hat{\eta}(0) = 0 \}. 
\end{align}
We endow $\mathcal{X}^s$ with the natural product norm defined via 
	\begin{align}
		 \norm{(u,p,\eta)}_{\mathcal{X}^s}^2 = \norm{u}_{\Hzeros{s+1}}^2 + \norm{p}_{H^{s+1}}^2 + \norm{\eta}_{X^{s+5/2}}^2.     
	\end{align}
	\end{defn}
Next, we record an embedding result for $\mathcal{X}^s$. 
\begin{prop}\label{prop:embed X}
	Suppose $s \ge 0$ and $\mathcal{X}^s$ is the Banach space in Definition~\ref{defn:Xsfordata}.  If $s > n/2$, then we have the continuous inclusion 
		\begin{align}
			 \mathcal{X}^s \subseteq C_b^{\lf s - n/2 \rf}(\Omega ; \R^n) \times C_b^{1+\lf s - n/2\rf}(\Omega ; \R)\times C_0^{3+\lf s - n/2 \rf}(\Sigma; \R).
		\end{align}
		Moreover, if $(u,p,\eta) \in \mathcal{X}^s$, then 
		\begin{align} 
			\lim_{\abs{x_R'} \to \infty} \p^\alpha u(x) &= 0 \; \text{for all} \; \alpha \in \N^n \; \text{such that} \; \abs{\alpha} \le 2 + \left\lfloor s - \frac{n}{2} \right\rfloor \\
			\lim_{\abs{x_R'} \to \infty} \p^\alpha p(x) &= 0 \; \text{for all} \; \alpha \in \N^n \; \text{such that} \; \abs{\alpha} \le 1 + \left\lfloor s - \frac{n}{2} \right\rfloor.
		\end{align}
	\end{prop}

\begin{proof}
This follows from standard Sobolev embedding and the first item of Theorem~\ref{thm:Xs}.
\end{proof}

In the subsection to follow, we establish some preliminary results to be utilized in the subsequent analysis. 

\subsection{Preliminary results}
In this subsection we use the asymptotics recorded in Section~\ref{sec:asymptotics} to show that we can construct the free surface function $\eta$ from a given data tuple $(f,g,h,k)$. First, we study an auxiliary function defined in terms of the multipliers $V, Q$ defined in \eqref{eq:multiplers}.

\begin{lem}\label{lem:psi}
	Let $s \ge 0$, $\gamma \in \R$ and $(f,g,h,k) \in \mathcal{Y}^s$, where $\mathcal{Y}^s$ is the Hilbert space defined in \eqref{defn:Ys}. Consider $V(\cdot,\cdot,-\gamma) : \hat{\Sigma} \times [0,b] \to \C^n$, $Q(\cdot,\cdot,-\gamma) : \hat{\Sigma} \times [0,b] \to \C$ defined in \eqref{eq:multiplers}. We define the measurable function $\psi : \hat{\Sigma} \to \C$ via
	\begin{align}\label{eq:psi defn}
		 \psi(\xi) = \int_0^b \left( \hat{f}(\xi,x_n) \cdot \overline{V(\xi,x_n,-\gamma)} - \hat{g}(\xi,x_n) \overline{Q(\xi,x_n,-\gamma)} \right) \; dx_n - \hat{k}(\xi) \cdot \overline{V(\xi,b,-\gamma)} + \hat{h}(\xi).
	\end{align}

	Then the following hold.
	\begin{enumerate}
	 \item If $R_\Sigma = \varnothing$, then $\psi(0) = 0$. 
	 \item $\overline{\psi(\xi)} = \psi(-\xi)$ for every $\xi \in \hat{\Sigma}$.
	 \item We have the estimate
	 \begin{align}\label{eq:psi bound}
		  \int_{\hat{\Sigma}} \frac{1}{\abs{\xi}^2} \abs{\psi(\xi)}^2 \mathbbm{1}_{B(0,r)}+ (1+\abs{\xi}^2)^{s+3/2} \abs{\psi(\xi)}^2 \mathbbm{1}_{B(0,r)^c}
		  \; d\xi \lesssim_{n,s,\gamma,b} \norm{(f,g,h,k)}_{\mathcal{Y}^s}^2. 
	 \end{align}
	\end{enumerate}
	\end{lem}
    
\begin{proof}
To prove the first item, we assume $R_\Sigma = \varnothing$ and rewrite
\begin{multline}\label{eq:psi rewrite}
	\psi(\xi) = \int_0^b \left( \hat{f}(\xi,x_n) \cdot \overline{V(\xi,x_n,-\gamma)} - \hat{g}(\xi,x_n) (\overline{Q(\xi,x_n,-\gamma)} -1) \right) \; dx_n \\ - \hat{k}(\xi) \cdot \overline{V(\xi,b,-\gamma)}  + \left(\hat{h}(\xi) - \int_0^b \hat{g}(\xi,x_n) \; dx_n \right).
\end{multline}
We note that since $h,g$ must satisfy \eqref{eq:divtrace}, by Remark~\ref{rem:zero fourier} we have $\hat{h}(0) - \int_0^b \hat{g}(0,x_n) \; dx_n  = 0$. We also note that by the first item of Theorem~\ref{thm:pseudo asymp}, $\overline{V(0,x_n,-\gamma)} = 0$ and $\overline{Q(\xi,x_n,-\gamma)} = 0$. The first item follows immediately from these two observations. 

To prove the second item, we first note that by using the second item of Theorem~\ref{thm:pseudo asymp} and by using the fact that $f,g,k,h$ are real-valued,
\begin{multline}
	 \overline{\psi(\xi)} = \int_0^b \left( \overline{\hat{f}(\xi,x_n)} \cdot V(\xi,x_n,-\gamma) - \overline{\hat{g}(\xi,x_n)} Q(\xi,x_n,-\gamma) \right) \; dx_n - \overline{\hat{k}(\xi)} \cdot V(\xi,b,-\gamma) + \overline{\hat{h}(\xi)} 
	 \\
	 = \int_0^b \left( \hat{f}(-\xi,x_n) \cdot \overline{V(-\xi,x_n,-\gamma)}
	 - \hat{g}(-\xi,x_n) \overline{Q(-\xi,x_n,-\gamma)} \right) \; dx_n - \hat{k}(-\xi) \cdot \overline{V(-\xi,b,-\gamma)} + \hat{h}(-\xi) = \psi(-\xi).
\end{multline}
To prove the third, we first rewrite $\psi$ again as in \eqref{eq:psi rewrite} and apply the Cauchy-Schwarz inequality to obtain
\begin{multline}\label{eq:psi bound comb 1}
	 \abs{\psi(\xi)}^2 \lesssim  \left( \int_0^b \abs{\hat{f}(\xi,x_n)}^2 \; d\xi \right) \left( \int_0^b \abs{V(\xi,x_n)}^2 \; d\xi \right)  +   \left( \int_0^b \abs{\hat{g}(\xi,x_n)}^2 \; d\xi \right) \left( \int_0^b \abs{Q(\xi,x_n, - \gamma) - 1}^2 \; d\xi \right) \\
	 + \abs{\hat{k}(\xi)}^2 \abs{V(\xi,b,-\gamma)}^2 + \abs{\hat{h}(\xi) - \int_0^b \hat{g}(\xi,x_n) \; dx_n}^2.
\end{multline}
We note that by the definition of the $\dot{H}^{-1}$ seminorm, we have 
\begin{align}\label{eq:psi bound comb 2}
	 \int_{B(0,r)} \frac{1}{\abs{\xi}^2} \abs{\hat{h}(\xi) - \int_0^b \hat{g}(\xi,x_n) \; dx_n}^2 \; d\xi \le \left[ h - \int_0^b g(\cdot, x_n) \; dx_n \right]_{\dot{H}^{-1}}^2. 
\end{align}
Combining \eqref{eq:psi bound comb 1}, \eqref{eq:psi bound comb 2} with \eqref{eq:psi asymp 1}, Tonelli's theorem, and Parseval's theorem then gives us 
\begin{multline}\label{eq:psi bound fin 1}
	 \int_{B(0,r)} \frac{1}{\abs{\xi}^2} \abs{\psi(\xi)}^2 \; d \xi \lesssim_{n,\gamma,b} \int_{B(0,r)} \int_0^b \left( \abs{\hat{f}(\xi,x_n)}^2 + \abs{\hat{g}(\xi,x_n)}^2 \right) dx_n d\xi  \\ + \int_{B(0,r)} \abs{\hat{k}(\xi)}^2 \; d\xi + \left[ h - \int_0^b g(\cdot, x_n) \; dx_n \right]_{\dot{H}^{-1}}^2 \lesssim_{n,\gamma,b} \norm{(f,g,h,k)}_{\mathcal{Y}^s}^2. 
\end{multline}
If $\abs{\xi} \ge 1$, then we apply the Cauchy-Schwarz inequality directly to \eqref{eq:psi defn} and obtain 
\begin{multline}\label{eq:psi bound comb 3}
	\abs{\psi(\xi)}^2 \lesssim  \left( \int_0^b \abs{\hat{f}(\xi,x_n)}^2 \; d\xi \right) \left( \int_0^b \abs{V(\xi,x_n)}^2 \; d\xi \right)  +   \left( \int_0^b \abs{\hat{g}(\xi,x_n)}^2 \; d\xi \right) \left( \int_0^b \abs{Q(\xi,x_n, - \gamma)}^2 \; d\xi \right) \\
	+ \abs{\hat{k}(\xi)}^2 \abs{V(\xi,b,-\gamma)}^2 + \abs{\hat{h}(\xi)}^2.
\end{multline}
Combining \eqref{eq:psi bound comb 3} with \eqref{eq:psi asymp 2}, Tonelli's theorem, and Parseval's theorem then gives us 
\begin{multline}\label{eq:psi bound fin 2}
	 \int_{B(0,1)^c} (1+\abs{\xi}^2)^{s+3/2} \abs{\psi(\xi)}^2 \; d\xi \lesssim_{n,\gamma,b} \int_0^b \int_{\hat{\Sigma}} (1+\abs{\xi}^2)^s \abs{\hat{f}(\xi,x_n)}^2 \; d\xi dx_n + \int_0^b \int_{\hat{\Sigma}} (1+\abs{\xi}^2)^s \abs{\hat{g}(\xi,x_n)}^2 \; d\xi dx_n \\ + \int_{\hat{\Sigma}} (1+\abs{\xi}^2)^{s+1/2} \abs{\hat{k}(\xi)}^2 \; d\xi +  \int_{\hat{\Sigma}} (1+\abs{\xi}^2)^{s+3/2} \abs{\hat{h}(\xi)}^2 \; d\xi \lesssim_{n,\gamma,b} \int_0^b \norm{f(\cdot,x_n)}^2_{H^{s}(\Sigma)} \; dx_n \\ + \int_0^b \norm{g(\cdot,x_n)}^2_{H^{s}(\Sigma)} \; dx_n + \norm{k}_{H^{s+1/2}}^2 + \norm{h}_{H^{s+3/2}}^2 \lesssim_{n,s,\gamma,b} \norm{(f,g,h,k)}_{\mathcal{Y}^s}^2.   
\end{multline}
\end{proof}

Next we study the linear map $\Upsilon_{\gamma,\sigma} : \mathcal{X}^s \to \mathcal{Y}^s$ defined via 
\begin{align} \label{eq:upsilon}
	\Upsilon_{\gamma,\sigma}(u,p,\eta) = (\diverge S(p,u) - \gamma \p_1 u + (\nabla' \eta , 0), \diverge u, u_n \rvert_{\Sigma_b} + \gamma \p_1 \eta, S(p,u) e_n \rvert_{\Sigma_b} + \sigma \Delta'\eta  e_n)
\end{align}
which is the solution operator corresponding to the system \eqref{eq:gammacapillary}. The next result shows that this linear map is well-defined, bounded, and also injective. 

\begin{prop}\label{prop:Upsiloninjective}
Suppose $\gamma \in \R \setminus \{0\}$, $\sigma \ge 0$, and $s \ge 0$. Then the linear map $\Upsilon_{\gamma,\sigma}: \mathcal{X}^s \to \mathcal{Y}^s$ defined in \eqref{eq:upsilon} is well-defined, continuous, and injective.
\end{prop}

\begin{proof}
	
We first check that the map is well-defined and continuous. By the fifth item of Theorem~\ref{thm: X^s gen}, the first component $f = \diverge S(p,u) - \gamma \p_1 u + (\nabla' \eta, 0)$ belongs to $H^{s}(\Omega ; \R^n)$ with the estimate
\begin{align}
	 \norm{f}_{H^s} \lesssim \norm{u}_{\Hzeros{s+2}} + \norm{p}_{H^{s+1}} + \norm{\eta}_{X^{s+5/2}}.
\end{align}
The second component $g = \diverge u$ belongs to $H^{s+1}(\Omega ; \R)$ with the estimate $\norm{g}_{H^{s+1}} \lesssim \norm{u}_{\Hzeros{s+2}}$. By the fifth item of Theorem~\ref{thm: X^s gen} and standard trace theory, we deduce that the third component $k = S(p,u) e_n \rvert_{\Sigma_b} + \sigma \Delta'\eta e_n$ belongs to $H^{s+1/2}(\Sigma ; \R^n)$ with the estimate 
\begin{multline}
	 \norm{k}_{H^{s+1/2}} \le \norm{p}_{H^{s+1/2}(\Sigma)} + \norm{\mathbb{D} u}_{H^{s+1/2}(\Sigma)} + \norm{\sigma \Delta' \eta}_{H^{s+1/2}(\Sigma)} \lesssim \norm{p}_{H^{s+1}(\Omega)} + \norm{u}_{\Hzeros{s+2}} + \norm{\eta}_{X^{s+5/2}}.  
\end{multline}
By the fifth item of Theorem~\ref{thm: X^s gen} and standard trace theory, the fourth component $h =  u_n \rvert_{\Sigma_b} + \gamma \p_1 \eta$ belongs to $H^{s+3/2}(\Sigma_b; \R)$ with the estimate
\begin{align}
	 \norm{h}_{H^{s+3/2}} \lesssim \norm{u}_{\Hzeros{s+2}} + \norm{\eta}_{X^{s+5/2}}. 
\end{align}
Next we note that since $g = \diverge u$, 
\begin{align}
	 u_n(x',b) - \int_0^b g(x',x_n) \; dx_n = \int_0^b \p_n u_n(x',x_n) \; dx_n - \int_0^b g(x',x_n) = - \diverge' \int_0^b u'(x',x_n) \; dx_n
\end{align}
for a.e. $x' \in \Sigma$. Writing $H^{s+3}(\Sigma ; \R) \ni R (x')=  \int_0^b u'(x',x_n) \; dx_n$, we first note that if $R_\Gamma = \varnothing$, then $\widehat{\diverge R} (0) = 0$ and $\widehat{\gamma \p_1 \eta}(0) = 0$. Then we note that 
\begin{align}
	\left[ u_n- \int_0^b g(\cdot, x_n) \; dx_n \right]_{\dot{H}^{-1}} = \left[ \diverge R\right]_{\dot{H}^{-1}} \lesssim \norm{u}_{L^2},
\end{align}
therefore 
\begin{align}
	\left[ h - \int_0^b g(\cdot, x_n) \; dx_n \right]_{\dot{H}^{-1}} \le \left[ u_n- \int_0^b g(\cdot, x_n) \; dx_n \right]_{\dot{H}^{-1}} + [\gamma \p_1 \eta]_{\dot{H}^{-1}} \lesssim \norm{u}_{\Hzeros{s+2}} + \norm{\eta}_{X^{s+5/2}}.  
\end{align}
Combining the estimates above shows that $ \Upsilon_{\gamma,\sigma}$ is well-defined and continuous. 

To show that $\Upsilon_{\gamma,\sigma}$ is injective, we suppose $(u, p, \eta) \in \mathcal{X}^s$ and  $\Upsilon_{\gamma,\sigma} \left( u, p, \eta \right) = 0$. We note that if $\tilde{p} = p - \eta$, then $\nabla \tilde{p} = \nabla p - (\nabla' \eta,0)$ and $\tilde{p}I = p I - \eta I$. Therefore $\Upsilon_{\gamma,\sigma} (u,p,\eta) =0$ if and only if $(u,\tilde{p},\eta)$ satisfies 
	\begin{align}\label{eq:upsilon=0}
		\begin{cases}
			\diverge S(\tilde{p},u) - \gamma \p_1 u = 0, & \text{in} \; \Omega \\
			\diverge u = 0, & \text{in} \; \Omega\\
			S(\tilde{p},u) e_n = (\eta - \sigma \Delta' \eta) e_n, & \text{on} \; \Omega_b\\
			u_n + \gamma \p_1 \eta = 0, & \text{on} \; \Sigma_b \\
			u = 0, & \text{on} \; \Sigma_0.
		\end{cases}
	\end{align}
    We note that by Tonelli's theorem, Parseval's theorem, and the fifth item of Theorem~\ref{thm: X^s gen} we have $\hat{u}(\xi, \cdot) \in H^{s}((0,b);\C^n)$ and $\widehat{\tilde{p}}(\xi,\cdot) \in H^1((0,b);\C)$, for a.e. $\xi \in \hat{\Sigma}$. By the second item in Theorem~\ref{thm:Xs}, $\hat{\eta} \in L^1(\Sigma; \R) + L^2(\Sigma, (1+\abs{\xi}^2)^{(s+5/2)/2} d\xi; \R)$. Thus, we may apply the horizontal Fourier transform to \eqref{eq:upsilon=0} to deduce for  a.e. $\xi \in \hat{\Sigma}$, $w = \hat{u}(\xi,\cdot), q = \widehat{\tilde{p}}(\xi,\cdot)$ satisfies
	\begin{align} \label{eq:odesys}
	\begin{cases}
	\left( - \p_n^2 + 4\pi^2 \abs{\xi}^2 \right) w' + 2\pi i \xi q - 2\pi i \xi_1 \gamma w' = 0, & \text{in} \; (0, b) \\
	\left( - \p_n^2 + 4\pi^2 \abs{\xi}^2 \right) w_n + \p_n q - 2\pi i \xi_1 \gamma w_n = 0, & \text{in} \; (0, b)  \\
	2\pi i \xi \cdot w' + \p_n w_n = 0, & \text{in} \; (0, b)  \\
	-\p_n w' - 2\pi i \xi w_n = 0,  & \text{for} \; x_n = b \\
	q - 2 \p_n w_n = (1+4\pi^2 \abs{\xi}^2 \sigma) \hat{\eta}, & \text{for} \; x_n = b \\
	w_n + 2\pi i \xi_1 \gamma \hat{\eta} = 0,  & \text{for} \; x_n = b \\
	w = 0, &\text{for} \; x_n = 0.
	\end{cases}
	\end{align}

	First consider the special case when $R_\Sigma = \varnothing$ and $\xi = 0$.  In this case, the third and sixth equations imply that  $w_n \equiv 0$. Since $\eta \in \mathring{X}^{s+5/2}(\Sigma;\R)$, we have $\hat{\eta} (0) = 0$ and therefore by the second and fifth equations, $q \equiv 0$. The first, fourth, and last equations then tell us that $w' \equiv 0$. Therefore, we can conclude that at $\xi = 0$, $w, q \equiv 0$.

	Next consider the general case for which $\xi \in \hat{\Sigma} \setminus \{ 0 \}$ and \eqref{eq:odesys} holds. Using \eqref{eq:odesys} and  performing integration by parts (for further details, see Proposition 4.1 in \cite{leonitice}.), we deduce that 
\begin{multline}
		\int_0^b - \gamma 2\pi i \xi_1 \abs{w}^2 + 2 \abs{\p_n w_n}^2 + \abs{\p_n w' + 2\pi i \xi w_n}^2 + \frac{1}{2} \abs{2 \pi i \xi \otimes w' + w' \otimes 2\pi i \xi}^2 \; dx_n = - 2\pi i \xi_1 \gamma(1+4\pi^2\abs{\xi}^2 \sigma) \abs{\hat{\eta}(\xi)}^2. 
\end{multline}	 
		By taking the real part of this expression, we see that we must have for a.e. $\xi \in \hat{\Sigma}$, $\p_n w_n \equiv 0$ and $\p_n w' + 2\pi i \xi w_n \equiv 0$ in $(0,b)$, but since $w(0) = 0$, we must have $w \equiv 0$ in $[0,b]$. Then by the first equation, we must have $q \equiv 0$. By the second to last equation, we find that $\eta \equiv 0$. From this we find that $(u,p,\eta) = (0,0,0)$, so we can conclude that $\Upsilon_{\gamma, \sigma}$ is injective.

\end{proof}

Next we show that $\Upsilon_{\gamma,\sigma}$ is surjective. To do so we must construct the free surface function $\eta$ from a given data tuple $(f,g,h,k) \in \mathcal{Y}^s$. For the reader's convenience, we record this construction in the next subsection. 

\subsection{Construction of the free surface function and the isomorphism associated to \eqref{eq:gammacapillary}}

\begin{lem}[Construction of $\eta$ in the presence of surface tension]\label{lem:etaconstruct1}
Suppose $\gamma \in \R \setminus \{0\}$, $\sigma > 0$ and $n \ge 2$, $s \ge 0$. Then for every $(f,g,h,k) \in \mathcal{Y}^s$, there exists an $\eta \in X^{s+\frac{5}{2}}(\Sigma;\R)$ for which the modified data tuple
	\begin{align}\label{eq:moddata}
			(f - (\nabla' \eta, 0), g, h - \gamma \p_1 \eta, k - \sigma \Delta'\eta e_n)
			 \in H^{s} (\Omega;\R^n) \times H^{s+1}(\Omega; \R) \times H^{s+\frac{3}{2}}(\Sigma_b; \R) \times H^{s+\frac{1}{2}}(\Sigma_b; \R^n) 
	\end{align}
	belongs to the range of $\Upsilon_{\gamma,\sigma}$ defined in \eqref{eq:upsilon}. Moreover, we have the estimate 
	\begin{align}\label{eq:lem eta est}
		 \norm{\eta}_{X^{s+\frac{5}{2}}} \lesssim \norm{(f,g,h,k)}_{\mathcal{Y}^s}. 
	\end{align}
\end{lem}
\begin{proof}
	Given $(f,g,h,k) \in \mathcal{Y}^s$, we propose to define $\eta \in X^{s+\frac{5}{2}}(\Sigma; \R)$ through
	\begin{align}\label{eq: defn eta hat}
		 \hat{\eta}(\xi) = \begin{cases}
			\frac{\psi(\xi)}{\rho_\gamma(\xi)}, &\xi \neq 0 \\
			0, & \xi = 0,
		 \end{cases}
	\end{align}
	where $\rho_\gamma$ is defined in \eqref{eq:rho} and $\psi$ is defined in terms of $(f,g,h,k)$ in \eqref{eq:psi defn}. We note that the choice of $\hat{\eta}(0) = 0$ is only relevant in the case when $R_{\Gamma} = \varnothing$.
	
	We first note that by the first item of Lemma~\ref{lem:rho} and the second item of Lemma~\ref{lem:psi}, we have $\overline{\hat{\eta}(\xi)} =  \hat{\eta}(-\xi)$. Furthermore, by the second item of Lemma~\ref{lem:rho} and the third item of Lemma~\ref{lem:psi}, we have the estimate  
	\begin{multline}\label{eq: eta 1}
			 \int_{\hat{\Sigma}} \left(  \frac{\xi_1^2 + \abs{\xi}^4}{\abs{\xi}^2} \mathbbm{1}_{B(0,r)} + (1+\abs{\xi}^2)^{s+\frac{5}{2}} \mathbbm{1}_{B(0,r)^c}\right)  \abs{\hat{\eta}(\xi)}^2 \; d\xi   
			\\ \lesssim \int_{\hat{\Sigma}} \left( \frac{1}{\abs{\xi}^2} \mathbbm{1}_{B(0,r)} + (1+\abs{\xi}^2)^{s+\frac{3}{2}} \mathbbm{1}_{B(0,r)^c}\right) \abs{\rho_\gamma(\xi)}^2 \abs{\hat{\eta}(\xi)}^2 \; d\xi
			\\
			 \lesssim \int_{\hat{\Sigma}} \left( \frac{1}{\abs{\xi}^2} \mathbbm{1}_{B(0,r)} + (1+\abs{\xi}^2)^{s+\frac{3}{2}} \mathbbm{1}_{B(0,r )^c}\right) \abs{\psi(\xi)}^2\; d\xi  \lesssim \norm{(f,g,h,k)}_{\mathcal{Y}^s}^2.
	\end{multline}
	Consequently, if we define $\eta = (\hat{\eta})^\vee$, then by \eqref{eq: eta 1}, and $\overline{\hat{\eta}(\xi)} = \hat{\eta}(\xi)$, we have $\eta \in X^{s+5/2}(\Sigma;\R)$ with the estimate \eqref{eq:lem eta est}. 
	
	To conclude our proof it suffices to show that the modified data given in \eqref{eq:moddata} belongs to the range of $\Upsilon_{\gamma,\sigma}(u,p,\eta)$. By Theorem~\ref{thm:overlinear}, it suffices to show that the modified data tuple has the desired regularity and satisfies the divergence-trace compatibility condition \eqref{eq:divtrace} and the adjoint compatibility condition \eqref{eq: 2nd comp fourier}. We note that since $\eta \in X^{s+5/2}(\Sigma; \R)$, by the fifth item of Theorem~\ref{thm: X^s gen} and the third item of Theorem~\ref{thm:Xs}, we have $(f, g, h - \gamma \p_1 \eta, k - \sigma \Delta'\eta e_n)
	\in H^{s} (\Omega;\R^n) \times H^{s+1}(\Omega; \R) \times H^{s+\frac{3}{2}}(\Sigma_b; \R) \times H^{s+\frac{1}{2}}(\Sigma_b; \R^n)$ with $(h - \gamma \p_1 \eta) - \int_0^b g(\cdot, x_n) \in \dot{H}^{-1}(\Sigma ; \R)$. To check \eqref{eq: 2nd comp fourier}, we write $0= \psi - \rho_\gamma \hat{\eta}$ for $\xi \neq 0$ and use the definition of $\psi$ and $\rho_\gamma$ to obtain 
	\begin{multline}\label{eq: eta 2}
		 0 = \int_0^b \left( \hat{f}(\xi,x_n) \cdot \overline{V(\xi,x_n,-\gamma)} - \hat{g}(\xi,x_n) \overline{Q(\xi,x_n,-\gamma)}  \right) \; dx_n \\ - (\hat{k}(\xi) + (1 + 4\pi^2 \abs{\xi}^2 \sigma) \hat{\eta}(\xi) e_n )\cdot \overline{V(\xi,b,-\gamma)}  + \hat{h}(\xi) - 2\pi i \xi_1 \gamma \hat{\eta}(\xi).
	\end{multline} 
	Using the third equation in \eqref{eq:VQ ode} we have $2\pi i \xi \cdot V'(\xi,x_n, -\gamma) + \p_n V_n (\xi,x_n,-\gamma) = 0$, and since $V(\xi,0,-\gamma) =0$, we have 
	\begin{align}
		 \overline{V_n(\xi,b,-\gamma)} = \int_0^b \p_n \overline{V_n(\xi, x_n, -\gamma)} \; dx_n = \int_0^b 2\pi i \xi \cdot \overline{V'(\xi,x_n,-\gamma)} \; dx_n.
	\end{align}
	Thus, we can rewrite \eqref{eq: eta 2} as 
	\begin{multline}\label{eq: eta 3}
		0 = \int_0^b \left( \left( \hat{f}(\xi,x_n) - ( \widehat{\nabla' \eta}(\xi), 0) \right) \cdot \overline{V(\xi,x_n,-\gamma)} - \hat{g}(\xi,x_n) \overline{Q(\xi,x_n,-\gamma)}  \right) \; dx_n \\ - (\hat{k}(\xi) + \widehat{\sigma \Delta' \eta}(\xi) e_n )\cdot \overline{V(\xi,b,-\gamma)}  + \hat{h}(\xi) - 2\pi i \xi_1 \gamma \hat{\eta}(\xi),
   \end{multline} 
   and the desired conclusion follows immediately. 
\end{proof}

For the special case of $n =2$, we can also construct the free surface function $\eta$ in the case without surface tension.

\begin{lem}[Construction of the free surface function without surface tension]\label{lem:etaconstruct2}
	Suppose $\gamma \in \R \setminus \{0\}$, $\sigma = 0$ and $n = 2$, $s \ge 0$. Then for every $(f,g,h,k) \in \mathcal{Y}^s$, there exists an $\eta \in H^{s+\frac{5}{2}}(\Sigma;\R)$ for which the modified data tuple
	\begin{align}
		(f - \p_1 \eta e_1 , g, h - \gamma \p_1 \eta, k)
		 \in H^{s} (\Omega;\R^n) \times H^{s+1}(\Omega ; \R) \times H^{s+\frac{3}{2}}(\Sigma_b ; \R) \times H^{s+\frac{1}{2}}(\Sigma_b; \R^n) 
\end{align}
	belongs to the range of $\Upsilon_{\gamma,\sigma}$ defined in \eqref{eq:upsilon}. Moreover, we have the estimate 
	\begin{align}\label{eq: eta est 2}
		 \norm{\eta}_{X^{s+\frac{5}{2}}} \lesssim \norm{(f,g,h,k)}_{\mathcal{Y}^s}. 
	\end{align}
\end{lem}
\begin{proof}
We follow the construction in the previous lemma and propose to define $\eta$ via \eqref{eq: defn eta hat}. Then $\overline{\hat{\eta}(\xi)} = \hat{\eta}(-\xi)$, $\hat{\eta}(0) = 0$, and by the third item of Lemma~\ref{lem:rho} and the third item of Lemma~\ref{lem:psi}, we have 
\begin{multline}\label{eq: eta 0}
	\int_{\hat{\Sigma}} \left(  (1+\abs{\xi}^2) \mathbbm{1}_{B(0,1)} + (1+\abs{\xi}^2)^{s+\frac{5}{2}} \mathbbm{1}_{B(0,1)^c}\right)  \abs{\hat{\eta}(\xi)}^2 \; d\xi   
	\\ \lesssim \int_{\hat{\Sigma}} \left( \frac{1}{\abs{\xi}^2} \mathbbm{1}_{B(0,1)} + (1+\abs{\xi}^2)^{s+\frac{3}{2}} \mathbbm{1}_{B(0,1)^c}\right) \abs{\rho_\gamma(\xi)}^2 \abs{\hat{\eta}(\xi)}^2 \; d\xi
   \\
	\lesssim \int_{\hat{\Sigma}} \left( \frac{1}{\abs{\xi}^2} \mathbbm{1}_{B(0,1)} + (1+\abs{\xi}^2)^{s+\frac{3}{2}} \mathbbm{1}_{B(0,1)^c}\right) \abs{\psi(\xi)}^2\; d\xi  \lesssim \norm{(f,g,h,k)}_{\mathcal{Y}^s}^2.
\end{multline}
Consequently, we may define $\eta = (\hat{\eta})^\vee$. By \eqref{eq: eta 0} and Lemma~\ref{lem: Xs Hs equiv}, we have $\eta \in H^{s+5/2}(\Sigma ; \R) = X^{s+5/2}(\Sigma;\R)$ with the estimate \eqref{eq: eta est 2}. By the third item of Theorem~\ref{thm:Xs}, we have $(f - \p_1 \eta e_1 , g, h - \gamma \p_1 \eta, k ) \in H^{s} (\Omega;\R^n) \times H^{s+1}(\Omega; \R) \times H^{s+\frac{3}{2}}(\Sigma_b; \R) \times H^{s+\frac{1}{2}}(\Sigma_b; \R^n)$ and $(h - \gamma \p_1 \eta) - \int_0^b g(\cdot, x_2) \; dx_2 \in \dot{H}^{-1}(\Sigma; \R)$. 
To conclude we follow the same line of calculations as in Lemma~\ref{lem:etaconstruct1} to arrive at 
\begin{multline}
	0 = \int_0^b  \left( (\hat{f}(\xi,x_2) - 2 \pi i \xi_1 \hat{\eta} e_1)  \cdot \overline{V(\xi,x_2,-\gamma)} - \hat{g}(\xi,x_2) \overline{Q(\xi,x_2,-\gamma)}  \right) \; dx_2  - \hat{k}(\xi)\cdot \overline{V(\xi,b,-\gamma)}  \\ + \hat{h}(\xi) - 2\pi i \xi_1 \gamma \hat{\eta}(\xi),
\end{multline} 
which verifies the overdetermined compatibility condition \eqref{eq: 2nd comp fourier}.
\end{proof}

Now we are ready to prove that $\Upsilon_{\gamma,\sigma}: \mathcal{X}^s \to \mathcal{Y}^s$ is an isomorphism when $\sigma > 0$ and $n \ge 2$, and when $\sigma = 0$ and $n = 2$.

\begin{thm}[Existence and uniqueness of solutions to \eqref{eq:gammacapillary}]\label{thm:upsiloniso}
Suppose $\gamma \in \R \setminus \{0\}, s \ge 0$.
\begin{enumerate}
	 \item (Isomorphism in the case with surface tension) If $\sigma > 0$ and $n \ge 2$, then the bounded linear map $\Upsilon_{\gamma,\sigma} : \mathcal{X}^s \to \mathcal{Y}^s$ defined in \eqref{eq:upsilon} is an isomorphism.
	 \item (Isomorphism in the case without surface tension) If $\sigma = 0$ and $n = 2$, then the bounded linear map $\Upsilon_{\gamma,0} : \mathcal{X}^s \to \mathcal{Y}^s$ defined in \eqref{eq:upsilon} is an isomorphism.
\end{enumerate}

\end{thm}
\begin{proof}
To prove the first item, by Proposition~\ref{prop:Upsiloninjective}, it suffices to show that $\Upsilon_{\gamma,\sigma}$ is surjective. Suppose $(f,g,h,k) \in \mathcal{Y}^s$, and define the free surface function $\eta \in X^{s+5/2}(\Sigma ; \R)$ by the construction in Lemma~\ref{lem:etaconstruct1}. By Theorem~\ref{thm:overlinear}, there exists $(u,p) \in \Hzeros{s+2}(\Omega ; \R^n) \times H^{s+1}(\Omega ; \R)$ such that $\Psi_\gamma (u,p) = (\diverge S(p,u) - \gamma \p_1 u, \diverge u , u_n \rvert_{\Sigma_b}, S(p,u) e_n \rvert_{\Sigma_b}) = (f - (\nabla' \eta, 0), g, h - \gamma \p_1 \eta, k - \sigma \Delta'\eta e_n)$. Therefore, we find that $\Upsilon_{\gamma,\sigma} (u,p, \eta) = (f,g,h,k)$. This shows that $\Upsilon_{\gamma,\sigma}$ is surjective, and it follows that $\Upsilon_{\gamma,\sigma}$ is an isomorphism. 

To prove the second item we follow the same argument as above, using Lemma~\ref{lem:etaconstruct2} in place of Lemma~\ref{lem:etaconstruct1} and $\Upsilon_{\gamma,0}$ in place of $\Upsilon_{\gamma,\sigma}$.
\end{proof}

\subsection{Parameter regime for the linear isomorphism associated to \eqref{eq: main flattened}}\label{sec:kappa est}

In this subsection we consider the linearization of the flattened system \eqref{eq: main flattened} around the trivial solution $u=0,p=0,\eta =0$, which is given by the map $\mathcal{X}^s \ni (u,p,\eta) \mapsto   L_{\kappa, \sigma} \in \mathcal{Y}^s$ defined in \eqref{eq:intro Lkappa}. Our goal is to show that there is a parameter regime in terms of $\kappa$ for which $L_{\kappa, \sigma}$ remains an isomorphism between $\mathcal{X}^s$ and $\mathcal{Y}^s$. To achieve this we first define $L_{0,\sigma}, M_{\kappa} : \mathcal{X}^s \to \mathcal{Y}^s$ via 
\begin{multline} \label{eq:Lkappa}
	L_{0,\sigma} (u,p,\eta)= \Upsilon_{\gamma,\sigma}(u,p,\eta)=  (\diverge S(p,u) - \gamma \p_1 u + (\nabla' \eta, 0), \diverge u, u_n \rvert_{\Sigma_b} + \gamma \p_1 \eta , S(p,u) e_n \rvert_{\Sigma_b} + \sigma \Delta'\eta e_n),
\end{multline}
and
\begin{multline} \label{eq:defMkappa}
	M_\kappa (u, p,\eta) = (- \gamma \kappa x_n \p_1 \eta e_1 + \kappa  s_0(x_n) \p_1 u + \kappa s_0'(x_n) u_n e_1 +  \kappa^2 x_n s_0(x_n) \p_1 \eta e_1 - \kappa x_n \Delta' \eta e_1  - \kappa ( x_n \nabla' \p_1 \eta, \p_1 \eta)  ,\\ \kappa x_n \p_1 \eta, - \frac{\kappa b^2}{2} \p_1 \eta,  0),
\end{multline}
so that $L_{\kappa,\sigma} = L_{0,\sigma} + M_\kappa = L_{0,\sigma} (I+L_{0,\sigma}^{-1} M_\kappa)$, where $L_{\kappa,\sigma}$ is defined via \eqref{eq:intro Lkappa} and $s_0$ is defined via \eqref{eq: shear}. Our goal in this subsection is to show that $\norm{M_\kappa} < \frac{1}{\norm{L_{0,\sigma}}^{-1}}$ for sufficiently small $\kappa$. 

We first establish a preliminary lemma. While the following result is not at all surprising and follows from routine applications of Tonelli's theorem and H\"{o}lder's inequality, we record the proof below to show that the universal constants in the estimates are purely combinatorial and do not depend on the physical parameters. 

 \begin{lem}\label{lem:simpleprod}
	 Let $\Z \ni n \ge 2$ and suppose $f \in C^\infty ([0,b] ; \R)$.  The following hold.	 
	 \begin{enumerate}
	 \item If $g \in H^k(\Omega ; \R^n)$ for some integer $k \ge 0$ and for any $i \in \N, j \in \N^{n}$ with $\abs{j} \le k$ we define $(\p^if \p^j g)(x) = \p^i f (x_n) \p^jg(x)$, then we have  
	 \begin{align} 
		 \norm{\p^i f \p^j g }_{L^2} \le \norm{\p^i f}_{L^\infty} \norm{g}_{H^j}. 
	 \end{align}
	 Consequently, $fg \in H^k(\Omega ; \R^n)$ and there exists a combinatorial constant $C = C(k) > 0$ such that
	 \begin{align} \label{eq: simple eq 1}
		\norm{f g}_{H^k} \le C \max_{0 \le i \le k} \norm{\p^i f}_{L^\infty} \norm{g}_{H^k}.
	\end{align}
	 \item If $g \in H^k(\Sigma ; \R)$ for some integer $k \ge 0$ and for any $i \in \N, j \in \N^{n-1}$ with $\abs{j} \le k$ we define $(\p^i f \p^j g)(x) = \p^i f (x_n) \p^j g(x')$, then we have  
	 \begin{align} 
		 \norm{\p^i f \p^j g }_{L^2} \le \norm{\p^i f}_{L^2} \norm{g}_{H^j}.
	 \end{align}
	 Consequently, $fg \in H^k(\Omega ; \R)$ and there exists a combinatorial constant $C = C(k) > 0$ such that 
	 \begin{align} \label{eq: simple eq 2}
		\norm{f g}_{H^k} \le C \max_{0 \le i \le k} \norm{\p^i f}_{L^2} \norm{g}_{H^k}.
	\end{align}
	 \end{enumerate}
 \end{lem}
\begin{proof}
	To prove the first item, we note that by Tonelli's theorem and H\"{o}lder's inequality, 
	\begin{multline} 
		\norm{\p^i f \p^j g}_{L^2} = \left( \int_{\Sigma} \int_0^b   \abs{\p^i f(x_n)}^2 \abs{\p^jg(x)}^2  \; dx_n dx'\right)^{\frac{1}{2}} \le \norm{\p^i f}_{L^{\infty}(0,b)} \left( \int_{\Sigma} \int_0^b \abs{\p^j g(x)}^2 \; dx_n d x'  \right)^{\frac{1}{2}} 
		\\ \le \norm{\p^i f}_{L^{\infty}(0,b)} \norm{\p^j g}_{L^2\left( \Sigma \times (0,b) \right)} \le \norm{\p^i f}_{L^\infty} \norm{g}_{H^{j}}.
\end{multline}
Then by the Leibniz formula, we have 
\begin{align} 
	\norm{f g}_{H^k} \le\sum_{\substack{0 \le \abs{j} \le k, \; m \le j \in \N^{n} }} \binom{j}{m} \norm{\p^{j} f}_{L^\infty} \norm{\p^{j-m} g}_{L^2} \le C \max_{0 \le i \le k \in \N} \norm{\p^i f}_{L^\infty} \norm{g}_{H^k}
\end{align}	
where $C = C(k) > 0$ is a combinatorial constant depending only on $k$. 
To prove the second item, we note that by Tonelli's theorem and H\"{o}lder's inequality again, 
\begin{align} 
	\norm{\p^i f \p^j g}_{L^2}  = \left( \int_0^b   \abs{\p^i f(x_n)}^2 \; dx_n  \int_{\Sigma}  \abs{\p^j g(x')}^2 \; dx'\right)^{\frac{1}{2}} 	     
	= \norm{\p^i f}_{L^2} \norm{\p^j g}_{L^2} 
	\le \norm{\p^i f}_{L^2} \norm{g}_{H^{j}}. 
\end{align}
Then by the Leibniz formula,
\begin{align} 
	\norm{f g}_{H^k} \le \sum_{\substack{0 \le \abs{j} \le k, \; m \le j \in \N^{n-1} }} \binom{j}{m} \norm{\p^
{j} f}_{L^2} \norm{\p^{j-m} g}_{L^2} \le C \max_{0 \le i \le k \in \N} \norm{\p^i f}_{L^2} \norm{g}_{H^k}.
\end{align}	
\end{proof}

This gives us an immediate corollary that will prove useful in the next section. 

\begin{cor}[Multiplication with smooth functions in the vertical variable]\label{cor: mult xn}
	Let $\Sigma$ be defined as in \eqref{sigma} and suppose $\R \ni s \ge 0$. If $\varphi \in C^\infty([0,b]; \R)$, then the map $T_\varphi: H^s(\Sigma; \R) \to H^s(\Omega ; \R)$ defined via 
	\begin{align}
		 T_\varphi f (x) = \varphi(x_n) f(x')
	\end{align}
	 is well-defined and bounded. 
\end{cor}
\begin{proof}
 This follows directly from the second item of Lemma~\ref{lem:simpleprod} when $s$ is an integer. By interpolation, the second item also holds when $s$ is real-valued. 
\end{proof}

Now we are ready to prove the main result of this subsection.

 \begin{thm}\label{thm:M} 
	 Let $\R \ni s \ge 0, \kappa \in \R$ and consider the linear map $M_\kappa: \mathcal{X}^s \to \mathcal{Y}^s$ as defined in \eqref{eq:defMkappa}. Then the following hold.

	\begin{enumerate}
		\item There exists a constant $C > 0$ depending on $s,n$, and in the periodic cases on $L_i$, such that 
		\begin{align} \label{eq:Mbound}
			\norm{M_\kappa}_{\mathcal{L}(\mathcal{X}^s ; \mathcal{Y}^s)} \le C \left(\max \{ b^{\frac{1}{2}}, b^{\frac{7}{2}} \} \abs{\kappa}^2  + \left(\abs{\gamma} \max\{b^{1/2}, b^{3/2}\} + \max \{1,b^{2} \} \right) \abs{\kappa}\right). 
		\end{align}
		\item There exists $\kappa_0 > 0$ depending on $\gamma,s,b,n$, and in the toroidal cases on $L_i$, such that for all $\kappa \in (-\kappa_0, \kappa_0)$, the map $L_{\kappa,\sigma} : \mathcal{X}^s \to \mathcal{Y}^s$ as defined in \eqref{eq:Lkappa} is an isomorphism.
	\end{enumerate}
 \end{thm}
\begin{proof}
	To prove the first item, we first suppose that $(u,p,\eta) \in \mathcal{X}^{s}$ for some integer $s \ge 0$ and consider $(f,g,h,k) = M_{\kappa} (u,p,\eta)$, where $M_\kappa$ is defined in \eqref{eq:defMkappa}. Then by the fifth item of Theorem~\ref{thm: X^s gen}, the third item of Theorem~\ref{thm:Xs}, and Lemma~\ref{lem:simpleprod} we have $(f,g,h,k) = M_{\kappa} ( u, p,\eta) \in \mathcal{Y}^{s}$. Recall that $s_0(x_n) = bx_n - \frac{x_n^2}{2}$ and note that $s_0^{(i)}(x_n) = 0$ for $i \ge 3$ and $(x_n s_0(x_n))^{(i)} = 0$ for $i \ge 4$. Therefore, by Lemma~\ref{lem:simpleprod} and the fifth item of Theorem~\ref{thm: X^s gen} we have 
\begin{multline} 
	\norm{f}_{H^s} \lesssim_s  \abs{\gamma} \abs{\kappa} \max_{0 \le i \le 1} \norm{x_n^{(i)}}_{L^2} \norm{\p_1 \eta}_{H^s}  +\abs{\kappa} \max_{0 \le i \le 2} \norm{s_0^{(i)}}_{L^\infty} \norm{u}_{H^{s+1}} + \abs{\kappa} \max_{1 \le i \le 2} \norm{s_0^{(i)}}_{L^\infty} \norm{u}_{H^s} \\+ \abs{\kappa}^2 \max_{0 \le i \le 3} \norm{(x_n s_0)^{(i)}}_{L^2} \norm{\p_1 \eta}_{H^s} 
			      + \abs{\kappa}  \max_{0 \le i \le 1} \norm{x_n^{(i)}}_{L^2} \norm{\Delta' \eta}_{H^s}  + \abs{\kappa} \max_{0 \le i \le 1} \norm{x_n^{(i)}}_{L^2} \norm{\p_1 \eta}_{H^{s+1}} \\+ \abs{\kappa} \norm{\p_1 \eta}_{H^s}  
			      \lesssim_{s,n} \left[ \max \{ b^{\frac{1}{2}}, b^{\frac{7}{2}} \} \abs{\kappa}^2  + \left(\abs{\gamma} \max\{b^{1/2}, b^{3/2}\} + \max \{1,b^{2} \} \right) \abs{\kappa} \right] \left( \norm{u}_{H^{s+1}} + \norm{\eta}_{X^{s+2}} \right) \\
				  \lesssim_{s,n} \left(  \max \{ b^{\frac{1}{2}}, b^{\frac{7}{2}} \} \abs{\kappa}^2  + \left(\abs{\gamma} \max\{b^{1/2}, b^{3/2}\} + \max \{1,b^{2} \} \right) \abs{\kappa}  \right) \norm{(u,p,\eta)}_{\mathcal{X}^s}.
\end{multline}
 We also have 
 \begin{align} 
	 \norm{g}_{H^{s+1}} \lesssim_s \abs{\kappa} \max \{ b^{\frac{1}{2}}, b^{\frac{3}{2}} \} \norm{\p_1 \eta}_{H^{s+1}} \lesssim_{s,n} \abs{\kappa} \max \{ b^{\frac{1}{2}}, b^{\frac{3}{2}} \} \norm{ \eta}_{X^{s+2}}
 \end{align}
 and $\norm{h}_{H^{s+\frac{3}{2}}} \lesssim_{s,n} \abs{\kappa} b^2 \norm{\eta}_{X^{s+\frac{5}{2}}}$. Therefore, 
 \begin{multline} 
	 \norm{M_\kappa(u,p,\eta)}_{\mathcal{Y}^s} \lesssim_{s,n} \left(\max \{ b^{\frac{1}{2}}, b^{\frac{7}{2}} \} \abs{\kappa}^2  + \left(\abs{\gamma} \max\{b^{1/2}, b^{3/2}\} + \max \{1,b^{2} \} \right) \abs{\kappa}  \right) \norm{(u,p,\eta)}_{\mathcal{X}^s} \\ =: B(\gamma, b, \kappa)\norm{(u,p,\eta)}_{\mathcal{X}^s}
 \end{multline}
 which implies that $\norm{M_\kappa}_{\mathcal{L}(\mathcal{X}^s ; \mathcal{Y}^s)} \lesssim_{s,n}  B(\gamma, b, \kappa)$. We note that by the fifth item of Theorem~\ref{thm: X^s gen}, in the toroidal cases the universal constants in the preceding equations also depend on $L_i$.
 
 For any real valued $s \ge 0$, we may find an integer $m \ge 0$ and some $\theta \in [0,1]$ such that $s = \theta m + (1-\theta)(m+1)$. Then by interpolation, there exists some constant $C > 0$ depending on $s,n$, and in the periodic cases on $L_i$ such that
 \begin{align} \label{eq:kappa C}
	 \norm{M_\kappa}_{\mathcal{L}(\mathcal{X}^s ; \mathcal{Y}^s)} \le \norm{M_\kappa}_{\mathcal{L}(\mathcal{X}^m ; \mathcal{Y}^m)}^{\theta} \norm{M_\kappa}_{\mathcal{L}(\mathcal{X}^{m+1} ; \mathcal{Y}^{m+1})}^{1-\theta}
	 \le C B(\gamma,b,\kappa). 
 \end{align}
 This proves the first item. To prove the second item, we note that 
 \begin{align} \label{eq: kappa cases}
	B(\gamma,b,\kappa) = \max \{ b^{\frac{1}{2}}, b^{\frac{7}{2}} \} \abs{\kappa}^2  + \left(\abs{\gamma} \max\{b^{1/2}, b^{3/2}\} + \max \{1,b^{2} \} \right) \abs{\kappa} = \begin{cases}
		 b^{7/2} \abs{\kappa}^2 + (\abs{\gamma} b^{3/2} + b^2) \abs{\kappa}, & b > 1 \\
		 b^{1/2} \abs{\kappa}^2 + (\abs{\gamma} b^{1/2} + 1) \abs{\kappa}, & 0 < b \le 1
	\end{cases}.
 \end{align}
From \eqref{eq: kappa cases}, we may infer that for $c = C^{-1} \norm{L_{0,\sigma}}^{-1} > 0$, where $C > 0$ is the constant appearing in \eqref{eq:kappa C}, there exists $\kappa_0$ depending on $\gamma,s,b,n$ and in the toroidal cases on $L_i$ for which $B(\gamma,b,\kappa) < c$ if $\kappa \in (-\kappa_0,\kappa_0)$.  

Then by \eqref{eq:kappa C}, if $\kappa \in (-\kappa_0, \kappa_0)$ then $\norm{M_\kappa}_{\mathcal{L}(\mathcal{X}^s, \mathcal{Y}^s)} < \norm{L_{0,\sigma}}_{\mathcal{L}\left( \mathcal{Y}^s ; \mathcal{X}^s \right) }^{-1}$, thus the linear map $L_{0,\sigma}^{-1} M_{\kappa}: \mathcal{X}^s \to \mathcal{X}^s$ as defined in \eqref{eq:Lkappa} satisfies $\norm{ L_{0,\sigma}^{-1} M_\kappa } \le \norm{L_{0,\sigma}^{-1}} \norm{M_\kappa} <  1$.  Hence, $I +  L_{0,\sigma}^{-1} M_\kappa : \mathcal{X}^s \to \mathcal{X}^s$ is an isomorphism, and we conclude that $L_{\kappa,\sigma} = L_{0,\sigma} (I+ L_{0,\sigma}^{-1} M_\kappa)$ is an isomorphism from $\mathcal{X}^s$ to $\mathcal{Y}^s$. 
\end{proof}

\section{Nonlinear analysis} \label{nonlinear}

\subsection{Preliminaries}\label{sec: product est}

We first record a set of results on various maps defined in terms of $\eta$, including the flattening map $\mathfrak{F}_\eta$, that we will use in the subsequent analysis. 
\begin{thm}\label{thm: flattening maps est}
Let $\Sigma$ be defined as in \eqref{sigma}, $\N \ni n \ge 2$, $\N \ni k > \frac{n}{2}$, and $V$ be a real finite dimensional inner product space. 

\begin{enumerate} 
	\item Let $\upzeta \in C_b^{0,1}(\Sigma; \R)$ such that $\inf \upzeta > 0$. Then there exists $r_1 > 0$ depending on $n,b,s,\upzeta$, and in the toroidal cases on $L_i$, such that the maps $\Gamma_1, \Gamma_2 : B_{X^s}(0,r_1) \times H^s(\Omega_\upzeta; \R) \to H^s(\Omega_\upzeta; \R)$ given by 
	\begin{align}
		 \Gamma_1 (f,g) = \frac{g}{b+f}, \quad \Gamma_2(f,g) = \frac{gf}{b+f}
	\end{align}
	are well-defined and smooth. There also exists a constant $r_2 > 0$ depending on $n,s$, and in the toroidal cases on $L_i$, such that the map $\Gamma : B_{X^s(0,r_2)} \to H^{s}(\Omega ; \R^n)$ given by 
	\begin{align}
		 \Gamma(f) = \frac{f}{\sqrt{1 + \abs{f}^2}}
	\end{align}
	is well-defined and smooth. 
\item (The flattening map $\mathfrak{F}_\eta$ and its inverse) Let $\eta \in X^{k+\frac{5}{2}}(\Sigma ; \R)$ be such that $\norm{\eta}_{C^0_b} \le \frac{b}{2}$. Define $\mathfrak{G}_\eta:\Omega_{b + \eta} \to \Omega_b$ via 
\begin{align} 
	\mathfrak{G}_\eta(x) = \left( x', \frac{x_n b}{b+\eta(x')} \right).
\end{align}
Then the following hold. 
\begin{enumerate}
	\item $\mathfrak{G}_\eta \in C^r(\Omega_{b+\eta}; \Omega_b)$ is a diffeomorphism for $r = 3 + \left\lfloor k - \frac{n}{2} \right\rfloor$, with its inverse being $\mathfrak{F}_\eta \in C^r (\Omega ; \Omega_{b+\eta})$. 
	\item If $0 \le s \le k +2$ and $F \in H^s(\Omega_b ; V)$, then $F \circ \mathfrak{G}_\eta \in H^s(\Omega_{b+\eta} ; V)$. Moreover, there exists a constant $c > 0$ depending on $n,s,k,\norm{\eta}_{X^{k+2}}$, and in the toroidal cases on $L_i$, such that 
		\begin{align} 
			\norm{F \circ \mathfrak{G}_\eta}_{H^s(\Omega_{b + \eta}; V)} \le c \norm{F}_{H^s(\Omega_b; V)},
		\end{align}
		and the map $r \mapsto c(n,s,k,r)$ is non-decreasing. 
	\item If $0 \le s \le k+2$ and $F \in H^s(\Omega_{b +\eta}; V)$, then $F \circ \mathfrak{F}_\eta \in H^s(\Omega_b ; V)$. Moreover, there exists a constant $c > 0$ depending on $n,s,k,\norm{\eta}_{X^{k+2}}$, and in the toroidal cases on $L_i$, such that
		\begin{align} 
			\norm{F \circ \mathfrak{F}_\eta}_{H^s(\Omega_b; V)} \le c \norm{F}_{H^s(\Omega_{b+\eta} ; V)},
		\end{align}
		and the map $r \mapsto c(n,s,k,r)$ is non-decreasing. 
\end{enumerate}
	\item ($\omega$-lemma for compositions) For $\eta \in X^{k+\frac{1}{2}}(\Sigma_b; \R)$ we define the flattening map $\mathfrak{F}_\eta$ as in \eqref{eq:flattening}. Then there exists some $0 < \delta_* < 1$ such that the following hold:
	\begin{enumerate}
		\item The map $\Lambda_\Omega: H^{k+1}(\Sigma \times \R; \R^n) \times B_{X^{k+\frac{1}{2}}} (0,\delta_*) \to H^k (\Omega_b ; \R^d)$ given by 
			\begin{align} 
				\Lambda_\Omega (f,\eta) = f \circ \mathfrak{F}_\eta
			\end{align}
			is well-defined and $C^1$, with $D\Lambda_\Omega (f,\eta) (g, \zeta) = \frac{\varphi }{b} \p_n f \circ \mathfrak{F}_\eta \zeta + g \circ \mathfrak{F}_\eta$. 
		\item The map $\mathfrak{S}_b : H^{k+2} (\Sigma \times \R; \R^n) \times B_{X^{s+\frac{3}{2}}} (0,\delta_*) \to H^{s+\frac{1}{2}}(\Sigma_b ; \R^d)$ defined via 
			\begin{align} \label{eq:S_b}
				\mathfrak{S}_b (f,\eta) = f \circ \mathfrak{F}_\eta \rvert_{\Sigma_b}
			\end{align}
			is well-defined and $C^1$, with $D \mathfrak{S}_b (f,\eta) (g,\zeta) \left( \frac{\varphi}{b} \p_n f \circ \mathfrak{F}_\eta \zeta + g \circ \mathfrak{F}_\eta \right)\rvert_{\Sigma_b}$.
			\item The map $\mathfrak{T}_b: H^{k+2}(\Sigma \times \R; \R^{n \times n}_{\sym}) \times B_{X^{k+5/2}}(0,\delta_*) \to H^{k+1/2}(\Sigma_b ; \R^d)$ defined via
			\begin{align}
				\mathfrak{T}_b (\mathcal{T},\eta) = (\mathcal{T} \circ \mathfrak{F}_\eta \rvert_{\Sigma_b})\mathcal{N},
			\end{align}
			where $\mathcal{N}$ is defined via \eqref{eq:N}, is well-defined and $C^1$.
	\end{enumerate}
\end{enumerate}
\end{thm}

\begin{proof}
	The three items in the case $\Sigma = \R^{d}$ for $d \ge 1$ follow from Theorems 5.16, 5.17, A.12, Corollary 5.21, and Proposition 7.4 in \cite{leonitice}; the proofs therein can be adopted with minimal modification to handle to cases when $\Sigma \neq \R^d$. 
\end{proof}

Now we can synthesize the aforementioned results to show that all the nonlinear maps appearing in \eqref{eq: main flattened} are well-defined and smooth. 

\begin{thm}\label{thm:smoothnessnonlinear}
	Let $\R \ni s > \frac{n}{2}$ and $\sigma \ge 0$. For $\delta > 0$, define the open set of $\mathcal{X}^s$
	\begin{align} 
		U^s_\delta = \{ (u, p ,\eta) \in \mathcal{X}^s \mid \norm{\eta}_{X^{s+\frac{5}{2}}} < \delta\}.
	\end{align}
    For $\gamma \in \R, T \in H^{s+\frac{1}{2}}(\Sigma ; \R^{n\times n}_{\sym})$, $(u,p,\eta) \in U^s_\delta$, $U^1 = W_1 \circ \mathfrak{F}_\eta, U_2 \circ = W_2 \circ \mathfrak{F}_\eta$ with $W_1, W_2$ as defined in \eqref{eq:modU1}, \eqref{eq:modU2}, $\mathfrak{F}_\eta$ as defined in \eqref{eq:flattening}, and $\mathcal{J}, \mathcal{A}, \mathcal{N}, \mathcal{H}$ as defined in \eqref{eq:JandK}, \eqref{eq:A}, \eqref{eq:N} and \eqref{eq:H}, we define $f = \diverge_\mathcal{A} S_\mathcal{A} (p, u) - \gamma  e_1 \cdot  \nabla_\mathcal{A} u  - \gamma \kappa (x_n + \eta \frac{x_n}{b})\p_1 \eta e_1+ \left( u + U^1 + U^2\right) \cdot  \nabla_\mathcal{A} \left( u + U^1 + U^2 \right) + (\nabla' \eta,0) - \kappa \left( x_n + \eta \frac{x_n}{b}\right) \Delta' \eta e_1   - \kappa ( (x_n + \eta \frac{x_n}{b}) \nabla' \p_1 \eta, \p_1 \eta)$, $g = \mathcal{J}[	\diverge_\mathcal{A} u + \kappa \left( x_n + \eta(x') \frac{x_n}{b} \right) \p_1 \eta(x')]$, $h = u\cdot \mathcal{N} - ( - \gamma + s(\eta + b) + \kappa (\eta + b) \eta ) \p_1 \eta$, and $k = S_\mathcal{A}(p, u) \mathcal{N} - [ \sigma \mathcal{H}(\eta) I + T + \kappa (\eta + b)( e_1 \otimes (\nabla'\eta, 0 ) + (\nabla' \eta,0) \otimes e_1) ]\mathcal{N}$. Then $(f,g,h,k) \in \mathcal{Y}^s$, and the map 
	\begin{align}\label{eq:fghkmap}
		\R  \times \R \times H^{s+\frac{1}{2}}(\Sigma; \R^{n\times n}_{\sym}) \times U^s_\delta  \ni (\gamma, T, u, p , \eta)  \mapsto (f,g,h,k) \in \mathcal{Y}^s
	\end{align}
	is smooth. 
\end{thm}

\begin{proof}
By the first item in Theorem~\ref{thm:Xs} and standard Sobolev embedding, there exists a constant $\delta_0 > 0$ depending on $n,s,b$, and in the toroidal cases on $L_i$, such that if $s > n/2$ and $\eta \in X^{s+5/2}(\Sigma; \R)$ with $\norm{\eta}_{X^{s+5/2}}< \delta_0$, then $\norm{\eta}_{C^0_b} \le b/2$. We define $\delta = \min \{ \delta_0, r_1, r_2/c_1\}$, where $r_1, r_2$ are the radii from the first item of Theorem~\ref{thm: flattening maps est}, and $c_1$ is the embedding constant from \eqref{eq: grad Xs}.
By Theorem 7.3 in \cite{leonitice} and standard results from the theory of Sobolev spaces, the map 
\begin{multline}
	\R  \times \R \times H^{s+\frac{1}{2}}(\Sigma ; \R^{n\times n}_{\sym}) \times U^s_\delta  \ni (\gamma, T, u, p , \eta)  \mapsto \\
	(\diverge_\mathcal{A} S_\mathcal{A}(p,u) + (u- \gamma e_1)\cdot \nabla_\mathcal{A} u  + u \cdot \nabla_{\mathcal{A}} u, \mathcal{J}\diverge_{\mathcal{A}}u, u \cdot \mathcal{N} + \gamma\p_1 \eta, S_{\mathcal{A}}(p,u) \mathcal{N} - (\sigma \mathcal{H}(\eta)I + T ) \mathcal{N}  ) \in \mathcal{Y}^s 
\end{multline}
is well-defined and smooth, so it suffices to show that the map
\begin{multline}
	\R  \times H^{s+\frac{1}{2}}(\Sigma ; \R^{n\times n}_{\sym}) \times U^s_\delta  \ni (\gamma, T, u, p , \eta)  \mapsto  \\
	(-\gamma \kappa  \mathfrak{F}_\eta(x',x_n)e_n  \p_1 \eta e_1 + u \cdot \nabla_\mathcal{A} (U^1 + U^2) + (U^1 + U^2) \cdot \nabla_{\mathcal{A}}(u + U^1 + U^2) \\+ (\nabla' \eta, 0)- \kappa \mathfrak{F}_\eta(x',x_n)e_n \Delta' \eta(x') e_1 - \kappa(\mathfrak{F}_\eta(x',x_n)e_n\p_1 \nabla' \eta, \p_1),
	 \mathcal{J} \kappa ( x_n + \eta(x') \frac{x_n}{b}) \p_1 \eta(x'), \\
	 -s(\eta+b)\p_1 \eta -\kappa (\eta+b)\eta \p_1 \eta, 
	 \; \kappa (\eta + b) \left[e_1 \otimes (\nabla'\eta, 0 )  + (\nabla' \eta,0) \otimes e_1 \right]\mathcal{N} ) \in \mathcal{Y}^s
\end{multline}
is well-defined and smooth, where $\mathfrak{F}_\eta(x',x_n)e_n = x_n + \eta(x') \frac{x_n}{b}$. By the fifth item of Theorem~\ref{thm: X^s gen}, the second item of Theorem~\ref{thm:Xs}, and Corollary~\ref{cor: mult xn} the map 
\begin{multline}
	\R \times \R \times X^{s+5/2}(\Sigma; \R) \ni (\gamma, \kappa,\eta) \mapsto -\gamma \kappa \mathfrak{F}_\eta (x',x_n)e_n \p_1 \eta e_1 + (\nabla' \eta, 0) \\- \kappa \mathfrak{F}_\eta(x',x_n)e_n \Delta' \eta(x') e_1
	- \kappa(\mathfrak{F}_\eta(x',x_n)e_n\p_1 \nabla' \eta, \p_1 \eta) \in H^{s+1/2}(\Omega ; \R^n)
\end{multline}
is well-defined and smooth. By \eqref{eq:modU1}, \eqref{eq:modU2}, \eqref{eq:flattening}, and \eqref{eq:A}, we have
\begin{multline}\label{eq:udotu}
	 u \cdot \nabla_\mathcal{A} (U^1 + U^2) + (U^1 + U^2) \cdot \nabla_{\mathcal{A}}(u + U^1 + U^2) = \sum_{j,k=1}^n u_j \mathcal{A}_{jk} \p_k(U^1 + U^2)  + \sum_{k=1}^n (U_1^1 + U_1^2) \mathcal{A}_{1k} \p_k( u  + U^1  + U^2 ) \\
	 = \sum_{j=1}^{n-1} u_j \p_j (U^1 + U^2)   + \sum_{k=1}^{n-1} u_n \mathcal{A}_{n k} \p_k (U^1 + U^2)  + u_n\mathcal{K} \p_n (U^1 + U^2) + (U_1^1 + U_2^2) \p_1 (u + U^1 + U^2) \\ - (U_1^1 + U_1^1) x_n \mathcal{K} \frac{\p_1 \eta(x')}{b} \p_n(U^1 + U^2),
\end{multline}
where $\mathcal{K} = 1/\mathcal{J}$. By \eqref{eq:modU1}, \eqref{eq:modU2}, and \eqref{eq:flattening}, $(U^1 + U^2)(x,\eta) = \kappa (  b \mathfrak{F}_\eta(x)e_n -(\mathfrak{F}_\eta(x)e_n) ^2/2  +  \mathfrak{F}_\eta(x)e_n \eta(x')) e_1$ where $\mathfrak{F}_\eta(x)e_n = x_n (1+ \eta(x')/b) = x_n \mathcal{J}$, so we find that 
\begin{align}\label{eq:U1plusU2}
	 \p_k(U^1 + U^2)(x',x_n) = \begin{cases}
		\kappa  \mathcal{J} \frac{x_n(2b-x_n)}{b} \p_k \eta(x') e_1, & k \neq n \\
		 \kappa (b + x_n \mathcal{J} + \eta(x')) \mathcal{J}, & k = n.
	\end{cases}.
\end{align}
Thus, by the second item of Theorem~\ref{thm:Xs}, Corollary~\ref{cor: mult xn},  and the algebra properties of standard Sobolev spaces for $s > n/2$ the map $X^{s+5/2}(\Sigma;\R) \times H^s(\Omega; \R) \ni (\eta,\psi) \mapsto \psi\p_k(U^1 + U^2) \in H^s(\Omega; \R^n)$ is well-defined and smooth. By the first item of Theorem~\ref{thm: flattening maps est}, the map $B_{X^s}(0,\delta) \times H^s(\Omega; \R) \ni (\eta,\psi) \mapsto \psi \mathcal{K} \in H^s(\Omega; \R)$ is well-defined and smooth for any $\psi \in H^s(\Omega; \R^n)$. We also note that every non-trivial term in the components of \eqref{eq:udotu} is either a product of functions in $X^{s+5/2}(\Sigma ;\R)$ and functions in $H^s(\Omega; \R)$, or functions in $X^{s+5/2}(\Sigma ;\R)$ and derivatives of functions in $X^{s+5/2}(\Sigma ;\R)$. To summarize, by the observations made above, \eqref{eq:U1plusU2}, the second item of Theorem~\ref{thm:Xs}, Corollary~\ref{cor: mult xn}, and the first item of Theorem~\ref{thm: flattening maps est}, the map 
\begin{align}
	\R \times U^s_\delta \ni (\kappa, u,p,\eta) \mapsto  u \cdot \nabla_\mathcal{A} (U^1 + U^2) + (U^1 + U^2) \cdot \nabla_{\mathcal{A}}(u + U^1 + U^2) \in H^s(\Omega; \R^n)
\end{align}
is well-defined and smooth. Similarly, the maps
\begin{align}
	\R \times X^{s+5/2}(\Sigma; \R) \ni (\kappa,\eta) \mapsto \mathcal{J} \kappa ( x_n + \eta(x') \frac{x_n}{b}) \p_1 \eta(x') \in H^{s+3/2}(\Sigma;\R)
\end{align}
and
\begin{multline}
	 \R \times X^{s+5/2}(\Sigma ; \R) \ni (\kappa,\eta) \mapsto \\(s(\eta+b) \p_1\eta + \kappa(\eta+b) \eta \p_1\eta, \kappa(\eta+b)( e_1 \otimes (\nabla' \eta,0) + (\nabla'\eta, 0) \otimes e_1) \mathcal{N}) \in H^{3/2}(\Sigma;\R) \times H^{3/2}(\Sigma; \R^n )
\end{multline}
are also well-defined and smooth. Finally, it remains to show that 
\begin{align}
	-s(\eta+b)\p_1 \eta -\kappa (\eta+b)\eta \p_1 \eta - \int_0^b \mathcal{J} (\cdot,x_n) \kappa \left( x_n + \eta(\cdot) \frac{x_n}{b} \right) \p_1 \eta(\cdot) \; dx_n \in \dot{H}^{-1}(\Sigma; \R). 
\end{align}
By \eqref{eq:JandK}, we have
\begin{align}
	-s(\eta+b)\p_1 \eta -\kappa (\eta+b)\eta \p_1 \eta - \int_0^b \mathcal{J}(\cdot,x_n) \kappa \left( x_n + \eta(\cdot) \frac{x_n}{b} \right) \p_1 \eta(\cdot) \; dx_n = -\kappa (b^2 \p_1 \eta + b \p_1 \eta^2 + \frac{1}{3} \p_1 \eta^3).
\end{align}
Thus, by Theorem~\ref{thm: 1} and the third item of Theorem~\ref{thm:Xs} we have $b^2 \p_1 \eta + b \p_1 \eta^2 + \frac{1}{3} \p_1 \eta^3 \in \dot{H}^{-1}(\Sigma; \R)$. Therefore, we have shown that $(f,g,h,k) \in \mathcal{Y}^s$, and the map defined by \eqref{eq:fghkmap} is smooth. 
\end{proof}

\subsection{Solvability of the flattened system \eqref{eq: main flattened}}\label{sec: main flattened}
Now we are ready to construct solutions to \eqref{eq: main flattened} by using the implicit function theorem.

\begin{proof}[Proof of Theorem~\ref{thm:main1}]
We first consider the case with surface tension, $\sigma > 0$ and $n \ge 2$. Let $\delta$ be the minimum of the $\delta_1 > 0$ from Theorem~\ref{thm:smoothnessnonlinear} and $\delta_* > 0$ from the third item of Theorem~\ref{thm: flattening maps est}. Consider the open subset $U^s_\delta$ of $\mathcal{X}^s$ defined via
\begin{align} 
	U^s_\delta =\{ (u,p,\eta) \in \mathcal{X}^s \mid \norm{\eta}_{X^{s+\frac{5}{2}}  } < \delta \}. 
\end{align}
Using Proposition~\ref{prop:embed X} and standard Sobolev embedding, any open subset of $U^s_\delta$ containing $(0,0,0)$ satisfies the first assertion of the theorem. This proves the first item.

To prove the remaining items, we consider the Hilbert space
\begin{align} 
	\mathcal{E}^s = \R \times \R \times H^{s+2}(\Sigma \times \R ; \R^{n \times n}_{\sym}) \times H^{s+\frac{1}{2}}(\Sigma  ; \R^{n\times n}_{\sym}) \times H^{s+1}(\Sigma \times \R  ; \R^n) \times H^s(\Sigma  ; \R^n) 
\end{align}
and the solution map $\Xi : \mathcal{E}^s \times U^s_\delta \to \mathcal{Y}^s$ associated to \eqref{eq: main flattened} defined via 
\begin{multline} 
	\Xi(\gamma, \kappa,  \mathcal{T}, T, \mathfrak{f}, f, u,p, \eta) = (\diverge_\mathcal{A} S_\mathcal{A} (p, u) - \gamma e_1 \cdot \nabla_\mathcal{A} u - \gamma \kappa x_n \p_1 \eta e_1 + \left( u + U^1 + U^2\right) \cdot  \nabla_\mathcal{A} \left( u + U^1 + U^2 \right)   \\  + (\nabla' \eta, 0) -\kappa \mathfrak{F}_\eta(x)e_n \Delta' \eta e_1   - \kappa ( \mathfrak{F}_\eta(x)e_n  \nabla' \p_1 \eta, \p_1 \eta)  - \mathfrak{f} \circ \mathfrak{F}_\eta - L_{\Omega_b} f , \\
	\mathcal{J}[\diverge_\mathcal{A} u + \kappa \mathfrak{F}_\eta(x)e_n \p_1 \eta],u\cdot \mathcal{N} - ( - \gamma + s(\eta + b) + \kappa (\eta + b) \eta ) \p_1 \eta, \\
	S_\mathcal{A}(p, u) \mathcal{N} - [ \sigma \mathcal{H}(\eta)I + \mathcal{T} \circ \mathfrak{F}_\eta + S_b T\rvert_{\Sigma_b} + \kappa (\eta + b) (e_1 \otimes (\nabla'\eta,0 ) + (\nabla' \eta, 0) \otimes e_1)]\mathcal{N}),
\end{multline}
where $L_{\Omega_b} f(x) = f(x'), S_b T(x',b) = T(x')$. By Theorem~\ref{thm: flattening maps est}, Theorem~\ref{thm:smoothnessnonlinear}, and Lemmas A.10 and A.11 in \cite{leonitice} the map $\Xi$ is well-defined and $C^1$. 

By the product structure of $\mathcal{E}^s \times U^s_\delta$, we can define $D_1 \Xi : \mathcal{E}^s \times U^s_\delta \to \mathcal{L}(\mathcal{E}^s ; \mathcal{Y}^s )$ and $D_2 \Xi : \mathcal{E}^s \times U^s_\delta \to \mathcal{L}(\mathcal{X}^s; \mathcal{Y}^s )$ to be the derivatives of $\Xi$ with respect to $\mathcal{E}^s$ and $U^s_\delta$, respectively. Note that by the second item of Theorem~\ref{thm: flattening maps est}, we have $D_2\mathfrak{S}_b(0,0) = 0$ and $D_2\Lambda_\Omega(0,0) = 0$. Therefore, for any $\gamma \in \R$, $\Xi( \gamma,0, 0,0,0,0, 0,0,0) = (0, 0, 0, 0)$ since we also have $U^2 = 0$, $U^1 \cdot \nabla_\mathcal{A} U^1 = (U_1^1) \p_1 U^1 = 0$ for $\eta = 0$, and $D_2 \Xi (\gamma, 0, 0,0,0,0, 0,0,0) = L_{\kappa,\sigma}$ where $L_{\kappa,\sigma}$ is defined in \eqref{eq:Lkappa}. By Theorem~\ref{thm:M}, for every $\gamma_* \neq 0$ there exists some $\kappa_0 > 0$ for which $D_2 \Xi (\gamma, 0, 0,0,0,0, 0,0,0)$ is a linear isomorphism for every $\kappa \in (-\kappa_0 , \kappa_0)$. Thus, by the implicit function theorem, there exists open sets $\mathcal{U}(\gamma_*) \subseteq \mathcal{E}^s$ and $\mathcal{O}(\gamma_*) \subseteq U^s_\delta$ such that $(\gamma_*, 0, 0, 0, 0, 0) \in \mathcal{U}(\gamma_*)$ and $(0,0,0) \in \mathcal{O}(\gamma_*)$, and a $C^1$ and Lipschitz map $\varpi_{\gamma_*} : \mathcal{U}(\gamma_*) \to \mathcal{O}(\gamma_*) \subseteq U^s_\delta$ such that $\Xi(\gamma, \kappa, \mathcal{T}, T, \mathfrak{f}, f, \varpi_{\gamma_*}(\gamma,\kappa, \mathcal{T}, T, \mathfrak{f}, f) ) = (0, 0, 0, 0)$ for all $(\gamma, \kappa, \mathcal{T}, T, \mathfrak{f}, f) \in \mathcal{U}(\gamma_*)$. Moreover, $(u,p,\eta) = \varpi_{\gamma_*}(\gamma, \kappa, \mathcal{T}, T, \mathfrak{f}, f)$ is the unique solution to $\Xi(\gamma, \kappa, \mathcal{T}, T, \mathfrak{f}, f, u,p,\eta ) = (0, 0, 0, 0)$ in $\mathcal{O}(\gamma_*)$. 

Next, we define the open sets 
\begin{align}\label{eq: final open}
	 \mathcal{U}^s = \bigcup_{\gamma_* \in \R \setminus \{0\}}\mathcal{U}(\gamma_*) \subseteq \mathcal{E}^s \; \text{and} \; \mathcal{O}^s = \bigcup_{\gamma_* \in \R \setminus \{0\}} \mathcal{O}(\gamma_*) \subseteq U^s_\delta.
 \end{align}
 We note that by construction, $(\R \setminus \{0\}) \times \{0 \} \times \{ 0\} \times \{ 0 \} \times \{ 0 \} \times \{0 \} \subset \mathcal{U}^s$. Furthermore, for every 
$(\gamma, \kappa, \mathcal{T},T,\mathfrak{f},f) \in \mathcal{U}^s$, there exists a $\gamma_* \in \R \setminus \{0\}$ for which $(\gamma, \kappa, \mathcal{T},T,\mathfrak{f},f) \in \mathcal{U}(\gamma_*)$ and $(u,p,\eta) = \varpi_{\gamma_*}(\gamma, \kappa, \mathcal{T},T,\mathfrak{f},f) \in \mathcal{O}(\gamma_*)$. By the observation above and the implicit function theorem, the map $\overline{\varpi}: \mathcal{U}^s \to \mathcal{O}^s$ defined via  $\overline{\varpi}(\gamma, \kappa, \mathcal{T},T,\mathfrak{f},f) = \varpi_{\gamma_*}(\gamma, \kappa, \mathcal{T},T,\mathfrak{f},f)$, where $\gamma_* \in \R \setminus \{ 0 \}$ is such that $(\gamma, \kappa, \mathcal{T},T,\mathfrak{f},f) \in \mathcal{U}(\gamma_*)$, is well-defined, $C^1$, and locally Lipschitz. This proves the remaining items for $\sigma > 0$ and $n \ge 3$. 

To prove the remaining items in the case without surface tension and $n=2$, we argue along the same lines but use the second item of Theorem~\ref{thm:upsiloniso} instead of the first and use the isomorphism $L_{\kappa,0}$. 
\end{proof}

 \subsection{Solvability of the unflattened system \eqref{eq:main unflattened}} \label{sec: main unflattened}
 We now examine the solvability of the system \eqref{eq:main unflattened} in the original Eulerian coordinates.

 \begin{proof}[Proof of Theorem~\ref{thm:main2}]
Consider the map $\varpi: \mathcal{U}^s \to \mathcal{O}^s$ constructed in Theorem~\ref{thm:main1}. By Theorem~\ref{thm:main1}, for every $(\gamma, \kappa, \mathcal{T}, T,\mathfrak{f},f) \in \mathcal{U}^s$ the tuple $\mathcal{O}^s \ni (u,p,\eta) = \varpi(\gamma, \kappa, \mathcal{T}, T,\mathfrak{f},f)$ solves \eqref{eq: main flattened} classically. Since $(u,p,\eta) \in \mathcal{O}^s$, we have $\norm{\eta}_{C_b^0} \le b/2$, therefore by the second item in Theorem~\ref{thm: flattening maps est}, the flattening map $\mathfrak{F}_\eta$ and its inverse $\mathfrak{F}_\eta^{-1}$ are both $C^{\lfloor s - n/2 \rfloor + 3}$ diffeomorphisms. 
Now we fix $(\gamma, \kappa, \mathcal{T}, T,\mathfrak{f},f) \in \mathcal{U}^s$ and set $(u,p,\eta) = \varpi(\gamma, \kappa, \mathcal{T}, T,\mathfrak{f},f), v = u \circ \mathfrak{F}_\eta^{-1}$, and  $q = p \circ \mathfrak{F}_\eta^{-1}$. Then by the second item of Theorem~\ref{thm: flattening maps est} and Sobolev embedding, we have $v \in \Hzeros{s+2}(\Omega_{b+\eta} ; \R^n) \cap C^{\lfloor s - n/2 \rfloor + 2}(\Omega_{b+\eta} ; \R^n )$ and $q \in H^{s+1}(\Omega_{b+\eta}; \R) \cap C_b^{\lfloor s - n/2 \rfloor + 1}(\Omega_{b+\eta}; \R)$. Since $(\mathfrak{F}_\eta^{-1}(x))' = x'$, we have $(\mathfrak{f} \circ \mathfrak{F}_\eta + L_{\Omega_b} f) \circ \mathfrak{F}_\eta^{-1}(x) = \mathfrak{f}(x) + L_{\Omega_{b+\eta}}f(x)$ and $(\mathcal{T} \circ \mathfrak{F}_\eta \rvert_{\Sigma_{b}} + S_b T) \circ \mathfrak{F}^{-1}_\eta (x)= \mathcal{T}\rvert_{\Sigma_{b + \eta}} (x)+ S_{b+\eta} T(x)$ for all $x \in \Omega_{b+\eta}$. Thus, if $(u,p,\eta)$ is a solution tuple to \eqref{eq: main flattened} then $(v,q,\eta)$ is a solution tuple to \eqref{eq:main unflattened}. The last item follows from the fact that $\varpi$ is locally Lipschitz. 
\end{proof}

\appendix

\section{Analysis tools}\label{appendix}

\subsection{Permutations of Cartesian products}\label{appendix: permutation}

Here we collect a number of simple tools and bits of notation related to Cartesian products, especially products of groups. 

\begin{defn}
Suppose that $d \ge 1$ and that for $1 \le i \le d$ we have a commutative monoid $X_i$ with identity element $0_i \in X_i$. Further suppose that $E \subseteq \{1,\dotsc,d\}$ and $x \in X$.  We define $x_E \in X$ via 
\begin{equation}\label{eq: xi R}
 (x_E)_j = 
\begin{cases}
 x_j &\text{if }j \in E \\
 0_j & \text{if }j \notin E.
\end{cases}
\end{equation}
The utility of this notation is seen in the formula $x = x_E + x_{E^c}$, valid for any set $E \subseteq \{1,\dotsc,d\}$.
\end{defn}

We will often use this notation when $X = \Gamma$ or $X = \hat{\Gamma}$, where $\Gamma$ is of the form \eqref{eq: Gamma f}.  In this context, we introduce further notation.

\begin{defn}\label{defn: RT notation}
Let $\Gamma$ be as in \eqref{eq: Gamma f}. 
\begin{enumerate}
	 \item We define the sets $R_\Gamma, R_{\hat{\Gamma}}, T_\Gamma, T_{\hat{\Gamma}} \subseteq \{1,\dotsc,d\}$ via 
	 \begin{align}\label{eq: R}
		 R_\Gamma &= R_{\hat{\Gamma}} = \{ i \in \{1, \dots, d\} : \Gamma_i = \R \} =: \{ \mathfrak{r}_1 ,\dots , \mathfrak{r}_{\abs{R_\Gamma}} \} \\ \label{eq: T}
		 T_\Gamma &= T_{\hat{\Gamma}} = \{ i \in \{1, \dots, d\} : \Gamma_i = L_i \T \} = \{1,\dotsc,d\} \backslash R_{\Gamma} = \{1,\dotsc,d\} \backslash R_{\hat{\Gamma}}  =: \{ \mathfrak{t}_1 , \dots , \mathfrak{t}_{\abs{T_\Gamma}} \}, 
	 \end{align}
	 where we order these such that $\mathfrak{r}_1 < \dots < \mathfrak{r}_{\abs{R_\Gamma}}$  and $\mathfrak{t}_1 < \dots < \mathfrak{t}_{\abs{T_\Gamma}}$.  This allows us to write $\Gamma \ni x = x_{R_\Gamma} + x_{T_\Gamma}$ and   $\hat{\Gamma} \ni \xi = \xi_{R_{\hat{\Gamma}}} + \xi_{T_{\hat{\Gamma}}}$. 
	 \item We define the sets $\Gamma_{R},\Gamma_{T}, \hat{\Gamma}_{R}, \hat{\Gamma}_T$ via 
	 \begin{align}\label{eq: Gamma_R}
		  \Gamma_{R} = \{ x \in \Gamma \mid x_{T_\Gamma} = 0 \}, \; \Gamma_T = \{ x \in \Gamma \mid x_{R_\Gamma} = 0 \}, \; \hat{\Gamma}_{R} = \{ \xi \in \hat{\Gamma} \mid \xi_{T_{\hat{\Gamma}}} = 0 \}, \; \hat{\Gamma}_T = \{ \xi \in \hat{\Gamma} \mid \xi_{R_{\hat{\Gamma}}} = 0 \}.
	 \end{align}
	 It follows that we have the direct sum decompositions
	 \begin{align}\label{eq: direct sum}
		  \Gamma = \Gamma_R \oplus \Gamma_T, \; \hat{\Gamma} = \hat{\Gamma}_R \oplus \hat{\Gamma}_T.
	 \end{align}
\end{enumerate}
\end{defn}

At numerous points in the paper it is convenient to reorder the factors appearing in a Cartesian product.  We introduce notation for this and related ideas now.

\begin{defn}\label{defn: reordering notation}
Suppose that $X_1,\dotsc, X_d$ are sets and consider the Cartesian product $X = \prod_{i=1}^d X_i$.  Let $S_d$ denote the symmetric group of permutations of the indices $\{1,\dotsc,d\}$.
\begin{enumerate}
 \item For $g \in S_d$ we write $gX$ for the reordered Cartesian product $g X = \prod_{i=1}^d X_{g(i)}$.

 \item Let $g \in S_d$. We define the induced reordering map $\mathcal{P}_g : X \to gX$ via 
\begin{equation}\label{eq: reordering}
 \mathcal{P}_g(x) = \mathcal{P}_g(x_1,\dotsc,x_d) = (x_{g(1)},\dotsc,x_{g(d)}).
\end{equation}
This map is clearly a bijection.  Sometimes it will be convenient to include the domain in the notation, in which case we write $\mathcal{P}_{g,X}$ in place of $\mathcal{P}_g$.

 \item Given sets $E$ and $Y$, write $\mathcal{F}(E;Y) = \{f : E \to Y\}$.  Let $g \in S_d$ and $Y$ be a set.  We define the map $\sigma_g : \mathcal{F}(X;Y) \to \mathcal{F}(g X; Y)$ via  $\sigma_g f = f \circ \mathcal{P}_g^{-1}$.  Since $\mathcal{P}_g$ is invertible, the map $\sigma_g$ is as well, and $\sigma_g^{-1} : \mathcal{F}(g X; Y) \to \mathcal{F}( X; Y)$ is given by  $\sigma_g^{-1} f = f \circ \mathcal{P}_g$.  As above, we sometimes indicate the domain by writing $\sigma_{g,X}$ in place of $\sigma_g$.
\end{enumerate}
\end{defn}

It will be convenient to record the following obvious result as a lemma.

\begin{lem}\label{lem: permutation induces iso everything}
 Suppose that $X_1,\dotsc, X_d$ are sets and consider the Cartesian product $X = \prod_{i=1}^d X_i$.  Let $g \in S_d$ and let $\mathcal{P}_g$ be the induced reordering map defined in the second item in Definition~\ref{defn: reordering notation}. Then the following hold.
\begin{enumerate}
 \item If each $X_i$ is a group, then $\mathcal{P}_g$ is a group isomorphism.
 \item If each $X_i$ is a topological space and $X$ is endowed with the product topology, then $\mathcal{P}_g$ is a homeomorphism.
 \item If each $X_i$ is a metric space and $X$ is endowed with a product metric, then $\mathcal{P}_g$ is an isometry.
 \item If each $X_i$ is a $\sigma-$finite measure space and $X$ is endowed with the product measure, then $\mathcal{P}_g$ is a measure preserving map.
 \item If each $X_i$ is a smooth manifold and $X$ is endowed with the smooth product manifold structure, then $\mathcal{P}_g$ is a smooth diffeomorphism. 
\end{enumerate}
\end{lem}

Next we introduce the notion of the canonical reordering of $\Gamma$ and $\hat{\Gamma}$, where $\Gamma$ is as in \eqref{eq: Gamma f}.

\begin{defn}
Let $\Gamma$ be as in \eqref{eq: Gamma f}, and let $R_\Gamma, R_{\hat{\Gamma}}, T_\Gamma, T_{\hat{\Gamma}} \subseteq \{1,\dotsc,d\}$ be as in Definition \ref{defn: RT notation}.

\begin{enumerate}
	 \item We define the canonical permutation associated to $\Gamma$ to be $g_\ast \in S_d$ defined by 
	 \begin{align}\label{eq: canonical g}
		 g_\ast(i) =\begin{cases}
			\mathfrak{r}_i, &\text{if }  1\le i \le \abs{R_\Gamma} \\
			\mathfrak{t}_{i - \abs{R_\Gamma}}, &\text{otherwise}.
		 \end{cases} 
	\end{align} 
	We use this to define the canonical reorderings of $\Gamma$ and $\hat{\Gamma}$ via 
	\begin{align}\label{eq: canonical Gamma}
		\Gamma_\ast = g_*\Gamma 
		= \begin{cases}
			 \R^{\abs{R_\Gamma}} \times \prod_{i \in T_\Gamma} L_{g(i)} \T & \abs{R_\Gamma} \ge 1, \\
			\prod_{i \in T_\Gamma} L_{g(i)} \T, & R_\Gamma = \varnothing
		\end{cases} \;
	\text{and} \;
		\hat{\Gamma}_\ast = g_*\hat{\Gamma}
		= \begin{cases}
			 \R^{\abs{R_\Gamma}} \times \prod_{i \in T_\Gamma} L_{g(i)}^{-1} \Z & \abs{R_\Gamma} \ge 1, \\
			\prod_{i \in T_\Gamma} L_{g(i)}^{-1} \Z, & R_\Gamma = \varnothing.
		\end{cases}
	\end{align}
	Following the convention in Definition~\ref{defn: RT notation}, we may write 
	\begin{align}\label{eq: Gamma_* decomp}
		 \Gamma_* = (\Gamma_*)_R \oplus (\Gamma_*)_T =: \Gamma_{*,R} \oplus \Gamma_{*,T}, \; \hat{\Gamma}_* = (\hat{\Gamma}_*)_R \oplus (\hat{\Gamma}_*)_T : = \hat{\Gamma}_{*,R} \oplus \hat{\Gamma}_{*,T}.
	\end{align}
	The essential point here is that all toroidal or integer factors are moved to the end of the canonical reordered product with the real factors preceding. In turn, we define the canonical reordering maps
	\begin{equation}\label{eq: canonical pGamma}
	 \mathcal{P}_{\Gamma} := \mathcal{P}_{g_\ast,\Gamma} : \Gamma \to \Gamma_\ast  \text{ and }  \mathcal{P}_{\hat{\Gamma}} := \mathcal{P}_{g_\ast,\hat{\Gamma}} : \hat{\Gamma} \to \hat{\Gamma}_\ast
	\end{equation}
	as well as the induced maps
	\begin{equation}\label{eq: canonical sigma}
	 \sigma_\Gamma := \sigma_{g_\ast,\Gamma} : \mathcal{F}(\Gamma;Y) \to \mathcal{F}(\Gamma_\ast;Y) \text{ and }  \sigma_{\hat{\Gamma}} := \sigma_{g_\ast,\hat{\Gamma}} :\mathcal{F}(\hat{\Gamma};Y) \to \mathcal{F}(\hat{\Gamma}_\ast;Y)
	\end{equation}
	for any set $Y$, where $\mathcal{F}(X;Y)$ is defined as in the third item of Definition~\ref{defn: reordering notation}.
	\item In the case when $1 \in R_\Gamma$, we define the surjective map $\uppi_{\hat{\Gamma}_*}: \hat{\Gamma}_* \to \R^{\abs{R_{\hat{\Gamma}}}}$ via 
	\begin{align}\label{eq: surjection on pGamma}
		 \uppi_{\hat{\Gamma}_*} (\xi) &= \uppi_{\hat{\Gamma}_*} (\xi_{\mathfrak{r}_1}, \dots , \xi_{\mathfrak{r}_{\abs{R_\Gamma}}}, \xi_{\mathfrak{t}_1}, \dots, \xi_{\mathfrak{t}_{\abs{T_\Gamma}}}) = (\xi_{\mathfrak{r}_1}, \dots , \xi_{\mathfrak{r}_{\abs{R_\Gamma}}}) 
	\end{align}
	where $\mathfrak{r}_i, \mathfrak{t}_i$ are defined via \eqref{eq: R} and \eqref{eq: T}. We also define the injective map $l_{\hat{\Gamma}_*}: \R^{\abs{R_{\hat{\Gamma}}}} \to \hat{\Gamma}_*$ via $l_{\hat{\Gamma}_*}(\omega) = (\omega,0)$.
\end{enumerate}

\end{defn}

In Section~\ref{sec:functional}, we generalize many of results in \cite{leonitice} concerning the low-frequency behavior of functions belonging to the anisotropic Sobolev space $X^s(\Gamma; \F)$. To introduce the analogue of the low-frequency ball over $\hat{\Gamma}$, we introduce the low frequency cutoff $0 < r \le 1$ via
 \begin{align}\label{eq: low freq r}
	  r = \begin{cases}
		 1, & \Gamma = \R^d \\
		 \min \{1, \min_{i \in \{1, \dots , \abs{T_\Gamma}\}} L_i^{-1} \}, & \text{otherwise},
	  \end{cases}
 \end{align}
and we note that $\xi \in B_{\hat{\Gamma}}(0,r) \implies \xi_{T_\Gamma} = 0$. 

To simplify some key computations when $1 \in R_\Gamma$, we reduce the computation over the low-frequency ball $B_{\hat{\Gamma}}(0,r)$ in $\hat{\Gamma}$ to the low frequency ball $B_{\R^{|R_{\hat{\Gamma}}|}}(0,r)$ in $\R^{|R_{\hat{\Gamma}}|}$. We introduce the map to make this reduction in the following lemma.

\begin{lem}\label{rem: iso everything}
	Let $\Gamma$ be defined as in \eqref{eq: Gamma f} and assume $1 \in R_\Gamma$. We define the surjective map $\uppi_{\hat{\Gamma}}: \hat{\Gamma} \to \R^{\abs{R_{\hat{\Gamma}}}}$ via 
	\begin{align}\label{eq: uppi_R}
		\uppi_{\hat{\Gamma}} =  \uppi_{\hat{\Gamma}_\ast} \circ \mathcal{P}_{\hat{\Gamma}}.
	\end{align}
	The following hold.
\begin{enumerate}
	 \item The restriction of the map $\uppi_{\hat{\Gamma}_*}$ on $\hat{\Gamma}_{\ast,R}$ is an isometric measure-preserving group isomorphism between $\hat{\Gamma}_{\ast,R}$ and $\R^{|R_{\hat{\Gamma}}|}$, where $\hat{\Gamma}_{\ast,R}$ is the first summand in the direct sum decomposition of $\hat{\Gamma}_*$ in \eqref{eq: Gamma_* decomp}.
	 \item As a consequence of the first item and Lemma~\ref{lem: permutation induces iso everything}, the restriction of the map $\uppi_{\hat{\Gamma}}$ on $\hat{\Gamma}_R$ is an isometric measure-preserving group isomorphism between $\hat{\Gamma}_R$ and $\R^{|R_{\hat{\Gamma}}|}$, with the map $\mathcal{P}_{\hat{\Gamma}}^{-1} \circ l_{\hat{\Gamma}_\ast}$ as its inverse. 
	 \item The restriction of $\uppi_{\hat{\Gamma}}$ on $B(0,r) \subset \hat{\Gamma}_R$ is also an isometric measure-preserving group isomorphism between $B_{\hat{\Gamma}}(0,r)$ to $B_{\R^{|R_{\hat{\Gamma}}|}}(0,r)$, with the restriction of the map $\mathcal{P}_{\hat{\Gamma}}^{-1}  \circ  l_{\hat{\Gamma}_\ast}$ on $B_{\R^{|R_{\hat{\Gamma}}|}}(0,r)$ as its inverse. 
\end{enumerate}
\end{lem}
In other words, we have $\xi \in B_{\hat{\Gamma}}(0,r) \implies \uppi_{\hat{\Gamma}}  \xi \in B_{\R^{|R_{\hat{\Gamma}}|}}(0,r)$ and  $\omega \in B_{\R^{|R_{\hat{\Gamma}}|}}(0,r) \implies \mathcal{P}_{\hat{\Gamma}}^{-1}l_{\hat{\Gamma}_\ast} \omega \in B_{\hat{\Gamma}}(0,r)$, and the maps $\uppi_{\hat{\Gamma}} ,  \mathcal{P}_{\hat{\Gamma}}^{-1}l_{\hat{\Gamma}_\ast}$ are isometric measure-preserving group isomorphisms.

\subsection{Tempered distributions and the Fourier transform}\label{tempered}

In this subsection we carefully define the class of tempered distributions on $\Gamma$ and its Pontryagin dual $\hat{\Gamma}$. First, we note that for any $\alpha \in \N^d$, we may follow the convention in \eqref{eq: xi R} and write $\alpha = \alpha_{R_\Gamma} + \alpha_{T_\Gamma}$. We then define 
\begin{align}
	 \N^d_{R_\Gamma} = \{ \alpha \in \N^d \mid \alpha_{T_\Gamma} = 0\}, \; \N^d_{T_\Gamma} = \{ \alpha \in \N^d \mid \alpha_{R_\Gamma} = 0\}.
\end{align}
\begin{defn}\label{defn:schwartz}
Let $\Gamma$ be defined as in \eqref{eq: Gamma f}. 
\begin{enumerate}
\item We define the Schwartz class on $\Gamma$ to be 
	 \begin{align}
		  \mathscr{S}(\Gamma ; \C) = \{ f \in C^{\infty}(\Gamma; \C) \mid [f]_{\Gamma, \alpha,\beta} = \sup_{x \in \Gamma} \abs{ x^{\alpha }\p^{\beta} f(x)} < \infty \; \text{for all } \alpha \in \N^d_{R_\Gamma},\beta \in \N^d \},
	 \end{align}
	 and on $\hat{\Gamma}$ to be 
	 \begin{multline}
		\mathscr{S}(\hat{\Gamma} ; \C) = \{ f : \hat{\Gamma} \to \C \mid f(\mathcal{P}_{\hat{\Gamma}}^{-1} (\cdot, 0) + \xi_T)  \in C^{\infty}( \R^{\abs{R_{\Gamma}}} ; \C) \; \text{for all} \; \xi_{T} \in \hat{\Gamma}_{T_{\hat{\Gamma}}} \; \text{and} \; \\ [f]_{\hat{\Gamma}, \alpha, \beta} = \sup_{\xi \in \hat{\Gamma}}  \abs{\xi^\alpha \p^\beta f(\xi)} < \infty \; \text{for all} \; \alpha \in \N^{d}, \beta \in \N_{R_{\hat{\Gamma}}}^{d}\}.
	 \end{multline}
	 We equip $\mathscr{S}(\Gamma ; \C)$ and $\mathscr{S}(\hat{\Gamma} ; \C)$ with the Fr\'{e}chet topology induced by the countable family of seminorms $\{[\cdot]_{\Gamma,\alpha,\beta}\}_{\alpha \in \N^d_{R_\Gamma}, \beta \in \N^d}$ and $\{[\cdot]_{\hat{\Gamma},\alpha,\beta}\}_{\alpha \in \N^d, \beta \in \N_{R_{\hat{\Gamma}}}^d}$,  by endowing $\mathscr{S}(\Gamma ; \C)$ and $\mathscr{S}(\hat{\Gamma} ; \C)$ with the metrics
	 \begin{align}
		  d_{\mathscr{S}(\Gamma)}(f,g) = \sum_{\alpha \in \N^d_{R_\Gamma}}\sum_{\beta \in \N^d} \frac{1}{2^{\abs{\alpha} + \abs{\beta}}} \frac{[f-g]_{\Gamma,\alpha,\beta}}{1 + [f-g]_{\Gamma,\alpha,\beta}}, \; d_{\mathscr{S}(\hat{\Gamma})}(f,g) = \sum_{\alpha \in \N^d}\sum_{\beta \in \N_{R_{\hat{\Gamma}}}^d} \frac{1}{2^{\abs{\alpha} + \abs{\beta}}} \frac{[f-g]_{\hat{\Gamma},\alpha,\beta}}{1 + [f-g]_{\hat{\Gamma},\alpha,\beta}}.
	 \end{align}
	 
\item We define the unitary Fourier and inverse Fourier transforms $\mathscr{F}^{\pm}_\Gamma: \mathscr{S}(\Gamma ; \C) \to \mathscr{S}(\hat{\Gamma} ; \C)$, $\mathscr{F}^{\pm}_{\hat{\Gamma}}: \mathscr{S}(\hat{\Gamma} ; \C) \to \mathscr{S}(\Gamma; \C)$  on $\mathscr{S}(\Gamma; \C)$ and $\mathscr{S}(\hat{\Gamma}; \C)$ via 
\begin{multline}\label{eq: Fourier def}
	 \mathscr{F}^{\pm}_\Gamma \{ f \} (\xi) = \frac{1}{\sqrt{\prod_{i \in T_\Gamma} L_i}}\int_{\Gamma} f(x) e^{\mp 2\pi i \xi\cdot x} \; dx = \mathscr{F}_\Gamma^{\mp} \{ f \} (-\xi), \\
	\mathscr{F}^{\pm}_{\hat{\Gamma}} \{ f \}(x) = \frac{1}{\sqrt{\prod_{i \in T_{\Gamma}} L_i}} \int_{\hat{\Gamma}} f(\xi) e^{\mp 2\pi i \xi \cdot x}  \; d\xi = \mathscr{F}_{\hat{\Gamma}}^{\mp} \{f\} (-x). 
\end{multline}

\item We define the class of tempered distributions $\mathscr{S}'(\Gamma ; \C)$ and $\mathscr{S}'(\hat{\Gamma} ; \C)$ to be the set of continuous linear functionals on $\mathscr{S}(\Gamma ; \C)$ and  $\mathscr{S}(\hat{\Gamma} ; \C)$, respectively. 

\end{enumerate}
\end{defn}

The next lemma shows that the restriction of $\sigma_{g,\Gamma}$ on $\mathscr{S}(\Gamma ; \F)$ is an isometric isomorphism between $\mathscr{S}(\Gamma ; \F)$ and $\mathscr{S}(g \Gamma ; \F)$, where $\sigma_{g,\Gamma}$ is defined in the third item of Definition~\ref{defn: reordering notation}.

\begin{lem}\label{lem: schwartz iso}
Let $\Gamma$ be defined as in \eqref{eq: Gamma f}, $g \in S_d$ be a permutation, and let $g\Gamma$ be as defined in the first item of Definition~\ref{defn: reordering notation}. Then the map $\sigma_{g,\Gamma}: \mathscr{S}(\Gamma ; \F) \to \mathscr{S}(g\Gamma ; \F)$ defined in the third item of Definition~\ref{defn: reordering notation} is an isometric isomorphism, with its inverse given by the map $\sigma_{g,\Gamma}^{-1}: \mathscr{S}(g\Gamma ; \F) \to \mathscr{S}(\Gamma ; \F)$, also defined in the third item of Definition~\ref{defn: reordering notation}.
\end{lem}
\begin{proof}
For $g \in S_d$, let $\mathcal{P}_{g,\Gamma}: \Gamma \to g\Gamma, \mathcal{P}_{g,\N^d}: \N^d \to g\N^d$ be the induced reordering maps defined via \eqref{eq: reordering}. We note that for any $F \in \mathscr{S}(\Gamma ; \F), \alpha \in \N_{R_{g \Gamma}}^d$, and $\beta \in \N^d$ we have $F(\mathcal{P}_g^{-1} \cdot) \in C^\infty(g \Gamma ; \F)$ with 
\begin{align}\label{eq: same seminorm1}
	 [F(\mathcal{P}_g^{-1} \cdot )]_{g \Gamma, \alpha, \beta} = \sup_{x 
	 \in g \Gamma} \abs{x^\alpha \p_x^\beta F(\mathcal{P}_{g,\Gamma}^{-1} x)} = \sup_{y \in \Gamma} \abs{y^{\mathcal{P}_{g,\N^d}^{-1} \alpha} \p_y^{\mathcal{P}_{g,\N^d}^{-1}\beta} F(y)} = [F]_{\Gamma, \mathcal{P}_{g,\N^d}^{-1} \alpha, \mathcal{P}_{g,\N^d}^{-1} \beta} < \infty, 
\end{align}
and for any $G \in \mathscr{S}(g \Gamma; \F), \alpha \in \N^d_{R_\Gamma}$, and $\beta \in \N^d$ we have $G(\mathcal{P}_{g,\Gamma}^{-1} \cdot) \in C^\infty(\Gamma;\F)$ with 
\begin{align}\label{eq: same seminorm2}
	 [G(\mathcal{P}_{g,\Gamma}^{-1}  \cdot)]_{\Gamma, \alpha, \beta} = \sup_{x 
	 \in \Gamma} \abs{x^\alpha \p_x^\beta G(\mathcal{P}_{g,\Gamma}^{-1}  x)} = \sup_{y \in g \Gamma} \abs{y^{\mathcal{P}_{g,\N^d} \alpha} \p _y^{\mathcal{P}_{g,\N^d}\beta} G(y)} = [G]_{g \Gamma, \mathcal{P}_{g,\N^d} \alpha, \mathcal{P}_{g,\N^d}\beta} < \infty.
\end{align}
This shows that $\sigma_{g,\Gamma}$ and $\sigma_{g,\Gamma}^{-1}$ are well-defined. Since $\sigma_{g,\Gamma}$ is a bijection, and $\mathcal{P}_{g,\N^d}$ is a bijection between $\N^{d}_{R_\Gamma}$ and $\N^{d}_{R_{g \Gamma}}$, and between $\N^d$ and $g \N^d = \N^d$, with \eqref{eq: same seminorm1} and \eqref{eq: same seminorm2} we can also conclude that 
\begin{align}
	d_{\mathscr{S}(\Gamma)}(F,G) = d_{\mathscr{S}(g \Gamma)}(F(\mathcal{P}_g^{-1}\cdot), G(\mathcal{P}_g^{-1}\cdot)).
\end{align}
This shows that $\sigma_{g,\Gamma}$ is an isometric isomorphism. 
\end{proof}

Following the exact same argument as above we can conclude that the restriction of $\sigma_{g,\hat{\Gamma}}$ on $\mathscr{S}(\hat{\Gamma} ;\F)$ is an isometric isomorphism between $\mathscr{S}(\hat{\Gamma} ;\F)$ and $\mathscr{S}(\hat{\Gamma}_\ast ;\F)$.

Next, we record some important properties of the Schwartz classes $\mathscr{S}(\Gamma ; \F)$ and $\mathscr{S}(\hat{\Gamma} ; \F)$. 

\begin{prop}\label{prop:schwartz}
Let $\Gamma$ be defined as in \eqref{eq: Gamma f}. Then the following hold.
\begin{enumerate}
	 \item We have the continuous inclusions $\mathscr{S}(\Gamma ; \C) \hookrightarrow L^p(\Gamma; \C)$ and $\mathscr{S}(\hat{\Gamma} ; \C) \hookrightarrow L^p(\hat{\Gamma}; \C)$ for any $1 \le p \le \infty$. 
	 \item We have the continuous inclusion $L^p(\Gamma; \C) \hookrightarrow \mathscr{S}'(\Gamma ; \C)$ and $L^p(\hat{\Gamma}; \C) \hookrightarrow \mathscr{S}'(\hat{\Gamma} ; \C)$ for all $1 \le p \le \infty$. 
	 \item For any $f \in \mathscr{S}(\Gamma; \C)$, $\alpha \in \N^d_{R_\Gamma}$, and $\beta \in \N^d$ the functions $g_{\alpha,\beta}, h_{\alpha,\beta} : \Gamma \to \C$ defined via $g_{\alpha,\beta}(x) = x^{\alpha} \p^\beta f(x)$ and $h_{\alpha, \beta}= \p^{\alpha} (x^\beta f(x))$ belong to $\mathscr{S}(\Gamma ; \C)$. Furthermore, the maps $f \mapsto g_{\alpha, \beta}$ and $f \mapsto h_{\alpha, \beta}$ are continuous. Similarly, for any $f \in \mathscr{S}(\hat{\Gamma}; \C)$, $\alpha \in \N^d$, and $\beta \in \N_{R_{\hat{\Gamma}}}^d$ the functions $g_{\alpha,\beta}, h_{\alpha,\beta} : \Gamma \to \C$ defined via $g_{\alpha,\gamma}(\xi) = \xi^{\alpha} \p^\beta f(\xi)$ and $h_{\alpha, \beta} = \p^{\alpha} (\xi^\beta f(\xi))$ belong to $\mathscr{S}(\hat{\Gamma} ; \C)$. Furthermore, the maps $f \mapsto g_{\alpha, \beta}$ and $f \mapsto h_{\alpha, \beta}$ are continuous.

	 \item If $f \in \mathscr{S}(\Gamma ; \C)$, then $\mathscr{F}^{\pm}_{\Gamma} \{f\} \in \mathscr{S}(\hat{\Gamma}; \C)$ and the map $f \mapsto \mathscr{F}_{\Gamma}^{\pm} \{f\}$ is continuous. Likewise, if $f \in \mathscr{S}(\hat{\Gamma} ; \C)$, then $\mathscr{F}^{\pm}_{\hat{\Gamma}} \{f\} \in \mathscr{S}(\Gamma ; \C)$ and the map $f \mapsto \mathscr{F}_{\hat{\Gamma}}^{\pm} \{f\}$ is continuous.

	 \item If $f \in \mathscr{S}(\Gamma; \C)$ and $g \in \mathscr{S}(\hat{\Gamma} ; \C)$, then
	 \begin{align}
		  \int_{\hat{\Gamma}} \mathscr{F}^{\pm}_{\Gamma} \{ f \}(\xi, k) g (\xi, k) \; d\xi d k = \int_\Gamma f(x,y) \mathscr{F}^{\pm}_{\hat{\Gamma}} \{g\} (x,y) \; dx dy.
	 \end{align}

	 \item  If $f \in \mathscr{S}(\Gamma ; \C)$, then $f = \mathscr{F}^{\mp}_{\hat{\Gamma}} \circ \mathscr{F}^{\pm}_{\Gamma} \{f\}$. Similarly, if $f \in \mathscr{S}(\hat{\Gamma}; \C)$, then $f = \mathscr{F}^{\mp}_{\Gamma} \circ \mathscr{F}^{\pm}_{\hat{\Gamma}} \{f\}$.
	 
	 \item The Fourier and inverse Fourier transforms $\mathscr{F}^{\pm}_{\Gamma}: \mathscr{S}(\Gamma; \C) \to \mathscr{S}(\hat{\Gamma} ; \C)$ on $\mathscr{S}(\Gamma; \C)$ are isomorphisms from $\mathscr{S}(\Gamma; \C)$ to $\mathscr{S}(\hat{\Gamma} ; \C)$. Likewise, the Fourier and inverse Fourier transforms $\mathscr{F}^{\pm}_{\hat{\Gamma}}: \mathscr{S}(\hat{\Gamma}; \C) \to \mathscr{S}( \Gamma; \C)$ on $\mathscr{S}(\hat{\Gamma}; \C)$ are isomorphisms from $\mathscr{S}(\hat{\Gamma}; \C)$ to $\mathscr{S}(\Gamma ; \C)$
\end{enumerate}
\end{prop}
\begin{proof}
Define
\begin{align}
	\tau = \begin{cases}
		1, & \Gamma = \R^d \\
		\max \{1, 2\max_{i \in \{1, \dots , d_2\}} L_i \}, & \text{otherwise}.
	 \end{cases}
\end{align}
To prove the first item, we note that for any $1 \le p < \infty$, $f \in \mathscr{S}(\Gamma; \C)$, and $\alpha, \beta \in \N^d_{R_\Gamma}$ with $\abs{\alpha} = d+1, \abs{\beta} = \lfloor (d+1)/p \rfloor +1$ we have 
\begin{multline}
	 \norm{f}_{L^p(\Gamma)} \lesssim  \left( \int_{B(0,\tau)} \abs{f(x)}^p \; dx + \int_{B(0,\tau)^c} d_{\Gamma}(x,0)^{\alpha} \abs{ f(x)}^p d_{\Gamma}(x,0)^{-\alpha} \; dx \right)^{1/p} \\ \lesssim \left( \norm{f}_{L^\infty}^p +  \left( \sup_{x \in B(0,\tau)^c } d_{\Gamma}(x,0)^{\alpha/p} \abs{ f(x)} \right)^p \int_{B(0,\tau)^c} d_{\Gamma}(x,0)^{-\alpha} \; dx \right)^{1/p} \\ \lesssim \norm{f}_{L^\infty} + \sup_{x\in B(0,\tau)^c} d_{\Gamma}(x,0)^{\beta} \abs{ f(x)} \lesssim [f]_{\Gamma, 0, 0} + \sum_{ \beta \in \N^d_{R_\Gamma}, \abs{\beta} = \lfloor (d+1)/p \rfloor +1} [f]_{\Gamma, \beta, 0},
\end{multline}
and when $p = \infty$ we have $\norm{f}_{L^\infty(\Gamma)} = [f]_{\Gamma, 0,0}$. It follows then that $\mathscr{S}(\Gamma ; \C) \hookrightarrow L^p(\Gamma ; \C)$ for all $1 \le p \le \infty$. A similar set of computations can be performed to prove $\mathscr{S}(\hat{\Gamma} ; \C) \hookrightarrow L^p(\hat{\Gamma}; \C)$ for $1 \le p \le \infty$. For the second item, suppose that we have a sequence $\{f_k\}_{k=1}^\infty \subset L^p(\Gamma;\C)$ for $1 \le p \le \infty$ such that $f_k \to f$ in $L^p(\Gamma; \C)$. We note that the canonical distribution $T_{f_k} : \mathscr{S}(\Gamma ; \C) \to \C$ associated to each $f_k$ defined via 
\begin{align}
	 \langle T_{f_k}, \varphi \rangle_{\mathscr{S'}, \mathscr{S}} = \int_{\Gamma} f_{k} \varphi \; dx \; \text{for all} \; \varphi \in \mathscr{S}(\Gamma; \C)
\end{align}
is clearly linear. By the first item and H\"{o}lder's inequality, if $\{ \varphi_m \}_{m=1}^\infty \subset \mathscr{S}(\Gamma; \C)$ and $\varphi_m \to \varphi$ as $m \to \infty$ in $\mathscr{S}(\Gamma ; \C)$, then for each $k$ we have 
\begin{align}
	\abs{\langle T_{f_k}, \varphi_m - \varphi \rangle_{\mathscr{S'}, \mathscr{S}}} \le \norm{f_k}_{L^p} \norm{\varphi_m \to \varphi}_{L^q} \to 0.
\end{align}
This implies that $\mathcal{T}_{f_k} \in \mathscr{S}'(\Gamma; \C)$. Furthermore, since $f_k \to f$ in $L^p(\Gamma; \C)$, by the first item and H\"{o}lder's inequality again, we have 
\begin{align}
	\abs{\langle T_{f_k} - T_f, \varphi \rangle_{\mathscr{S'}, \mathscr{S}}} \le \norm{f_k - f}_{L^p} \norm{\varphi}_{L^q} \to 0 \; \text{for all} \; \varphi \in \mathscr{S}(\Gamma; \C). 
\end{align}
Since the map $L^p(\Gamma ; \C) \ni f \mapsto T_f \in \mathscr{S}'(\Gamma; \C)$ is also clearly injective, the inclusion map $i : L^p(\Gamma; \C) \to \mathscr{S}'(\Gamma; \C)$ as is well-defined and continuous.
This proves the second item for $\mathscr{S}(\Gamma ; \C)$, and the same argument can be applied to prove the second item for $\mathscr{S}(\hat{\Gamma} ; \C)$. The third item follows directly from the Leibniz rule. 

For the remaining items, we note that by using the isometric isomorphisms $\sigma_\Gamma: \mathscr{S}(\Gamma; \F) \to \mathscr{S}(\Gamma_\ast; \F)$ and $\sigma_{\hat{\Gamma}}: \mathscr{S}(\hat{\Gamma}; \F) \to \mathscr{S}(\hat{\Gamma}_\ast; \F)$ induced by the canonical permutation $g_*$, it suffices to prove the remaining items for $\mathscr{S}(\Gamma_*; \F)$ and $\mathscr{S}(\hat{\Gamma}_\ast; \F)$. Indeed, if the remaining items are true for $\mathscr{S}(\Gamma_*; \F)$ and $\mathscr{S}(\hat{\Gamma}_\ast; \F)$, then it suffices to check that 
\begin{align}\label{eq: appendix fourier comp}
\mathscr{F}_{\Gamma}^{\pm} = \sigma_{\hat{\Gamma}}^{-1} \circ \mathscr{F}_{\Gamma_\ast}^{\pm}  \circ \sigma_{\Gamma}, \; \mathscr{F}_{\hat{\Gamma}}^{\pm} = \sigma_{\Gamma}^{-1} \circ \mathscr{F}_{\hat{\Gamma}_\ast}^{\pm} \circ \sigma_{\hat{\Gamma}}, 
\end{align}
and the remaining items also follow for $\mathscr{S}(\Gamma; \F)$ and $\mathscr{S}(\hat{\Gamma}; \F)$ as the following diagram commutes:
\begin{equation}
	\begin{tikzcd}
		\mathscr{S}(\Gamma; \C) \arrow{r}{\mathscr{F}^{\pm}_{\Gamma}} \arrow{d}{\sigma_\Gamma} & \mathscr{S}(\hat{\Gamma}; \C) \arrow{d}{\sigma_{\hat{\Gamma}}} \arrow{r}{\mathscr{F}^{\mp}_{\hat{\Gamma}}} & \mathscr{S}(\Gamma; \C) \arrow{d}{\sigma_{\Gamma}}\\
		\mathscr{S}(\Gamma_\ast; \C) \arrow{r}{\mathscr{F}^{\pm}_{\Gamma_\ast}} & \mathscr{S}(\hat{\Gamma}_\ast; \C) \arrow{r}{\mathscr{F}^{\mp}_{\hat{\Gamma}_\ast}} &
		\mathscr{S}(\Gamma_\ast ; \C)
		\end{tikzcd} 
	\end{equation}
To verify \eqref{eq: appendix fourier comp}, we note that for any $f \in \mathscr{S}(\Gamma; \C)$, we have 
\begin{multline}\label{eq: check commute}
	(\sigma_{\hat{\Gamma}}^{-1}\mathscr{F}_{\Gamma_\ast}^{\pm} \sigma_{\Gamma} f)(\xi) =  \sigma_{\hat{\Gamma}}^{-1}\mathscr{F}_{\Gamma_\ast}^{\pm} [f(\mathcal{P}_{\Gamma}^{-1}\cdot)](\xi) = \sigma_{\hat{\Gamma}}^{-1}  \frac{1}{\sqrt{\prod_{i \in T_\Gamma} L_i}}\int_{\Gamma_\ast} f(\mathcal{P}_{\Gamma}^{-1}x) e^{\mp 2\pi i \xi\cdot x} \; dx \\ 
	= \sigma_{\hat{\Gamma}}^{-1} \frac{1}{\sqrt{\prod_{i \in T_\Gamma} L_i}}\int_{\Gamma} f(x) e^{\mp 2\pi i \mathcal{P}_{\hat{\Gamma}}^{-1} \xi\cdot x} \; dx = \sigma_{\hat{\Gamma}}^{-1} \mathscr{F}_{\Gamma}^{\pm} [f(\cdot)](\mathcal{P}_{\hat{\Gamma}}^{-1}\xi) = \mathscr{F}_{\Gamma}^{\pm} [f(\cdot)](\xi).
\end{multline}
A similar set of calculations can be performed to verify that $\sigma_{\Gamma}^{-1} \circ \mathscr{F}_{\hat{\Gamma}_\ast}^{\pm} \circ \sigma_{\hat{\Gamma}}$. Consequently, for the rest of  the proof we may assume without loss of generality that $\Gamma = \R^{d_1} \times \prod_{i=1}^{d_2} L_i \T$, with $d = \dim \Gamma = d_1 + d_2$. 

To prove the fourth item, we first note that if $f \in \mathscr{S}(\Gamma ; \C)$, by the first and third items we have $\p^\alpha ((-2\pi i x)^\gamma f) \in \mathscr{S}(\Gamma ; \C) \subset L^1(\Gamma ; \C)$ for any $\alpha, \gamma \in \N^{d_1}$. This shows that $\mathscr{F}^{\pm}_{\Gamma}{f}(\cdot, k) \in C^\infty(\R^{d_1} ; \C)$ for any $k \in \prod_{i=1}^{d_2} L_i^{-1}\Z$. Then for any $\alpha, \gamma \in \N^{d_1}, \beta \in \N^{d_2}$, we have 
\begin{multline}
(-2\pi i \xi)^\alpha (-2\pi i  k)^\beta \p_{\xi}^\gamma \mathscr{F}^{\pm} \{f\}(\xi,k) = (-2\pi i \xi)^\alpha (-2\pi i  k)^\beta  \p_{\xi}^\gamma \int_{\Gamma} f(x,y) e^{ \mp 2\pi i \xi\cdot x} e^{\mp 2\pi i k \cdot y} \; dxdy \\ = \int_\Gamma f(x,y) (\mp 2\pi i x)^\gamma \p_x^{\alpha} e^{\mp 2\pi i \xi \cdot x} \p_y^\beta e^{\mp 2\pi i k \cdot y} \; dx dy = (-1)^{\abs{\alpha} + \abs{\beta}} \int_{\Gamma} \p_x^\alpha[ (\mp 2\pi i x)^\gamma \p_y^\beta f(x,y)] e^{\mp 2\pi i \xi \cdot x} e^{\mp 2 \pi i k \cdot y} \; dx dy.
\end{multline}
By the first and third items, we find that the preceding inequality implies $\mathscr{F}^{\pm} \{f\} \in \mathscr{S}(\hat{\Gamma}; \C)$ and if $\{f_k\}_{k=1}^\infty \subset \mathscr{S}(\Gamma ; \C)$ and $f_k \to f$ in $\mathscr{S}(\Gamma; \C)$, then $\mathscr{F}_{\Gamma}^{\pm} \{ f_k \} \to \mathscr{F}_{\Gamma}^{\pm} \{ f\} $ in $\mathscr{S}(\hat{\Gamma} ; \C)$. Therefore, we can conclude that if $f \in \mathscr{S}(\Gamma ; \C)$ then $\mathscr{F}_{\Gamma}^{\pm} \{ f \} \in \mathscr{S}(\hat{\Gamma} ; \C)$, and the map $f \mapsto \mathscr{F}_{\Gamma}^{\pm} \{f\}$ is continuous. Likewise, by a similar calculation we find that if $f \in \mathscr{S}(\hat{\Gamma} ; \C)$ then $\mathscr{F}_{\hat{\Gamma}}^{\pm} \{ f \} \in \mathscr{S}(\Gamma ; \C)$, and the map $f \mapsto \mathscr{F}_{\hat{\Gamma}}^{\pm} \{f\}$ is continuous.
To prove the fifth item, we note that by the first item and Fubini's theorem, 
\begin{multline}
	\int_{\hat{\Gamma}} \mathscr{F}^{\pm}_{\Gamma} \{ f \}(\xi, k) g (\xi, k) \; d\xi d k =	\int_{\hat{\Gamma}} \left( \int_{\Gamma} f(x,y) e^{\mp 2\pi i \xi\cdot x} e^{\mp 2\pi i k \cdot y} \; dxdy \right)  g (\xi, k) \; d\xi d k \\ = \int_\Gamma f(x,y) \left( \int_{\hat{\Gamma}} g(\xi, k) e^{\mp 2\pi i \xi\cdot x} e^{\mp 2\pi i k \cdot y} \; d\xi dk \right) = \int_{\Gamma} f(x,y) \mathscr{F}^{\pm}_{\hat{\Gamma}} \{g\} (x,y) \; dx dy.
\end{multline}
To prove the sixth item, suppose $f \in \mathscr{S}(\Gamma ; \C)$ and for any fixed $t > 0$ and $(x,y) \in \Gamma$ we consider the function 
\begin{align}
	 \phi^{\pm}(\xi,k) =  e^{-\pi t^2 (\abs{\xi}^2 + \abs{k}^2)} e^{\pm 2\pi i \xi \cdot x } \frac{e^{\pm 2\pi i k \cdot y}}{\sqrt{\prod_{i=1}^{d_2} L_i}}.
\end{align}
We note that 
\begin{align}\label{eq: g conv}
	\mathscr{F}^{\pm}_{\hat{\Gamma}} \{ \phi^{\pm} \}  (r,s) = \mathscr{F}^{\pm}_{\hat{\Gamma}} \{ g_t \} (r-x, s-y),
\end{align}
where $g_t(\xi,k) = (\prod_{i=1}^{d_2} L_i)^{-1/2}e^{-\pi t^2 (\abs{\xi}^2 + \abs{k}^2)}$. Furthermore, we note that the dual lattice of $\Delta = \prod_{i=1}^{d_2} L_i^{-1} \Z$ is $\Delta' = \prod_{i=1}^{d_2} L_i \Z$ and the volume of the fundamental domain of $\Delta'$ is $\prod_{i=1}^{d_2} L_i$. Therefore, 
\begin{multline} \label{eq: g fourier}
	\mathscr{F}^{\pm}_{\hat{\Gamma}} \{g_t \} (x,y) 
	= \mathscr{F}^{\pm}_{\R^{d_1}} \{ e^{-\pi t^2 \abs{\cdot}^2 } \} (x) \cdot \mathscr{F}^{\pm}_{\prod_{i=1}^{d_2} L_i^{-1}\Z} \{ e^{-\pi t^2 \abs{\cdot}^2 } \} (y) \\
	= t^{-d_1} e^{-\pi \abs{x}^2 / t^2 } \left( \prod_{i=1}^{d_2} L_i \right)^{-1} \sum_{k \in \prod_{i=1}^{d_2} L_i^{-1}\Z} e^{-\pi t^2 \abs{k}^2} e^{2\pi i k \cdot y} 
	= t^{-d_1} e^{-\pi \abs{x}^2 / t^2 } t^{-d_2} \sum_{k \in \prod_{i=1}^{d_2} L_i\Z} e^{-\pi \abs{y+k}^2 / t^2} \\
	= t^{-(d_1 + d_2) } e^{-\pi \abs{x}^2 / t^2 } \sum_{k \in \prod_{i=1}^{d_2} L_i \Z} e^{-\pi \abs{y+k}^2 / t^2},
\end{multline}
where we used the Poisson summation formula for $h(y) = e^{-\pi t^2 \abs{y}^2 }$ on lattices to justify the second to last equality. By the fifth item, \eqref{eq: g conv} and \eqref{eq: g fourier}, we have 
\begin{multline}
	 \int_{\hat{\Gamma}} e^{-\pi t^2 (\abs{\xi}^2 + \abs{k}^2)} e^{\pm 2\pi i \xi \cdot x } \frac{e^{\pm 2\pi i k \cdot y}}{\sqrt{\prod_{i=1}^{d_2} L_i}} \mathscr{F}^{\pm}_{\Gamma} 
	\{f\} (\xi, k) \; d\xi dk = \int_{\Gamma} f(r,s) \mathscr{F}^{\pm}_{\hat{\Gamma}} \{ g_t \} (r - x, s - y) \; dr ds  \\  =  \int_{\Gamma} f(r,s) \mathscr{F}^{\pm}_{\hat{\Gamma}} \{ g_t \} (x - r, y - s) \; dr ds =f* \mathscr{F}_{\hat{\Gamma}}^{\pm} \{ g_t \} (x,y).
\end{multline}
We note that by the dominated convergence theorem, we have 
\begin{align}
	 \lim_{t \to 0^+}  \int_{\hat{\Gamma}} e^{-\pi t^2 (\abs{\xi}^2 + \abs{k}^2)} e^{\pm 2\pi i \xi \cdot x } \frac{e^{ \pm 2\pi i k \cdot y}}{\sqrt{\prod_{i=1}^{d_2} L_i}} \mathscr{F}^{\pm}_{\Gamma} 
	 \{f\} (\xi, k) \; d\xi dk = \mathscr{F}^{\mp}_{\hat{\Gamma}} \circ \mathscr{F}^{\pm}_{\Gamma} \{ f \} (x,y)
\end{align}
for all $(x,y) \in \Gamma$. On the other hand, we note that $\{ \mathscr{F}^{\pm}_{\hat{\Gamma}}\{g_t \} \}_{t > 0}$ form an approximate identity on $\Gamma$ since for all $t > 0$, $\mathscr{F}^{\pm}\{g_t \} (r,s) \ge 0$ on $\Gamma$ and by \eqref{eq: g fourier},
\begin{multline}\label{eq: approx iden}
	\int_{\Gamma} \mathscr{F}^{\pm}_{\hat{\Gamma}}\{g_t \} (r,s)  \; dr ds = t^{-(d_1 + d_2)}\int_{\Gamma} e^{-\pi \abs{r}^2 /t^2 } \sum_{k \in \prod_{i=1}^{d_2} L_i \Z} e^{-\pi \abs{s+k}^2 / t^2} \; dr ds \\ = t^{-(d_1 + d_2)} \int_{\R^{d_1}} e^{-\pi \abs{r}^2 / t^2} \; dr \int_{\R^{d_2}} e^{-\pi \abs{s}^2/t^2} \; ds = 1,
\end{multline}
and by \eqref{eq: approx iden} we also have $\lim_{t \to 0} \int_{\Gamma \setminus B(0,\delta)} \mathscr{F}_{\hat{\Gamma}}^{\pm}\{g_t \} (r,s)  \; dr ds \to 0$ for all $\delta > 0$. Therefore, for all $(x,y) \in \Gamma$ we have 
\begin{align}
	\mathscr{F}^{\mp}_{\hat{\Gamma}} \circ \mathscr{F}^{\pm}_{\Gamma} \{f\} (x,y) = \lim_{t\to 0^+} f*\mathscr{F}\{g_t\} (x,y) = f(x,y).
\end{align}
To prove the analogous inversion formula on $\mathscr{S}(\hat{\Gamma}; \C)$, for any given $t > 0$ and $(\xi, k) \in \hat{\Gamma}$ we consider the function 
\begin{align}
	 \varphi^{\pm}(x,y) =  \left(\prod_{i=1}^{d_2} L_i \right)^{1/2} e^{-\pi t^2 \abs{x}^2 } e^{\pm 2\pi i \xi \cdot x }   e^{\pm 2\pi i k \cdot y}.
\end{align}
We note that 
\begin{align}\label{eq: h approx iden}
	 \mathscr{F}_{\Gamma}^{\pm} \{ \varphi^{\pm} \} (\eta, m) = \mathscr{F}_{\Gamma}^{\pm} \{ h_t \} ( \eta - \xi, m - k) 
\end{align}
where $h_t(x,y) = \left(\prod_{i=1}^{d_2} L_i \right)^{1/2} e^{- \pi t^2 \abs{x}^2}$, and by a change of variables we have 
\begin{align}
	 \mathscr{F}_{\Gamma}^{\pm} \{ h_t \} (\xi, k) = \left(\prod_{i=1}^{d_2} L_i \right)^{1/2} \mathscr{F}_{\R^{d_1}} \{ e^{-\pi t^2 \abs{\cdot}^2 } \} (\xi) \cdot \mathscr{F}_{\prod_{i=1}^{d_2} L_i \T} \{ 1 \} (k) = t^{-d_1} e^{-\pi \abs{\xi}^2/t^2 } \delta_{0,k}, 
\end{align}
where $\delta_{0,k}$ is the Kronecker delta. Thus, by the fifth item, we have 
\begin{multline}
	\int_{\Gamma} \left(\prod_{i=1}^{d_2} L_i \right)^{1/2} e^{-\pi t^2 \abs{x}^2 } e^{\pm 2\pi i \xi \cdot x } e^{\pm 2\pi i k \cdot y} \mathscr{F}^{\pm}_{\hat{\Gamma}} 
   \{f\} (x, y) \; dx dy = \int_{\hat{\Gamma}} f(\eta, m ) \mathscr{F}^{\pm}_{\Gamma} \{ h_t \} (\eta - \xi, m - k) \; d\eta dm  \\ =  \int_{\hat{\Gamma}} f(\eta,m) \mathscr{F}^{\pm}_{\Gamma} \{ g_t \} (\xi - \eta, k- m) \; d\eta dm =f* \mathscr{F}_{\Gamma}^{\pm} \{ h_t \} (\xi,k).
\end{multline}
We note that by \eqref{eq: h approx iden}, $\{ \mathscr{F}_{\Gamma}^{\pm} \{ h_t \} \}_{t > 0}$ form an approximate identity on $\hat{\Gamma}$, therefore by repeating the same argument on $\mathscr{S}(\Gamma; \C)$ we find that $f (\xi, k) = \mathscr{F}^{\mp}_{\Gamma} \circ \mathscr{F}^{\pm}_{\hat{\Gamma}} \{f\} (\xi, k)$ for all $(\xi, k) \in \hat{\Gamma}$. This proves the sixth item.

The last item follows immediately from the fourth item and the seventh item. 
\end{proof}

We then extend the definition of the Fourier transform to the class of tempered distributions. 

\begin{defn}
Suppose $T_1 \in \mathscr{S}'(\Gamma ; \C)$ and $T_2 \in \mathscr{S}'(\hat{\Gamma} ; \C)$, where $\Sigma$ is defined via \eqref{eq: Gamma f}. 
\begin{enumerate}
\item We define the Fourier and inverse Fourier transform of $T_1$ to be $\mathscr{F}_{\Gamma}^{\pm} \{ T_1 \} \in \mathscr{S}'(\hat{\Gamma} ; \C)$ given by 
\begin{align}
	 \langle \mathscr{F}_{\Gamma}^{\pm} \{ T_1 \}, \varphi \rangle = \langle T_1, \mathscr{F}_{\Gamma}^{\pm} \{ \varphi \}\rangle \; \text{for all} \; \varphi \in \mathscr{S}(\hat{\Gamma} ; \C)
\end{align}
\item We define the Fourier and inverse Fourier transform of $T_2$ to be $\mathscr{F}_{\hat{\Gamma}}^{\pm} \{ T_2 \} \in \mathscr{S}'(\Gamma ; \C)$ given by 
\begin{align}
	 \langle \mathscr{F}_{\hat{\Gamma}}^{\pm} \{ T_2 \}, \varphi \rangle = \langle T_2, \mathscr{F}_{\hat{\Gamma}}^{\pm} \{ \varphi \} \rangle \; \text{for all} \; \varphi \in \mathscr{S}(\Gamma ; \C)
\end{align}
\end{enumerate}
\end{defn}
We note that by Proposition~\ref{prop:schwartz}, the maps $\mathscr{F}_{\Gamma}^{\pm}: \mathscr{S}'(\Gamma ; \C) \to \mathscr{S}'(\hat{\Gamma}; \C)$ and $\mathscr{F}_{\hat{\Gamma}}^{\pm}:  \mathscr{S}'(\hat{\Gamma}; \C) \to \mathscr{S}'(\Gamma; \C)$ are also isomorphisms.

Throughout the paper we will follow standard notation by denoting the Fourier and inverse Fourier transforms for a Schwartz function $f \in \mathscr{S}(\Gamma; \C)$ or a tempered distribution $f \in \mathscr{S}'(\Gamma; \C)$ by $\hat{f} = \mathscr{F}_\Gamma^{+} \{ f \}, \; \check{f} = \mathscr{F}_\Gamma^{-} \{ f \}$. Sometimes we will also write  $\mathscr{F}_{\Gamma}[f] = \mathscr{F}_\Gamma^{+} \{ f \}, \; \mathscr{F}_{\Gamma}^{-1} [f]  = \mathscr{F}_\Gamma^{-} \{ f \}$.

\subsection{Spaces defined via Fourier multipliers} \label{appendix: spaces}
We note that since $\hat{\Gamma}$ and $\hat{\Gamma}_\ast$ are always subgroups of $\R^d$, we can extend the reordering map $\mathcal{P}_{g,\hat{\Gamma}}$ trivially to $\R^d$. In the next lemma, we show that the restriction of the map $\sigma_{g,\Gamma}$ on a space $\mathscr{X}_\omega(\Gamma ; \F)$, where $\sigma_{g,\Gamma}$ is defined in the third item of Definition~\ref{defn: reordering notation}, is an isometric isomorphism  between $\mathscr{X}_\omega(\Gamma ; \F)$ and $\mathscr{X}_\omega(g \Gamma ; \F)$, where $\mathscr{X}_\omega(\Gamma ; \F)$ is a function space defined in terms of a Fourier multiplier $\omega : \R^d \to [0,\infty)$ that is invariant under $\mathcal{P}_{g,\hat{\Gamma}}$. 

\begin{lem}\label{lem: gen X^s permutation}
Let $\Gamma$ be defined as in \eqref{eq: Gamma f}. Let $g \in S_d$, $g\Gamma$ be defined as in the first item of Definition~\ref{defn: reordering notation}, and $\mathcal{P}_{g,\Gamma}, \mathcal{P}_{g, \hat{\Gamma}}$ be the reordering maps defined as in \eqref{eq: reordering}. Let $\omega: \R^d \to [0,\infty)$ be a Fourier multiplier that is positive $\mathcal{L}^d$-almost everywhere such that $\omega(\cdot) = \omega(\mathcal{P}_{g,\hat{\Gamma}} \cdot)$. We define the function space $\mathscr{X}_\omega(\Gamma ; \F)$ via 
\begin{align}
    \mathscr{X}_\omega(\Gamma ; \F)=
	\begin{cases}
		\tcb{f\in\mathscr{S}'(\Gamma;\C) \mid \hat{f}\in L^1_{\loc}(\hat{\Gamma};\C), \;f= \overline{f},\;\tnorm{\sqrt{\omega}\hat{f}}_{L^2}<\infty}, & \F = \R \\
		\tcb{f\in\mathscr{S}'(\Gamma;\C) \mid \hat{f}\in L^1_{\loc}(\hat{\Gamma};\C),\;\tnorm{\sqrt{\omega}\hat{f}}_{L^2}<\infty}, & \F = \C,
	\end{cases}
\end{align}
Then the map $\sigma_{g,\Gamma}: \mathscr{X}_\omega(\Gamma ; \F) \to  \mathscr{X}_\omega(g\Gamma; \F)$ defined via $\sigma_{g,\Gamma} f = f \circ \mathcal{P}_{g,\Gamma}^{-1}$ is an isometric isomorphism, with its inverse given by the map $\sigma_{g,\Gamma}^{-1}:  \mathscr{X}_\omega(g \Gamma ; \F) \to  \mathscr{X}_\omega(\Gamma ; \F)$ defined via $\sigma_{g,\hat{\Gamma}}^{-1} f = f \circ \mathcal{P}_{g,\Gamma}$. 
\end{lem}
\begin{proof}
We first note that by Lemma~\ref{lem: schwartz iso}, the map $\sigma_{g,\Gamma}: \mathscr{S}'(\Gamma; \F) \to \mathscr{S}'(g \Gamma; \F)$, with its inverse given by $\sigma_{g,\Gamma}^{-1}: \mathscr{S}'(g \Gamma ; \F) \to \mathscr{S}'(\Gamma ; \F)$. Then for any $f \in \mathscr{X}_\omega(\Gamma; \F)$, we have $\sigma_{g,\Gamma} f \in \mathscr{S}'(g \Gamma; \C), \mathscr{F}_{g \Gamma}^+ [\sigma_{g, \Gamma} f]\in L^1_{\loc}(\hat{\Gamma}_\ast;\C)$, and $\sigma_{g,\Gamma} f = \overline{\sigma_{g,\Gamma} f} $. Furthermore, by \eqref{eq: check commute} we have $\mathscr{F}_{g \Gamma}^+ [\sigma_{g,\Gamma} f] = \sigma_{g,\hat{\Gamma}} \mathscr{F}_\Gamma^+ [f]$. Then 
\begin{multline}\label{eq: switching calculation}
	 \norm{ \sigma_{g,\Gamma} f}_{\mathscr{X}_\omega(g \Gamma)}^2 = \int_{g \hat{\Gamma}}  \omega(\xi) \abs{\mathscr{F}^{+}_\Gamma[f](\mathcal{P}_{g,\hat{\Gamma}}^{-1} \xi)}^2 \; d\xi 
	 = \int_{\hat{\Gamma}} \omega(\mathcal{P}_{g,\hat{\Gamma}} \xi) \abs{\mathscr{F}^{+}_{\Gamma}[f]( \xi)}^2 \; d\xi \\ = \int_{\hat{\Gamma}} \omega(\xi) \abs{\mathscr{F}^{+}_{\Gamma}[f]( \xi)}^2 \; d\xi= \norm{f}_{\mathscr{X}_\omega(\Gamma)}^2.
\end{multline}
This shows that $\sigma_{g,\Gamma} f \in \mathscr{X}_\omega(g \Gamma; \F)$ with $\norm{\sigma_{g,\Gamma} f}_{\mathscr{X}_\omega(g \Gamma)} = \norm{f}_{\mathscr{X}_\omega(\Gamma)}$. Following the exact same calculations as above we also find that for any $f \in \mathscr{X}_\omega(g \Gamma; \F)$, we have $\sigma_{g,\Gamma}^{-1} f=  f \circ \mathcal{P}_{g, \Gamma} \in \mathscr{X}_\omega(\Gamma ; \F)$ with $\Vert\sigma_{g,\Gamma}^{-1}f\Vert_{\mathscr{X}_\omega(\Gamma)} = \norm{f}_{\mathscr{X}_\omega(g \Gamma)}$. This shows that $\sigma_{g,\Gamma}: \mathscr{X}_\omega(\Gamma ; \F) \to  \mathscr{X}_\omega(g\Gamma; \F)$ is an isometric isomorphism with $\sigma_{g,\Gamma}^{-1}$ as its inverse. 
\end{proof}

\subsection*{Acknowledgement} The authors would like to thank Beno\^{i}t Pausader for pointing out the proof for Theorem~\ref{thm:boundedness of I}.

\nocite{*}
\bibliographystyle{abbrv}
\bibliography{bib.bib}

\end{document}